\documentclass[reqno,usenames]{amsart}

\usepackage{tikz}
\usepackage{amsfonts}
\usepackage{bbm} 
\usepackage{bbold} 
\usepackage{amssymb}
\usepackage{MnSymbol} 
\usepackage{amsbsy}
\usepackage{verbatim}
\usepackage{amscd}
\usepackage{graphicx} 
\usepackage[geometry]{ifsym} 
\usepackage{xypic}
\usepackage{graphics}
\usepackage{rotating}
\usepackage[colorlinks=true,citecolor=red,urlcolor=red]{hyperref}

\makeatletter{}
\newtheorem{theorem}{Theorem}[section]
\newtheorem{prop}[theorem]{Proposition}
\newtheorem{lemma}[theorem]{Lemma}
\newtheorem{con}[theorem]{Conjecture}

\newtheorem{obv}[theorem]{Observation}
\theoremstyle{definition}
\newtheorem{definition}[theorem]{Definition}

\theoremstyle{remark}
\newtheorem{remark}[theorem]{Remark}
\newtheorem{example}[theorem]{Example}

\theoremstyle{claim}

\def\cA{\mathcal A}\def\cB{\mathcal B}\def\cC{\mathcal C} \def\cD{\mathcal D}\def\cE{\mathcal E}\def\cF{\mathcal F} \def\cG{\mathcal G}\def\cH{\mathcal H}\def\cI{\mathcal I} \def\cJ{\mathcal J}\def\cK{\mathcal K}\def\cL{\mathcal L}
\def\cM{\mathcal M}\def\cN{\mathcal N}\def\cO{\mathcal O} \def\cP{\mathcal P}\def\cQ{\mathcal Q}\def\cR{\mathcal R} \def\cS{\mathcal S}\def\cT{\mathcal T}\def\cU{\mathcal U} \def\cV{\mathcal V}\def\cW{\mathcal W}\def\cX{\mathcal X}
\def\cY{\mathcal Y}\def\cZ{\mathcal Z}

\def\AA{\mathbb A}\def\BB{\mathbb B}\def\CC{\mathbb C}\def\DD{\mathbb D}
\def\EE{\mathbb E}\def\FF{\mathbb F}\def\GG{\mathbb G}\def\HH{\mathbb H}
\def\II{\mathbb I}\def\JJ{\mathbb J}\def\KK{\mathbb K}\def\LL{\mathbb L}
\def\MM{\mathbb M}\def\NN{\mathbb N}\def\OO{\mathbb O}\def\PP{\mathbb P}
\def\QQ{\mathbb Q}\def\RR{\mathbb R}\def\SS{\mathbb S}\def\TT{\mathbb T}
\def\UU{\mathbb U}\def\VV{\mathbb V}\def\WW{\mathbb W}\def\XX{\mathbb X}
\def\YY{\mathbb Y}\def\ZZ{\mathbb Z}

\def\a{\alpha}
\def\b{\beta}
\def\d{\delta}
\def\e{\eta}
\def\eps{\epsilon}
\def\l{\lambda}
\def\r{\rho}
\def\s{\sigma}
\def\t{\theta}
\def\vp{\varphi}
\def\z{\zeta}

\DeclareSymbolFont{bbgreek}{U}{bbold}{m}{n}

\DeclareMathSymbol{\bbalpha}{\mathbb}{bbgreek}{'13}
\DeclareMathSymbol{\bbbeta}{\mathbb}{bbgreek}{'14}
\DeclareMathSymbol{\bbchi}{\mathbb}{bbgreek}{'37}
\DeclareMathSymbol{\bbdelta}{\mathbb}{bbgreek}{'16}
\DeclareMathSymbol{\bbepsilon}{\mathbb}{bbgreek}{'17}
\DeclareMathSymbol{\bbphi}{\mathbb}{bbgreek}{'36}
\DeclareMathSymbol{\bbgamma}{\mathbb}{bbgreek}{'15}
\DeclareMathSymbol{\bbeta}{\mathbb}{bbgreek}{'21}
\DeclareMathSymbol{\bbiota}{\mathbb}{bbgreek}{'23}
\DeclareMathSymbol{\bbkappa}{\mathbb}{bbgreek}{'24}
\DeclareMathSymbol{\bblambda}{\mathbb}{bbgreek}{'25}
\DeclareMathSymbol{\bbmu}{\mathbb}{bbgreek}{'26}
\DeclareMathSymbol{\bbnu}{\mathbb}{bbgreek}{'27}
\DeclareMathSymbol{\bbpi}{\mathbb}{bbgreek}{'31}
\DeclareMathSymbol{\bbtheta}{\mathbb}{bbgreek}{'22}
\DeclareMathSymbol{\bbrho}{\mathbb}{bbgreek}{'32}
\DeclareMathSymbol{\bbsigma}{\mathbb}{bbgreek}{'33}
\DeclareMathSymbol{\bbvarsigma}{\mathbb}{bbgreek}{'20}
\DeclareMathSymbol{\bbtau}{\mathbb}{bbgreek}{'34}
\DeclareMathSymbol{\bbupsilon}{\mathbb}{bbgreek}{'35}
\DeclareMathSymbol{\bbomega}{\mathbb}{bbgreek}{'177}
\DeclareMathSymbol{\bbxi}{\mathbb}{bbgreek}{'30}
\DeclareMathSymbol{\bbpsi}{\mathbb}{bbgreek}{'40}

\DeclareMathSymbol{\bbzero}{\mathbb}{bbgreek}{'60}
\DeclareMathSymbol{\bbunit}{\mathbb}{bbgreek}{'61}

\DeclareMathSymbol{\bbAlpha}{\mathbb}{bbgreek}{'101}
\DeclareMathSymbol{\bbBeta}{\mathbb}{bbgreek}{'102}
\DeclareMathSymbol{\bbChi}{\mathbb}{bbgreek}{'130}
\DeclareMathSymbol{\bbDelta}{\mathbb}{bbgreek}{'1}
\DeclareMathSymbol{\bbEpsilon}{\mathbb}{bbgreek}{'105}
\DeclareMathSymbol{\bbPhi}{\mathbb}{bbgreek}{'10}
\DeclareMathSymbol{\bbGamma}{\mathbb}{bbgreek}{'0}
\DeclareMathSymbol{\bbEta}{\mathbb}{bbgreek}{'110}
\DeclareMathSymbol{\bbIota}{\mathbb}{bbgreek}{'111}
\DeclareMathSymbol{\bbKappa}{\mathbb}{bbgreek}{'113}
\DeclareMathSymbol{\bbLambda}{\mathbb}{bbgreek}{'3}
\DeclareMathSymbol{\bbMu}{\mathbb}{bbgreek}{'115}
\DeclareMathSymbol{\bbNu}{\mathbb}{bbgreek}{'116}
\DeclareMathSymbol{\bbO}{\mathbb}{bbgreek}{'117}
\DeclareMathSymbol{\bbPi}{\mathbb}{bbgreek}{'5}\DeclareMathSymbol{\bbTheta}{\mathbb}{bbgreek}{'2}
\DeclareMathSymbol{\bbRho}{\mathbb}{bbgreek}{'120}
\DeclareMathSymbol{\bbSigma}{\mathbb}{bbgreek}{'6}
\DeclareMathSymbol{\bbTau}{\mathbb}{bbgreek}{'124}
\DeclareMathSymbol{\bbUpsilon}{\mathbb}{bbgreek}{'131}
\DeclareMathSymbol{\bbOmega}{\mathbb}{bbgreek}{'12}
\DeclareMathSymbol{\bbXi}{\mathbb}{bbgreek}{'4}
\DeclareMathSymbol{\bbPsi}{\mathbb}{bbgreek}{'11}
\DeclareMathSymbol{\bbZeta}{\mathbb}{bbgreek}{'132}

\DeclareMathOperator{\Aut}{Aut}
\DeclareMathOperator{\coev}{coev}
\DeclareMathOperator{\End}{End}
\DeclareMathOperator{\ev}{ev}
\DeclareMathOperator{\Gal}{Gal}
\DeclareMathOperator{\GL}{GL}
\DeclareMathOperator{\Hom}{Hom}
\DeclareMathOperator{\id}{id}
\DeclareMathOperator{\Id}{Id}
\DeclareMathOperator{\Iso}{Iso}
\DeclareMathOperator{\Map}{Map}
\DeclareMathOperator{\Mat}{Mat}
\DeclareMathOperator{\Mor}{Mor}
\DeclareMathOperator{\qdim}{qdim}
\DeclareMathOperator{\Rep}{Rep}
\DeclareMathOperator{\Span}{Span}
\DeclareMathOperator{\Vect}{Vect}

\newcommand{\nn}{\nonumber}
\newcommand{\nid}{\noindent}
\newcommand{\ra}{\rightarrow}
\newcommand{\la}{\leftarrow}
\newcommand{\lra}{\longrightarrow}
\newcommand{\lla}{\longleftarrow}
\newcommand{\xra}{\xrightarrow}
\newcommand{\xla}{\xleftarrow}
\newcommand{\into}{\hookrightarrow}
\newcommand{\weq}{\xrightarrow{\sim}}
\newcommand{\cofib}{\rightarrowtail}
\newcommand{\fib}{\twoheadrightarrow}
\newcommand{\g}{\ensuremath{\mathfrak{g}} }
\newcommand{\VectC}{\ensuremath{\Vect_\CC^{\mbox{\scriptsize fin}}}}
\def\defeq{\ensuremath{\buildrel\triangle\over =}} 
\def\llarrow{   \hspace{.05cm}\mbox{\,\put(0,-2){$\leftarrow$}\put(0,2){$\leftarrow$}\hspace{.45cm}}}
\def\rrarrow{   \hspace{.05cm}\mbox{\,\put(0,-2){$\rightarrow$}\put(0,2){$\rightarrow$}\hspace{.45cm}}}
\def\lllarrow{  \hspace{.05cm}\mbox{\,\put(0,-3){$\leftarrow$}\put(0,1){$\leftarrow$}\put(0,5){$\leftarrow$}\hspace{.45cm}}}
\def\rrrarrow{  \hspace{.05cm}\mbox{\,\put(0,-3){$\rightarrow$}\put(0,1){$\rightarrow$}\put(0,5){$\rightarrow$}\hspace{.45cm}}}

\def\bfA{\mathbf A}
\def\uk{\underline{k}}
\def\del{\partial}
\def\qbar{\bar{\mathbb{Q}}}

\newcommand{\cech}[1]{\ensuremath{#1^{*}}}
\newcommand{\rcech}[1]{\ensuremath{^{\scriptscriptstyle \vee}\!\!#1}}

\newcommand{\ov}[1]{\ensuremath{\overline{#1}}}

\newcommand{\ignore}[1]{}
\newcommand{\mb}[1]{\ensuremath{\mbox{\scriptsize #1}}} \newcommand{\cat}[1]{\mathcal{\uppercase{#1}}}
\newcommand{\fld}[1]{\mathbf{#1}}
\newcommand{\catvec}[1]{\cat{VEC}(\mathbf{#1})}
\newcommand{\catvecfin}[1]{\cat{VEC}_{fin}(\mathbf{#1})}
\newcommand{\leftexp}[2]{{\vphantom{#2}}^{#1}{#2}}
\newcommand{\standardfigsized}[3]{
	\begin{figure}
		\centerline{\includegraphics[#3]{figures/#1}}
		\ignore{    \caption{#2 Reference \verb{#1}.}\label{#1}     }
		\caption{#2}\label{#1}
	\end{figure}
}
\newcommand{\substandardfigsized}[3]{
  \begin{figure}
    \caption{INSERT PICTURE. #2 Reference $#1$.}\label{#1}
  \end{figure}
}
\newcommand{\standardfig}[2]{
  \standardfigsized{#1}{#2}{}
}\newcommand{\substandardfig}[2]{
  \substandardfigsized{#1}{#2}{}
}

\newcommand{\standardgraph}[1]{\includegraphics{figures/#1}}
\newcommand{\trivobject}{{\mathbf 1}}
\newcommand{\nnr}[2]{\langle #1,#2 \rangle}
\newcommand{\snnr}[2]{(#1,#2)}

\newcommand{\orit}[1]{{\color{blue}#1}}
\newcommand{\tobe}[1]{{\color{red}#1}}
\newcommand{\com}[1]{{\color{ComColor}#1}}

\renewcommand{\labelitemi}{\FilledSmallTriangleRight}

\ignore{\newcommand{\fixme}[1]{}}

\newcommand{\bitm}{\begin{itemize}}
\newcommand{\eitm}{\end{itemize}}

\newenvironment{mylist}{\begin{list}{(\roman{listi})}{\usecounter{listi}\setlength{\topsep}{0ex}\setlength{\itemsep}{0ex}\setlength{\parsep}{0ex}\setlength{\parskip}{0ex}\setlength{\leftmargin}{1ex}}}{\end{list}}

\def\sec{\S section}

\begin{document}

\title[AMC]
{On Arithmetic Modular Categories}

\author{Orit Davidovich$^{1}$, Tobias Hagge$^{2}$, and Zhenghan Wang$^{3,4}$}
\address{$^1$Department of Mathematics\\Northwestern University\\Evanston, IL 60208}
\address{$^2$Department of Mathematical Sciences\\University of Texas at Dallas\\Richardson, TX 75080}
\address{$^3$Microsoft Station Q\\ University of California\\ Santa Barbara, CA 93106}
\address{$^4$Department of Mathematics\\University of California\\Santa Barbara, CA 93106}

\maketitle

\setlength{\baselineskip}{3.5ex}
\setlength{\intextsep}{1ex} 
\newlength{\pips}
\newcounter{listi}

\section{Introduction} 

Modular categories are elaborate algebraic structures which may be regarded as categorifications of finite abelian groups \cite{Wang12}.  They serve as the algebraic input of $(2+1)$-topological quantum field theories (TQFTs) and $(1+1)$-conformal field theories (CFTs) \cite{Tu}\cite{MS}.  Unitary modular categories have recently found applications in condensed matter physics and quantum computing \cite{Wangbook}.  

A program to classify modular categories revealed a deep connection between modular category theory and number theory \cite{RSW}\cite{BNRW}.  One manifestation of such a connection is the congruence property of the kernel subgroup of the modular representation of the modular group $\textrm{SL}(2,\mathbb{Z})$ of a modular category \cite{NS}.  Given a modular category $\cC$ over $\mathbb{C}$, it is natural to ask whether  there exists an algebraic number field $k$ and a modular category $\cC'$ defined over $k$ so that $\cC\cong \cC' \otimes_k \CC$.  Such a number field $k$ is called a {\it defining number field} for $\cC$.

 Etingof, Nikshych and Ostrik observed that the existence of a defining algebraic number field for a fusion category over $\CC$ follows from their proof of Ocneanu rigidity, and further asked whether every fusion category over $\mathbb{C}$ has a defining cyclotomic field (see remark and question after Theorem $2.30$ in \cite{ENO}.)  Morrison and Snyder proved that some fusion categories from the (extended) Haagerup subfactors cannot have cyclotomic defining number fields \cite{MoS}.  It is further known that the modular data of modular categories and CFTs can be presented within cyclotomic number fields \cite{Gan}\cite{RSW}.  
 
 Our paper is an elementary treatment of arithmetic definitions of modular categories, and serves as the first step towards an arithmetic theory of modular categories and their applications to quantum physics and quantum computing.  Our technical tool is an explicit definition of a modular category as a set of numbers satisfying a collection of polynomials up to an equivalence.  Therefore, modular categories with the same fusion rules are equivalence classes of solutions to the same set of algebraic equations, i.e., equivalent classes of points on some algebraic variety. The numerical data of a modular category $\cC$ is conveniently organized into matrices: $F$-matrices, $R$-matrices, $S$-matrix, and $T$-matrix, and a pivotal coefficient vector $\epsilon$ with coordinates in $\{\pm 1\}$.  The existence of such a numerical definition of a modular category follows from the categorical definition.  Our contribution is to pin down the details, which is not straightforward.  This numerical definition of a modular category is an analogy of the definition of a connection in differential geometry and gauge theory by the Christoffel symbols.  While the fusion rules $N_{ij}^k$, the $S$-matrix, the $T$-matrix, and the eigenvalues of the $R$-matrices are intrinsic, the rest of data are in general gauge dependent in the sense that they depend on the choices of bases of the morphism spaces.  But in applications to physics, the matrix entries of $F$-matrices are called {\it $6j$ symbols}, and directly enter the definitions of Hamiltonians and appear as probability amplitudes of quantum processes.  As such, they should be computable numbers which include all algebraic numbers.  We prove that under a suitable gauge choice, every modular category over $\mathbb{C}$ has a definition by $F$-matrices, $R$-matrices, $S$-matrix, and $T$-matrix whose entries are all within some algebraic number field $K_\cC$.  Our existence proof of the defining algebraic number fields for a modular category does not provide any information on wether or not there is a canonical choice of such a number field.  While not true for fusion categories, we believe that every modular category over $\mathbb{C}$ has a cyclotomic defining number field. If so, then a canonical choice of a defining number field for a given modular category $\cC$ would be a cyclotomic field of minimal degree.

Unitary modular categories are algebraic models of anyons \cite{Ki}\cite{Wangbook}.  Therefore, classification of unitary modular categories has direct application in the identification of topological phases of matter \cite{GaloisConjugate}\cite{JWB}.  For a classification, it is very useful to generate new modular categories from old ones.  Our arithmetic definition provides such a method: Galois twist.  Galois twists also provide an interesting way of organizing modular categories into orbits under Galois actions w.r.t. a suitably chosen number field. All modular categories within a single orbit have the same fusion rules.  It would have been desirable if all modular categories with the same fusion rules were in the same Galois orbit with respect to some suitably chosen number field.  Unfortunately, this is not the case as we shall see below. 

 The contents of the paper are as follows.  In Section $2$, we set up our conventions and observe that for fusion categories, the two defining identities for rigidity are not independent.  In Section $3$, we introduce the modular system for a modular category.  The modular system is an algebraic variety and defining a modular category is equivalent to solving those equations up to equivalence.  In this sense, the theory of modular category is completely elementary.  In Section $4$, we show that every modular category has an algebraic defining number field using our modular systems.  Then, we define Galois twists of modular categories and examine several examples.  We also make several conjectures about orbits of Galois twists and the intrinsic data of modular categories.  Finally, we speculate on potential applications of defining number fields to the search for exotic modular categories.    

\section{Preliminaries} 

\subsection{Conventions} 

\label{sec:conventions}

Throughout the paper $k$ will denote a field of characteristic $0$. 

By a {\em fusion category} defined over $k$ we mean a $k$-linear category which is also abelian, semisimple and rigid monoidal, having finitely many simple objects (up to isomorphism) including the neutral object $\mathbf{1}$, such that the endomorphism space of every simple object is identified with $k$.  When $k$ is algebraically closed the last requirement is redundant\footnote{The definition of a fusion category in the algebraically closed case typically assumes only $\cC(\mathbf{1},\mathbf{1})=k$ (see \cite{ENO}). In the non-algebraically closed case the definition of a fusion category requires extra assumptions.  In \cite{BK} a simple object is defined to be a non-zero object such that every injection into it is either 0 or an isomorphism.  In this case we add the requirement $\cC(a,a)=k$ for every simple object $a \in \cC$ AND the neutral object is simple.}.  A {\em modular category} defined over $k$ is a ribbon fusion category over $k$ with an invertible $S$-matrix.

We typically denote simple objects in semi-simple abelian categories by the letters $a,b,c,\ldots$ whereas generic objects are denoted by the letters $x,y,\ldots$.  When working with monoidal categories we sometimes write `$xy$' to mean `$x \otimes y$'. We use a numerical in bold `$\mathbf{1}$' to designate the object in a monoidal category neutral with respect to tensor product.  We use a numerical `$1_x$' to designate the identity morphism $1_x : x \ra x$.  By a $(m,n)$-Hom-space we mean a morphism space of the form $\cC(a_1 \otimes\cdots\otimes a_m, b_1 \otimes\cdots\otimes b_n)$ where $a_i,b_j$ are simple in $\cC$.  Functors between (braided) monoidal categories are required to be (braided) monoidal unless stated otherwise.  In particular, equivalences between fusion / modular categories are required to respect structure, unless stated otherwise.  

In many proofs we employ graphical calculus.  Our conventions comply with \cite{BK}, in particular, diagrams are read bottom to top.  However, we do not identify a line labeled by an object $x$ with a line labeled by an object $x^{**}$, a distinction which allows us to remove arrows from diagrams.

\subsection{Linear Abelian Categories} 

\label{ss:lac}

In all $k$-linear semi-simple abelian categories considered throughout, we assume that there are finitely many isomorphism classes of simple objects, that morphism spaces are finite dimensional, and that functors between $k$-linear abelian categories are $k$-linear and additive.  

Let $\cC$ be a $k$-linear semi-simple abelian category with {\em absolutely simple} objects (i.e. the endomorphism space of every simple object is identified with $k$). Let $L$ be a set of representatives for all isomorphism classes of simple objects in $\cC$.  Since every object $x \in \cC$ admits a decomposition into simple objects $c \in L$ there are
\begin{eqnarray}
&& \eta_x^{c,i} \in \cC(x,c) \quad , \quad i = 1 \ldots \dim_k \cC(x,c) \label{eq:project} \\
&& \eta^x_{c,j} \in \cC(c,x) \quad , \quad j = 1 \ldots \dim_k \cC(c,x) \label{eq:inject}
\end{eqnarray}
such that
\[ \eta_x^{c,i} \circ \eta^x_{c,j} = \d_{i,j} 1_c \]
\[ \sum_{c \in L} \sum_{i} \eta^x_{c,i} \circ \eta_x^{c,i} = 1_x \]
In particular, there is a perfect pairing defined by the composition
\[ \cC(c,x) \times \cC(x,c) \lra \cC(c,c) \equiv k \]
with respect to which the bases $\{ \eta_{c,j}^x \}$ and $\{ \eta_x^{c,i} \}$ are dual. 

This allows for a notion of 'linear extension' natural transformations of functors between linear abelian categories.
\begin{lemma}
\label{lem:nonsense}
Let $\cC$ and $\cD$ be $k$-linear semi-simple abelian categories.  Let $F,G : \cC^{\times n} \ra \cD$ be functors additive in each component.  Assume for any $n$-tuple of simple objects $(a_1,\ldots,a_n) \in \cC^{\times n}$ there is a morphism
\[ \eta(a_1,\ldots,a_n) : F(a_1,\ldots,a_n) \lra  G(a_1,\ldots,a_n) \]
such that for any $n$-tuple of morphisms $(f_1,\ldots,f_n) :  (a_1,\ldots,a_n) \ra  (a'_1,\ldots,a'_n)$ the following diagram commutes
\[
\xymatrix{
F(a_1,\ldots,a_n) \ar[rr]^{ F(f_1,\ldots,f_n)} \ar[d]_{ \eta(a_1,\ldots,a_n)} && 
F(a'_1,\ldots,a'_n) \ar[d]^{ \eta(a'_1,\ldots,a'_n)} \\
G(a_1,\ldots,a_n) \ar[rr]_{ G(f_1,\ldots,f_n)} && G(a'_1,\ldots,a'_n)
}
\]
Then $\eta$ extends uniquely to an additive natural transformation $\eta: F \ra G$. In other words, $\eta$ is $n$-functorial.
\end{lemma}
The above lemma is proved by abstract nonsense.    

\subsection{Duals} 

\label{ss:dual}

A {\em monoidal} category $\cC$ is a category equipped with a bi-functor $\otimes : \cC \times \cC \ra \cC$, a distinguished object $\mathbf{1} \in \cC$, and natural isomorphisms
\begin{eqnarray}
&& \a_{x,y,z} : x \otimes (y \otimes z) \lra (x \otimes y) \otimes z \\
&& \l_x : \mathbf{1} \otimes x \lra x \\
&& \r_x: x \otimes \mathbf{1} \lra x
\end{eqnarray}
such that the pentagonal diagram
\begin{equation}
\label{eqn:pent_diag}
\xymatrix{
x \otimes (y \otimes (z \otimes w)) \ar[r] ^{\a} \ar[d]_{1_x\otimes\a} & (x \otimes y) \otimes (z \otimes w) \ar[r]^{\a} & ((x \otimes y) \otimes z) \otimes w  \\
x \otimes ((y \otimes z) \otimes w) \ar[rr] _{\a} && (x \otimes (y \otimes z)) \otimes w \ar[u]_{\a \otimes 1_w}}\end{equation}
and the triangular diagram
\begin{equation}
\label{eqn:tri_diag}
\xymatrix{
x \otimes (\mathbf{1} \otimes y) \ar[rr] ^{\a} \ar[d] _{1_x \otimes \l} && (x \otimes \mathbf{1}) \otimes y \ar[d] ^{\r \otimes 1_y} \\
x \otimes y \ar@{=}[rr] && x \otimes y
}\end{equation}
both commute.  Two other triangular diagrams (obtained by placing $\mathbf{1}$ to the left of $x$ or to the right of $y$) commute as a consequence of \eqref{eqn:pent_diag} and \eqref{eqn:tri_diag} combined (\cite{Ka}, Lemma XI.2.2).

A monoidal category is said to have \emph{right duals} if for every object $x \in \cC$ there exists an object $x^* \in \cC$ and co-evaluation and evaluation morphisms
\[ c_x : \mathbf{1} \lra x^* \otimes x \quad , \quad e_x : x \otimes x^* \lra \mathbf{1} \]
such that the composite
\begin{equation}
\label{eq:rda1}
\xymatrix{
x \ar[r]^{\r_x^{-1}} & x \otimes \mathbf{1} \ar[r]^{1_x \otimes c_x} & x \otimes (x^* \otimes x) \ar[r]^{\a_{x,x^*,x}} & (x \otimes x^*) \otimes x \ar[r]^{e_x \otimes 1_x} & \mathbf{1} \otimes x \ar[r]^{\l_x} & x
}
\end{equation}
equals the identity on $x$, and the composite
\begin{equation}
\label{eq:rda2}
\xymatrix{
x^* \ar[r]^{\l^{-1}} & \mathbf{1} \otimes x^* \ar[r]^{c_x \otimes 1} & (x^* \otimes x) \otimes x^* \ar[r]^{\a_{-1}} & x^* \otimes (x \otimes x^*) \ar[r]^{1 \otimes e_x} & x^* \otimes \mathbf{1} \ar[r]^{\r} & x^*
}
\end{equation}
equals the identity as well.  We refer to the first composite (\ref{eq:rda1}) as the \emph{first right duality axiom} (RDA-I), and to the second composite (\ref{eq:rda2}) as the \emph{second right duality axiom} (RDA-II).  Our conventions conform with the notion of right adjoint functors, i.e. the functor $x^* \otimes (-) : \cC \ra \cC$ is right adjoint to the functor $x \otimes (-) : \cC \ra \cC$.

As a consequence of \cite{O}, a monoidal category with right duals has left duals as well if and only if for every object $x$ there exists an object $y$ such that $x \cong y^*$.  This is indeed the case for $k$-linear semi-simple abelian categories satisfying the above finiteness assumptions (see \S\ref{ss:lac}), hence rigidity follows from the existence of right duals.  

When a monoidal category is $k$-linear semi-simple abelian, one duality axiom can be derived from the other.

\begin{lemma}
\label{lem:RA}
In a $k$-linear, semi-simple, abelian monoidal category $\cC$, where every simple object $a \in \cC$ is absolutely simple, i.e. $\cC(a,a) = k$, the first right duality axiom (RDA-I) implies the second right duality axiom (RDA-II).
\end{lemma}

\begin{proof}
Let $\cC$ be a $k$-linear, semi-simple, abelian monoidal category.  Let $x$ be an object for which there exists a right dual $x^*$ and evaluation and co-evaluation morphisms $e_x$ and $c_x$ such that the first duality axiom (RDA-I) holds.  Consider the map $\varphi_x : \cC(x^*,x^*) \ra \cC(x,x)$ defined by mapping $f \in \cC(x^*,x^*)$ to the composite
\begin{equation*}
\xymatrix{
x \ar[d]_{\r^{-1}} \ar[rrrr]^{\vp_x(f)} &&&& x \\ 
x \otimes \mathbf{1} \ar[r]^{1\otimes c\ \ \ \ \ \ } & x \otimes (x^* \otimes x) \ar[r]^{\a} &
 (x \otimes x^*) \otimes x \ar[r]^{(1 \otimes f) \otimes 1} & (x \otimes x^*) \otimes x \ar[r]^{\ \ \ \ e \otimes 1} & \mathbf{1} \otimes x \ar[u]^{\l}
}
\end{equation*}
Since $\cC$ is semi-simple, it is enough to assume $x$ is simple.  In this case $\varphi_x$ is the identity on $k$.  Hence it is enough to show $\varphi_x$ sends the composite in (RDA-II) to $1_x$.   In graphical terms, this means demonstrating

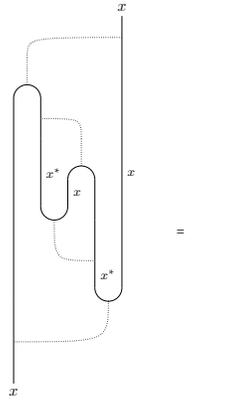
\begin{center}
\scalebox{.6}{
}
\end{center}

\noindent We begin by writing the left hand side as the composition of two morphisms

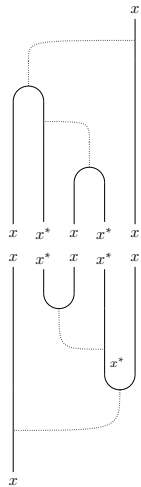
\begin{center}
\scalebox{.6}{%
}

\end{center}

\noindent Using the triangle identity twice we have

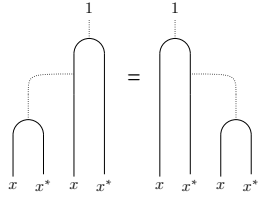
\begin{center}
\scalebox{.6}{%
}

\end{center}
\label{fig:step3}

\noindent Applying the above identity and far commutativity we get
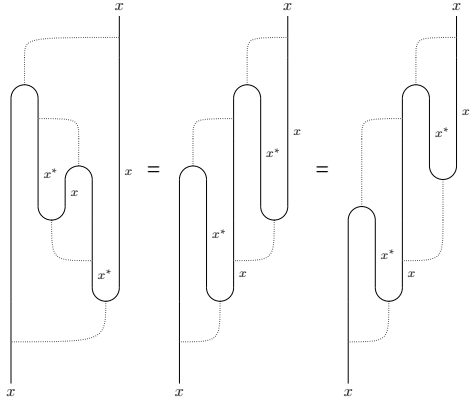
\begin{center}
\scalebox{.6}{%
}
\end{center}
\label{fig:step4}

\noindent Applying far commutativity and $(1_\mathbf{1} \otimes \r_x) \circ (\l_x^{-1} \otimes 1_\mathbf{1}) = \l_x^{-1} \circ \r_x$ to right hand side we get
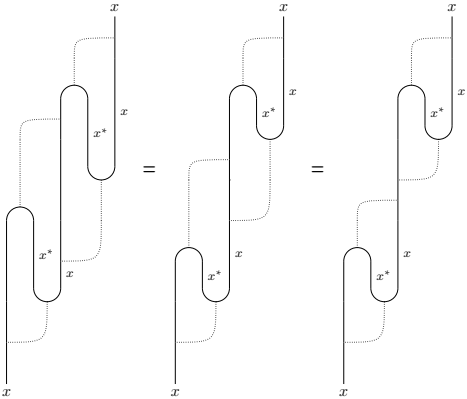
\begin{center}
\scalebox{.6}{%
}
\label{fig:step5}
\end{center}

\noindent Applying the first right duality axiom twice to the left hand side we get what we wanted to prove.
\end{proof}

The assumptions in Lemma \ref{lem:RA} can be weakened.  The proof only uses the fact that for every simple object $a$ the map $\varphi_a : \cC(a^*,a^*) \ra \cC(a,a)$ is injective.

\subsection{The Double Dual.} 

\label{ss:dd}

Let $\cC$ be a monoidal category.  A monoidal functor $F=(F_0,F_1,F_2)$ is a triplet consisting of a functor $F_1: \cC \ra \cC$, an isomorphism $F_0: F_1(\mathbf{1}) \ra \mathbf{1}$, and a natural isomorphism $F_2 : \otimes \circ (F_1 \times F_1) \ra F_1 \circ \otimes$, all together satisfying certain compatibility conditions (\cite{Ka}, Definition XI.4.1).

Let $\cC$ be a fusion category.  The double-dual is the monoidal functor given by the triplet $\Delta = (\Delta_0,\Delta_1,\Delta_2)$, where $\Delta_0$ may be taken to be the identity on $\mathbf{1}$, $\Delta_1 : \cC \ra \cC$ is the functor defined by taking the double-dual $(-)^{**}$ of object and morphism, and $\Delta_2$ is the natural isomorphism $\Delta_2 : \otimes \circ (\Delta_1 \times \Delta_1) \ra \Delta_1 \circ \otimes$ defined graphically to be
\begin{center} 
\raisebox{5cm}{$\Delta_2(x,y)=$}
\scalebox{.6}{\begin{tikzpicture}
\makeatletter{}\draw (0.7,9.3) node [minimum height=0, minimum width=1.2cm] (deltadualdualCircleTimesxy0) {$(xy)^{**}$};
\draw (2.35,8.4) node [] (deltarhodualdualCircleTimesxy0) {};
\coordinate (deltarhodualdualCircleTimesxy0pos1) at (0.7,8.175);
\coordinate (deltarhodualdualCircleTimesxy0pos2) at (4.,8.175);
\coordinate (deltarhodualdualCircleTimesxy0center) at (0.7,8.625);
\coordinate (deltarhodualdualCircleTimesxy0control) at (4.,8.625);
\draw (deltarhodualdualCircleTimesxy0pos1) to (deltarhodualdualCircleTimesxy0center);
\draw[densely dotted] (deltarhodualdualCircleTimesxy0pos2)..controls (deltarhodualdualCircleTimesxy0control) and (deltarhodualdualCircleTimesxy0control)..(deltarhodualdualCircleTimesxy0center);
\draw (0.7,5.85) node [minimum height=0] (deltaiddualdualCircleTimesxy1) {$$};
\draw (4.,5.85) node [] (deltaddualydualxCircleTimesxy0) {} +(-2.1,-0.225) coordinate (deltaddualydualxCircleTimesxy0pos1) arc (180:90:2.1) coordinate (deltaddualydualxCircleTimesxy0center) arc (90:0:2.1) coordinate (deltaddualydualxCircleTimesxy0pos6)
(deltaddualydualxCircleTimesxy0)+(-1.5,-0.225) coordinate (deltaddualydualxCircleTimesxy0pos2) arc (180:90:1.4) arc (90:0:1.4) coordinate (deltaddualydualxCircleTimesxy0pos5)
(deltaddualydualxCircleTimesxy0)+(-0.6,-0.225) coordinate (deltaddualydualxCircleTimesxy0pos3) arc (180:90:0.45) arc (90:0:0.45) coordinate (deltaddualydualxCircleTimesxy0pos4)
;
\draw (0.7,4.95) node [minimum height=0] (deltaiddualdualCircleTimesxy0) {$$};
\draw (1.9,4.95) node [minimum height=0] (deltaiddualy0) {$$};
\draw (2.5,4.95) node [minimum height=0] (deltaiddualx13) {$$};
\draw (3.4,4.95) node [draw, minimum width=0.8cm] (deltatensorxy0) {$1_{xy}$};
\draw (4.3,4.95) node [minimum height=0] (deltaiddualCircleTimesxy0) {$$};
\draw (5.3,4.95) node [minimum height=0] (deltaiddualdualx1) {$$};
\draw (6.1,4.95) node [minimum height=0] (deltaiddualdualy2) {$$};
\draw (2.5,2.7) node [] (deltabxydualCircleTimesxy0) {} +(-1.8,1.575) coordinate (deltabxydualCircleTimesxy0pos1) arc (180:270:1.8) coordinate (deltabxydualCircleTimesxy0center) arc (270:360:1.8) coordinate (deltabxydualCircleTimesxy0pos6)
(deltabxydualCircleTimesxy0)+(-0.6,1.575) coordinate (deltabxydualCircleTimesxy0pos2) arc (180:270:0.9) arc (270:360:0.9) coordinate (deltabxydualCircleTimesxy0pos5)
(deltabxydualCircleTimesxy0)+(0.,1.575) coordinate (deltabxydualCircleTimesxy0pos3) arc (180:270:0.3) arc (270:360:0.3) coordinate (deltabxydualCircleTimesxy0pos4)
;
\draw (5.3,2.7) node [minimum height=0] (deltaiddualdualx0) {$$};
\draw (6.1,2.7) node [minimum height=0] (deltaiddualdualy1) {$$};
\draw (3.9,1.8) node [] (deltaPowerlambdadualdualx-10) {};
\coordinate (deltaPowerlambdadualdualx-10control) at (2.5,1.575);
\coordinate (deltaPowerlambdadualdualx-10center) at (5.3,1.575);
\coordinate (deltaPowerlambdadualdualx-10pos1) at (2.5,2.475);
\coordinate (deltaPowerlambdadualdualx-10pos2) at (5.3,2.475);
\draw (deltaPowerlambdadualdualx-10center) to (deltaPowerlambdadualdualx-10pos2);
\draw[densely dotted] (2.5,2.475)..controls (deltaPowerlambdadualdualx-10control) and (deltaPowerlambdadualdualx-10control)..(deltaPowerlambdadualdualx-10center);
\draw (6.1,1.8) node [minimum height=0] (deltaiddualdualy0) {$$};
\draw (5.3,0.45) node [minimum height=0, minimum width=0.6cm] (deltadualdualx0) {$x^{**}$};
\draw (6.1,0.45) node [minimum height=0, minimum width=0.6cm] (deltadualdualy0) {$y^{**}$};
\draw[solid] ([xshift=0.cm]deltadualdualx0.north) to (deltaPowerlambdadualdualx-10center);
\draw[densely dotted] (deltaPowerlambdadualdualx-10pos1) to (deltabxydualCircleTimesxy0center);
\draw[solid] (deltabxydualCircleTimesxy0pos4) to node [right,font=\footnotesize] {$x$} ([xshift=-0.3cm]deltatensorxy0.south);
\draw[solid] (deltabxydualCircleTimesxy0pos5) to node [right,font=\footnotesize] {$y$} ([xshift=0.3cm]deltatensorxy0.south);
\draw[solid] (deltabxydualCircleTimesxy0pos2) to node [right,font=\footnotesize] {$y^*$} (deltaddualydualxCircleTimesxy0pos1);
\draw[solid] (deltabxydualCircleTimesxy0pos3) to node [right,font=\footnotesize] {$x^*$} (deltaddualydualxCircleTimesxy0pos2);
\draw[solid] ([xshift=0.cm]deltatensorxy0.north) to node [right,font=\footnotesize] {$xy$} (deltaddualydualxCircleTimesxy0pos3);
\draw[solid] (deltabxydualCircleTimesxy0pos6) to node [right,font=\footnotesize] {$(xy)^*$} (deltaddualydualxCircleTimesxy0pos4);
\draw[solid] (deltaPowerlambdadualdualx-10pos2) to (deltaddualydualxCircleTimesxy0pos5);
\draw[solid] ([xshift=0.cm]deltadualdualy0.north) to (deltaddualydualxCircleTimesxy0pos6);
\draw[solid] (deltabxydualCircleTimesxy0pos1) to node [right,font=\footnotesize] {$(xy)^{**}$} (deltarhodualdualCircleTimesxy0pos1);
\draw[densely dotted] (deltaddualydualxCircleTimesxy0center) to (deltarhodualdualCircleTimesxy0pos2);
\draw[solid] (deltarhodualdualCircleTimesxy0center) to ([xshift=0.cm]deltadualdualCircleTimesxy0.south);
 
\end{tikzpicture}}
\end{center}
The isomorphism $\Delta_2(x,y)$ measures the discrepancy between co-evaluations $c_x \otimes c_y$ and $c_{x \otimes y}$, and evaluations $e_{x^*} \otimes e_{y^*}$ and $e_{(x \otimes y)^*}$; there is no reason to assume co-evaluation and evaluation morphisms behave monoidally.  

When the fusion category $\cC$ is {\em skeletal}, $\Delta_1$ may be taken to be the identity $1_\cC$.  In the non-skeletal case, the double-dual is naturally isomorphic to the identity $1_\cC$ however not monoidally.  If all possible, we refrain from using $\Delta_1(x) \equiv x$ to make our proofs more easily applicable to non-skeletal categories.

By definition, $\Delta_2(\mathbf{1},y) = 1_y$ and $\Delta_2(x,\mathbf{1}) = 1_x$, and 
\begin{center} 
\raisebox{5cm}{$\Delta_2(x,y)^{-1}=$}
\scalebox{.6}{
}
\end{center}

Let us generalize $\Delta_2$ to fit more general scenarios.  For $a,b,x \in \cC$ and $f : a \otimes b \ra x$, let $\Delta(a,b,x,f) : \Delta_1(a) \otimes \Delta_1(b) \ra \Delta_1(x)$ denote the morphism 
\begin{center} 
\label{def:D}
\raisebox{5cm}{$\Delta(a,b,x,f) =$}
\scalebox{.6}{%
}
}
\]

\end{proof}

\subsection{Pivotal and Spherical Structure} 

\label{ss:ps}

Let $\cC$ be a fusion category.  A {\em pivotal} structure on $\cC$ constitutes a monoidal natural isomorphism
\[ \eps : \Delta \xra{\ \cong\ } 1_\cC, \]
where $\Delta$ is the double dual functor discussed in \S\ref{ss:dd}.  To ensure $\eps$ is monoidal the following diagram must commute
\begin{equation}
\label{eq:eta1}
\xymatrix{
\Delta_1(x) \otimes \Delta_1(y) \ar[rr] ^{\eps(x) \otimes \eps(y)} \ar[d] _{\Delta_2(x,y)} && x \otimes y \ar[d] ^{1_{x \otimes y}} \\
\Delta_1(x \otimes y) \ar[rr] _{\eps(x \otimes y)} && x \otimes y
}\end{equation}
Pivotal structures have the following duality-respecting symmetry.
\begin{lemma}
\label{lem:eps_sym}
$ \eps_{x^*} = (\eps_x^{-1})^*$.
\end{lemma}
\begin{proof}
By $\eps$ being natural and monoidal:
\[
\xymatrix{
\Delta_1(a) \Delta_1(a^*) \ar[rr]^{\Delta_2(a,a^*)} \ar[d]^{\eps_a \otimes \eps_{a^*}} && 
\Delta_1(a \otimes a^*) \ar[rr]^{\Delta_1(e_a)} \ar[d]_{\eps_{a \otimes a^*}} && 
\Delta_1(\mathbf{1}) \ar[d]_{\eps_{\mathbf{1}}} \\
a \otimes a^* \ar[rr]_{1} &&
a \otimes a^* \ar[rr]_{e_a} &&
\mathbf{1}
}
\]
Then we have
\[ e_a \circ ( \eps_a \otimes \eps_{a^*}  ) = \Delta_1(a) \circ \Delta_2(a,a^*) \]
By Lemma \ref{lem:g_in_Delta},
\[ e_a \circ ( \eps_a \otimes \eps_{a^*}  ) = \Delta(a,a^*,\mathbf{1},e_a) \]
Using right duality axiom twice reduces the right-hand side to
\[ e_a \circ ( \eps_a \otimes \eps_{a^*}  ) = e_{\Delta_1(a)} \]
Pre-composing both sides with $\eps_a^{-1} \otimes 1$
\[ e_a \circ ( 1_a \otimes \eps_{a^*}  ) = e_{\Delta_1(a)} \circ (\eps_a^{-1} \otimes 1) \]
Tensoring with $1_{a^*}$ on the left and pre-compose with $c_a \otimes 1$
\[ (1_{a^*} \otimes e_a) \circ (1_{a^*} \otimes 1_a \otimes \eps_{a^*}) \circ (c_a \otimes 1) =
(1_{a^*} \otimes e_{\Delta_1(a)}) \circ (1_{a^*} \otimes \eps_{a}^{-1} \otimes 1) \circ (c_a \otimes 1) \]
where $\l$, $\r$ and $\a$ should be inserted where appropriate.  Then by right duality axiom applied to left-hand side, and the definition of the dual of a morphism applied to right-hand side we get
\[ \eps_{a^*} = (\eps_{a}^{-1})^* \]
\end{proof}
The {\em left quantum dimension}, $q_l(x) : \mathbf{1} \ra \mathbf{1}$, and {\em right quantum dimension}, $q_r(x) : \mathbf{1} \ra \mathbf{1}$, are defined to be
\begin{eqnarray}
q_l(x) & \defeq & e_x \circ (\eps(x) \otimes 1_{x^*}) \circ c_{x^*} \label{def:qlx} \\
q_r(x) & \defeq & e_{x^*} \circ (1_{x^*} \otimes \eps(x)^{-1})  \circ c_x \label{def:qrx}
\end{eqnarray}
As a consequence of Lemma \ref{lem:eps_sym}, we have
\[ q_l(x^*) = q_r(x) \]

In \cite{Mu}, M\"uger proves a pivotal structure on a fusion category is {\em spherical} if and only if  the left and right quantum dimensions of every object $x \in \cC$ are identical. When a pivotal structure in $\cC$ is spherical we take $q_x$ to be the number defined via
\begin{equation} 
q_r(x) = q_x 1_{\mathbf{1}} = q_l(x) 
\label{eq:qdim} 
\end{equation}
known as the {\em quantum dimension} of $x$.

\subsection{Balancing} 

\label{ss:balance}

A braiding $\b$ in a monoidal category $\cC$ is a natural isomorphism $\b : \otimes \ra \otimes^{\text{op}}$ which satisfies the two Hexagon relations
\begin{equation}
\label{eq:hex1}
\xymatrix{
a \otimes (b \otimes c) \ar[rr] ^{1 \times \b_{b,c}} \ar[d] _{\a_{a,b,c}} && a \otimes (c \otimes b) \ar[rr] ^{\a_{a,c,b}} && (a \otimes c) \otimes b \ar[d] ^{\b_{a,c} \times 1} \\
(a \otimes b) \otimes c \ar[rr] _{\b_{(a \otimes b),c}} && c \otimes (a \otimes b) \ar[rr] _{\a_{c,a,b}} && (c \otimes a) \otimes b
}
\end{equation}
\begin{equation}
\label{eq:hex2}
\xymatrix{
a \otimes (b \otimes c) \ar[rr] ^{1 \times \b_{c,b}^{-1}} \ar[d] _{\a_{a,b,c}} && a \otimes (c \otimes b) \ar[rr] ^{\a_{a,c,b}} && (a \otimes c) \otimes b \ar[d] ^{\b_{c,a}^{-1}  \times 1} \\
(a \otimes b) \otimes c \ar[rr] _{\b_{c,(a \otimes b)}^{-1}} && c \otimes (a \otimes b) \ar[rr] _{\a_{c,a,b}} && (c \otimes a) \otimes b}
\end{equation}
A monoidal category with a braiding is said to be {\em braided}.

A braided monoidal category $\cC$ is \emph{balanced} if it is equipped with a natural isomorphism of the identity functor $\t : 1_\cC \ra 1_\cC$, which satisfies for every $x,y \in \cC$
\begin{eqnarray}
\t_{x \otimes y} & = & \b_{y,x} \circ \b_{x,y} \circ \t_x \otimes \t_y \label{eq:balancing1} \\%
\t_{\mathbf{1}} & = & 1_{\mathbf{1}} \label{eq:balancing3}
\end{eqnarray}
Unless $\cC$ is symmetric, relation (\ref{eq:balancing1}) implies $\t$ is not monoidal.  A balanced category $\cC$ with right duals is {\em tortile} if $\t$ respects duality, namely,
\begin{equation} \t_{x^*} = (\t_x)^* \label{eq:balancing2} \end{equation}
A {\em balancing} $\t$ which is tortile is also known in the literature as a {\em twist}.\footnote{We define balancing, following Yetter's earlier terminology in \cite{Ye}.  In \cite{BK} a balancing is assumed automatically to be tortile, namely, a balancing according to \cite{BK} satisfies (\ref{eq:balancing1}), (\ref{eq:balancing3}) and (\ref{eq:balancing2}).  In light of present-day distinction between pivotal and spherical structures we find the recourse to Yetter's terminology justified.}

It is shown in \cite{Ye} that $\cC$ is balanced if and only if it has pivotal structure. The relation established between $\t$ and $\eps$ is given by 
\begin{equation} \t_x = \psi_x \circ \eps_x^{-1}
\label{def:balancing} 
\end{equation}
where $\psi$ is the natural isomorphism $\psi : \Delta \ra 1_\cC$ defined by tracing the following commutative diagram from $x^{**}$ on the left to $x$ on the right (adding $\a$, $\l$ and $\r$ as needed)
\begin{equation}
\xymatrix{
& x^{**} \otimes x^* \otimes x \ar[rr]^{\b^{-1} \otimes 1} \ar[dd]_{\b_{x^{**},x^*x}} && x^* \otimes x^{**} \otimes x \ar[dr]^{e_{x^*} \otimes 1} \ar[dd]^{\b^{-1}_{x^*x^{**},x}} &\\
x^{**} \ar[ur]^{1 \otimes c_x} \ar[dr]_{c_x \otimes 1} &&&& x \\
& x^* \otimes x \otimes x^{**} \ar[rr]_{\b^{-1} \otimes 1} \ar[rruu]^{1 \otimes \b} && x \otimes x^* \otimes x^{**} \ar[ur]_{1 \otimes e_{x^*}} &
}
\label{def:psi}
\end{equation}
Graphically, $\psi$ is given by any of the following three pictures:
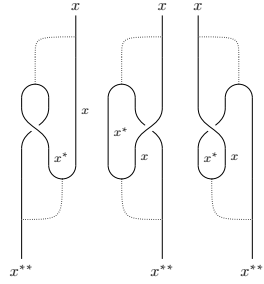
\begin{center} 
\scalebox{.6}{
}
\end{center}

\begin{lemma} 
Let $\cC$ be a braided fusion category. Then $\cC$ is tortile if and only if it is spherical.
\end{lemma}

This means a spherical braided fusion category is {\em ribbon}.

\begin{proof}
The existence of a pivotal structure $\eps$ implies the existence of a balancing $\t$, and vice versa.  Adding condition (\ref{eq:balancing2}), one needs $\eps$ to be spherical and vice versa.  By linear extension (Lemma \ref{lem:nonsense}) it is enough to test this for a simple object $a \in \cC$:

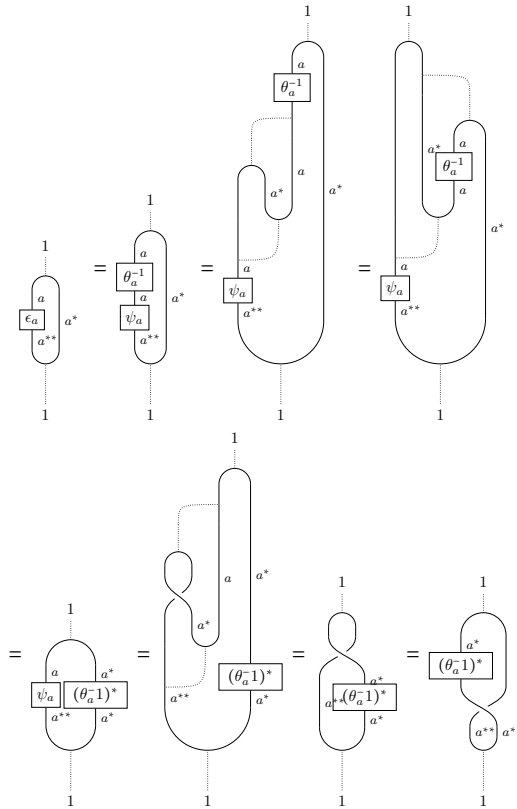
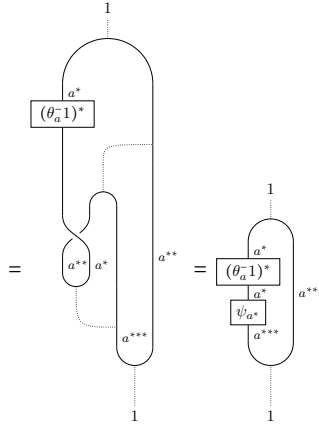
\begin{center} 
\scalebox{.6}{%
}

\end{center}

\noindent Therefore, for a simple object $a \in \cC$,
\[ q_l(a) = q_r(a) \Longleftrightarrow (\t_a)^* = \t_{a^*} \]
\end{proof}

\begin{remark} The proof above works for fusion categories with absolute simple objects.  One does not need the context of linear abelian categories to discuss spherical and tortile structures.  At the moment, we are not aware of a proof correlating spherical and tortile structures in the more general context of braided monoidal categories with right duals. \end{remark}

\subsection{Scalar Extension}

\label{sec:se} 

Let $\cC$ be a fusion category defined over $k$ and $\sigma:k \rightarrow k'$ a non-zero morphism of fields.  Define the scalar extension $\cC \otimes_{k}^\sigma k'$ of $\cC$ with respect to $\sigma$ to be the $k'$-linear category with a collection of objects
\[\text{Ob}(\cC \otimes_{k}^\sigma k') = \text{Ob}(\cC)\]
and for every pair of objects $x,y \in \text{Ob}(\cC \otimes_{k}^\sigma K)$, a morphism $k'$-space
\[
(\cC \otimes_{k}^\sigma k') (x,y) = \cC(x,y) \otimes_{k}^\sigma k'
\]
where the tensor product on the right hand side is taken with respect to $\sigma$.

\begin{remark}
The above definition of scalar extension needs refining for general $k$-linear abelian rigid monoidal categories.  If, for example, $G$ is a finite group, $k$ is a field of characteristic zero, which is not a splitting field for $G$, and $k \subset k'$ is a field extension, then $\Rep_k G \otimes_k k'$ might not be equivalent to $\Rep_{k'} G$; some irreducibles in $\Rep_k G$ might decompose over $k'$. However, if $k$ is not a splitting field for $G$, $\Rep_k G$ is not fusion by our definition, since there is necessarily an irreducible $V \in \Rep_kG$ such that $\Rep_kG(V,V) \neq k$. A more general notion of a scalar extension of $k$-linear abelian rigid monoidal categories was introduced in \cite{St}.
\end{remark}

When $k'=k$ Galois closed and $\sigma \in \Gal(k/\QQ)$ we refer to $\cC \otimes_{k}^\sigma k$ as the \emph{Galois twist} of $\cC$ with respect to $\sigma$ and denote it by $\cC^\sigma$.

\begin{lemma}
\label{lem:ext}
Let $\cC$ be a fusion (modular) category defined over $k$.  Then $\cC \otimes_{k}^\sigma k'$ is a fusion (modular) category defined over $k'$.
\end{lemma}

\begin{proof}
Let $\cC$ be a fusion category over $k$.  As a $k$-linear abelian category, $\cC$ is equivalent to a finite product of $\Vect_k$'s.  In turn, $\cC\otimes_k^\sigma k'$ is equivalent as a $k'$-linear abelian category to a finite product of $\Vect_{k'}$'s.  Showing $\cC \otimes_{k}^\sigma k'$ is rigid monoidal is a straight forward application of definitions, therefore $\cC \otimes_{k}^\sigma k'$ is a fusion category defined over $k'$. Assume in addition that $\cC$ is modular.  Showing $\cC \otimes_{k}^\sigma k'$ is ribbon is again a straight forward application of definitions.  The $S$-matrix associated with $\cC \otimes_{k}^\sigma k'$ is $\sigma(S)$ where $\sigma$ acts entry-wise and $S$ is the $S$-matrix of $\cC$.  If the later is invertible then so is the former, therefore $\cC \otimes_{k}^\sigma k'$ is modular.
\end{proof}

\section{Fusion and Modular Systems} 

\subsection{Fusion System.}
\begin{definition}
\label{def:fusion}
A \emph{fusion system} $(L,N,F)$ defined over $k$ consists of the following:
\begin{mylist}
\item \ \ A finite set $L$ with a distinguished element $\mathbf{1}\in L$.
\item \ \ An involution $*:L \rightarrow L$ such that $\mathbf{1}^*=\mathbf{1}$.
\item \ \ A collection of non-negative integers $N_{ab}^c$ associated with every triple $a,b,c\in L$ subject to the constraints
\begin{eqnarray}
&& N_{\mathbf{1} a}^b = \delta_{ab} = N_{a\mathbf{1}}^b \label{eqn:unit} \\
&& N_{ab}^{\mathbf{1}} = \delta_{a^*b} \label{eqn:dual} \\
&& N_{abc}^u := \sum_{e} N_{ab}^{e} N_{ec}^u =  \sum_{e'} N_{ae'}^u N_{bc}^{e'} \label{eqn:fusion}
\end{eqnarray}
\item \ \ For every quadruple $a,b,c,u\in L$, an invertible matrix $F_{abc}^u \in \Mat_{N_{abc}^u \times N_{abc}^u}(k)$, with each entry denoted
\[
\begin{array}{ccccccc}
F_{abc}^u \left[ \begin{array}{ccc} i & e & j \\ i' & e' & j' \end{array} \right]
& , &
	\begin{array}{l}
		e \in\ L \\
		e' \in L \rule{0ex}{3ex}
	\end{array}
& , &
	\begin{array}{l}
		i\in\{1,\ldots,N_{ae}^u\} \\
		i'\in\{1,\ldots,N_{ab}^{e'}\} \rule{0ex}{3ex}
	\end{array}
& , &
	\begin{array}{l}
		j\in\{1,\ldots,N_{bc}^e\} \\
		j'\in\{1,\ldots,N_{e'c}^u\} \rule{0ex}{3ex}
	\end{array}
\end{array}
\]
subject to the constraints
\begin{eqnarray} 
&& F_{a\mathbf{1}b}^u = I_{N_{ab}^u}  \label{eqn:triangle} \\ 
&& \rule{0ex}{6ex} u_a :=F_{a\cech{a}a}^{a}
\left[ \scalebox{.8}{$\begin{array}{ccc} 1 & 1 & 1 \\ 
1 & 1 & 1 \end{array}$} \right] \neq 0  \label{eqn:duality} \\  
&& 
\sum_{l=1}^{N_{eg}^u} 
F_{ecd}^u 
\left[ \scalebox{.6}{$\begin{array}{ccc} j & f & k \\ l & g & m \end{array}$} \right] 
F_{abg}^u 
\left[ \scalebox{.6}{$\begin{array}{ccc} i & e & l \\ n & h & o \end{array}$} \right] 
= \label{eqn:pentagon} \\  
&& 
\sum_{q \in L} \sum_{p=1}^{N_{aq}^f} \sum_{r=1}^{N_{bc}^q} \sum_{v=1}^{N_{qd}^h}
F_{abc}^f 
\left[ \scalebox{.6}{$\begin{array}{ccc} i & e & j \\ p & q & r \end{array}$} \right] 
 F_{aqd}^u
\left[ \scalebox{.6}{$\begin{array}{ccc} p & f & k \\ n & h & v \end{array}$} \right]
F_{bcd}^h
\left[ \scalebox{.6}{$\begin{array}{ccc} r & q & v \\ o & g & m \end{array}$} \right] 
\nonumber
\end{eqnarray}
\end{mylist}
for any $e,f,g,h \in L$, $i \in \{1,\ldots,N_{ab}^e\}$, $j \in \{1,\ldots,N_{ec}^f\}$, $k \in \{1,\ldots,N_{fd}^u\}$, $n \in \{1,\ldots,N_{ah}^u\}$, $o \in \{1,\ldots,N_{bg}^h\}$ and $m \in \{1,\ldots,N_{cd}^g\}$.
\end{definition}

\begin{remark} To actually make sense of the labeling of entries of $F_{abc}^u$ a bijection between the sets of labels $\{(i,e,j)\}$, $\{(i',e',j')\}$ and $\{1,\ldots,N_{abc}^u\}$ is required.  Such a bijection is possible for example if an order is chosen on the set $L$, and lexicographic order on $\{(i,e,j)\}$ and $\{(i',e',j')\}$ is established.   \end{remark}

\begin{example}
\label{ex:Fib}
Consider $L=\{\mathbf{1},x\}$ with fusion multiplicities $N_{ab}^c$ derived from $x^2 = \mathbf{1} + x$, namely, $N_{xx}^\mathbf{1} = 1 = N_{xx}^x$. In particular $x$ is self dual.  Denote
\[ z := F_{xxx}^\mathbf{1} \left[ \scalebox{.6}{$\begin{array}{ccc} 1 & x & 1 \\ 1 & x & 1 \end{array}$} \right] \]
\[ z_{11} := F_{xxx}^x \left[ \scalebox{.6}{$\begin{array}{ccc} 1 & \mathbf{1} & 1 \\ 1 & \mathbf{1} & 1 \end{array}$} \right] \,,\,
 z_{12} := F_{xxx}^x \left[ \scalebox{.6}{$\begin{array}{ccc} 1 & \mathbf{1} & 1 \\ 1 & x & 1 \end{array}$} \right] \,,\,
 z_{21} := F_{xxx}^x \left[ \scalebox{.6}{$\begin{array}{ccc} 1 & x & 1 \\ 1 & \mathbf{1} & 1 \end{array}$} \right] \,,\,
 z_{22} := F_{xxx}^x \left[ \scalebox{.6}{$\begin{array}{ccc} 1 & x & 1 \\ 1 & x & 1 \end{array}$} \right] \]
where $z \neq 0$ since $F_{xxx}^\mathbf{1}$ is invertible and $z_{11} \neq 0$ due to \eqref{eqn:duality}.  Equations \eqref{eqn:pentagon} then read for the case $a=b=c=d=x$ and $u=\mathbf{1}$
\[ 1 = z_{11}^2 + z_{12} z_{21} z \quad , \quad z^2 = z_{12}z_{21} + z_{22}^2 z \]
and for the case $a=b=c=d=x$ and $u=x$
\[z_{11} z^2 = z_{12} z_{21} \quad , \quad z_{11} = z_{12} z_{21} z \]
\[z_{22} z_{21} = z_{22} z_{21} z\quad , \quad z_{21} = z_{21} z_{11} + z_{22}^2 z_{21}\]
\[z_{22} z_{12} = z_{22} z_{12} z\ ,\ z_{12} = z_{11} z_{12} + z_{12} z_{22}^2 \ ,\ z_{22} = z_{21} z_{12} + z_{22}^3\]
As a result we have 
\[z=1 \ ,\ z_{11} = \pm z_{22} \ ,\ z_{11} + z_{11}^2 =1 \ ,\ z_{11} = z_{12} z_{21}\]
the requirement that $F_{xxx}^x$ is invertible is then automatically satisfied.  The case $z_{11} = -z_{22}$ results in the {\em Fibonacci theory} and its complex conjugate, and the case $z_{11} = z_{22}$ results in the {\em Yang-Lee theory} and its complex conjugate \cite{RSW}.  Further discussion of these theories can be found in \S\ref{sec:Gal}.
\end{example}

\begin{lemma}
\label{lem:fusion_extract}
Given a fusion category $\cC$ defined over $k$,  one can extract from it a fusion system defined over $k$.
\end{lemma}

The fusion system to be extracted from $\cC$ depends uniquely on a choice of basis $\{\eta_{ab}^{c,i}\}_{i=1}^{N_{ab}^c}$ for each morphism $k$-space $\cC(a \otimes b,c)$.  This choice is restricted so that $\{\l_a\}$ and $\{\r_a\}$ are taken to be the bases for the morphism $k$-spaces $\cC(\mathbf{1} \otimes a,a)$ and $\cC(a \otimes \mathbf{1},a)$ respectively, and $\{e_a\}$ is taken to be the basis for the morphism $k$-space $\cC(a \otimes a^*, \mathbf{1})$.  We denote a fusion system extracted from $\cC$ by $(L,N,F)(\cC)$.

\begin{proof}  

Let $\cC$ be a fusion category defined over $k$.  Let $L$ be a set of representatives of isomorphism classes of simple objects in $\cC$ one of which is the neutral object $\mathbf{1} \in L$.  For any simple object, $a^{**}\cong a$ (see \cite{O}), and $\mathbf{1}^* \cong \mathbf{1}$.  Therefore, taking right duals defines an involution on $L$ which fixes $\mathbf{1}$.
For every triple $a,b,c\in L$ set
\[ N_{ab}^c := \dim_k \cC(a \otimes b,c) \]
known in the literature as {\em fusion multiplicities}.  Since $\mathbf{1} \otimes a \cong a \cong a \otimes \mathbf{1}$ we have $N_{\mathbf{1}a}^b = \dim_k \cC(a,b) = N_{a\mathbf{1}}^b$.  By simplicity $\dim_k \cC(a,b) = \d_{ab}$, implying the constraint in (\ref{eqn:unit}).  There is a natural isomorphism $\cC(a \otimes b, \mathbf{1}) \cong \cC(b,a^* \otimes \mathbf{1})$, implying, by simplicity of $a^*$ and $b$, the constraint in (\ref{eqn:dual}).
There are isomorphisms
\begin{eqnarray*}
&& \bigoplus_e \cC(a \otimes b,e) \otimes \cC(e \otimes c,u) \ra \cC((a \otimes b) \otimes c,u) 
\quad , \quad
f \otimes g \mapsto g \circ (f \otimes 1_c) \\
&& \bigoplus_{e'} \cC(a \otimes e',u) \otimes \cC(b \otimes c,e') \ra \cC(a \otimes (b \otimes c),u)
\quad , \quad
f' \otimes g' \mapsto f' \circ (1_a \otimes \, g')
\end{eqnarray*}
which by associativity of monoidal structure imply the constraint in (\ref{eqn:fusion}).

The associator $\a_{a,b,c} : a \otimes (b \otimes c) \ra (a \otimes b) \otimes c$ defines the pullback
\[ \a^*_{a,b,c} : \cC((a \otimes b) \otimes c,u) \ra \cC(a \otimes (b \otimes c),u) \]
for every $u \in L$.  Let us choose a basis $\{\eta_{ab}^{c,i}\}_{i=1}^{N_{ab}^c}$ for each morphism $k$-space $\cC(a \otimes b,c)$, making sure to pick $\{\l_a\}$ and $\{\r_a\}$ as bases for the morphism $k$-spaces $\cC(\mathbf{1} \otimes a,a)$ and $\cC(a \otimes \mathbf{1},a)$ respectively, and $\{e_a\}$ as a basis for the morphism $k$-space $\cC(a \otimes a^*, \mathbf{1})$.

Graphically, our conventions are:
\begin{center}
\scalebox{.6}{\begin{tikzpicture}
\makeatletter{}\draw (0.8,2.7) node [minimum height=0, minimum width=0.4cm] (showetac43) {$c$};
\draw (0.8,1.8) node [draw, minimum width=1.4cm] (showetashoweta1) {$\eta_{a b}^{c,i}$};
\draw (0.5,0.45) node [minimum height=0, minimum width=0.4cm] (showetaa62) {$a$};
\draw (1.1,0.45) node [minimum height=0, minimum width=0.4cm] (showetab48) {$b$};
\draw[solid] ([xshift=0.cm]showetaa62.north) to ([xshift=-0.3cm]showetashoweta1.south);
\draw[solid] ([xshift=0.cm]showetab48.north) to ([xshift=0.3cm]showetashoweta1.south);
\draw[solid] ([xshift=0.cm]showetashoweta1.north) to ([xshift=0.cm]showetac43.south);
 
\end{tikzpicture}}
\raisebox{1cm}{$=$}
\scalebox{.6}{\begin{tikzpicture}
\makeatletter{}\draw (0.6,2.7) node [minimum height=0, minimum width=0.4cm] (etaabcic8) {$c$};
\draw (0.6,1.8) node [draw, rounded corners=3pt, fill=gray!10, minimum width=0.8cm] (etaabcifabci1) {$i$};
\draw (0.3,0.45) node [minimum height=0, minimum width=0.4cm] (etaabcia10) {$a$};
\draw (0.9,0.45) node [minimum height=0, minimum width=0.4cm] (etaabcib8) {$b$};
\draw[solid] ([xshift=0.cm]etaabcia10.north) to ([xshift=-0.3cm]etaabcifabci1.south);
\draw[solid] ([xshift=0.cm]etaabcib8.north) to ([xshift=0.3cm]etaabcifabci1.south);
\draw[solid] ([xshift=0.cm]etaabcifabci1.north) to ([xshift=0.cm]etaabcic8.south);
 
\end{tikzpicture}}
\hspace{1 cm}
\scalebox{.6}{\begin{tikzpicture}
\makeatletter{}\draw (0.6,2.7) node [minimum height=0, minimum width=0.4cm] (showlambdaaa64) {$a$};
\draw (0.6,1.8) node [draw, minimum width=0.8cm] (showlambdaashowlambdaa1) {$\lambda_a$};
\draw (0.3,0.45) node [minimum height=0, minimum width=0.4cm] (showlambdaa144) {$1$};
\draw (0.9,0.45) node [minimum height=0, minimum width=0.4cm] (showlambdaaa63) {$a$};
\draw[densely dotted] ([xshift=0.cm]showlambdaa144.north) to ([xshift=-0.3cm]showlambdaashowlambdaa1.south);
\draw[solid] ([xshift=0.cm]showlambdaaa63.north) to ([xshift=0.3cm]showlambdaashowlambdaa1.south);
\draw[solid] ([xshift=0.cm]showlambdaashowlambdaa1.north) to ([xshift=0.cm]showlambdaaa64.south);
 
\end{tikzpicture}}
\raisebox{1cm}{$=$}
\scalebox{.6}{\begin{tikzpicture}
\makeatletter{}\draw (0.9,2.7) node [minimum height=0, minimum width=0.4cm] (lambdaaa66) {$a$};
\draw (0.6,1.8) node [] (lambdaalambdaa4) {};
\coordinate (lambdaalambdaa4pos1) at (0.3,1.575);
\coordinate (lambdaalambdaa4pos2) at (0.9,1.575);
\coordinate (lambdaalambdaa4control) at (0.3,2.025);
\coordinate (lambdaalambdaa4center) at (0.9,2.025);
\draw (lambdaalambdaa4pos2) to (lambdaalambdaa4center);
\draw[densely dotted] (0.3,1.575)..controls (lambdaalambdaa4control) and (lambdaalambdaa4control)..(lambdaalambdaa4center);
\draw (0.3,0.45) node [minimum height=0, minimum width=0.4cm] (lambdaa145) {$1$};
\draw (0.9,0.45) node [minimum height=0, minimum width=0.4cm] (lambdaaa65) {$a$};
\draw[densely dotted] ([xshift=0.cm]lambdaa145.north) to (lambdaalambdaa4pos1);
\draw[solid] ([xshift=0.cm]lambdaaa65.north) to (lambdaalambdaa4pos2);
\draw[solid] (lambdaalambdaa4center) to ([xshift=0.cm]lambdaaa66.south);
 
\end{tikzpicture}}
\hspace{1 cm}
\scalebox{.6}{\begin{tikzpicture}
\makeatletter{}\draw (0.6,2.7) node [minimum height=0, minimum width=0.4cm] (showrhoaa68) {$a$};
\draw (0.6,1.8) node [draw, minimum width=0.8cm] (showrhoashowrhoa3) {$\rho_a$};
\draw (0.3,0.45) node [minimum height=0, minimum width=0.4cm] (showrhoaa67) {$a$};
\draw (0.9,0.45) node [minimum height=0, minimum width=0.4cm] (showrhoa146) {$1$};
\draw[solid] ([xshift=0.cm]showrhoaa67.north) to ([xshift=-0.3cm]showrhoashowrhoa3.south);
\draw[densely dotted] ([xshift=0.cm]showrhoa146.north) to ([xshift=0.3cm]showrhoashowrhoa3.south);
\draw[solid] ([xshift=0.cm]showrhoashowrhoa3.north) to ([xshift=0.cm]showrhoaa68.south);
 
\end{tikzpicture}}
\raisebox{1cm}{$=$}
\scalebox{.6}{\begin{tikzpicture}
\makeatletter{}\draw (0.3,2.7) node [minimum height=0, minimum width=0.4cm] (rhoaa70) {$a$};
\draw (0.6,1.8) node [] (rhoarhoa3) {};
\coordinate (rhoarhoa3pos1) at (0.3,1.575);
\coordinate (rhoarhoa3pos2) at (0.9,1.575);
\coordinate (rhoarhoa3center) at (0.3,2.025);
\coordinate (rhoarhoa3control) at (0.9,2.025);
\draw (rhoarhoa3pos1) to (rhoarhoa3center);
\draw[densely dotted] (rhoarhoa3pos2)..controls (rhoarhoa3control) and (rhoarhoa3control)..(rhoarhoa3center);
\draw (0.3,0.45) node [minimum height=0, minimum width=0.4cm] (rhoaa69) {$a$};
\draw (0.9,0.45) node [minimum height=0, minimum width=0.4cm] (rhoa147) {$1$};
\draw[solid] ([xshift=0.cm]rhoaa69.north) to (rhoarhoa3pos1);
\draw[densely dotted] ([xshift=0.cm]rhoa147.north) to (rhoarhoa3pos2);
\draw[solid] (rhoarhoa3center) to ([xshift=0.cm]rhoaa70.south);
 
\end{tikzpicture}}
\hspace{1 cm}
\scalebox{.6}{\begin{tikzpicture}
\makeatletter{}\draw (0.6375,2.7) node [minimum height=0, minimum width=0.4cm] (showepsilona148) {$1$};
\draw (0.6375,1.8) node [draw, minimum width=0.875cm] (showepsilonashowepsilona1) {$\epsilon_a$};
\draw (0.3,0.45) node [minimum height=0, minimum width=0.4cm] (showepsilonaa71) {$a$};
\draw (0.975,0.45) node [minimum height=0, minimum width=0.55cm] (showepsilonaduala4) {$a^*$};
\draw[solid] ([xshift=0.cm]showepsilonaa71.north) to ([xshift=-0.3375cm]showepsilonashowepsilona1.south);
\draw[solid] ([xshift=0.cm]showepsilonaduala4.north) to ([xshift=0.3375cm]showepsilonashowepsilona1.south);
\draw[densely dotted] ([xshift=0.cm]showepsilonashowepsilona1.north) to ([xshift=0.cm]showepsilona148.south);
 
\end{tikzpicture}}
\raisebox{1cm}{$=$}
\scalebox{.6}{\begin{tikzpicture}
\makeatletter{}\draw (0.6375,2.5875) node [minimum height=0, minimum width=0.4cm] (epsilona149) {$1$};
\draw (0.6375,1.8) node [] (epsilonada10) {} +(-0.3375,-0.225) coordinate (epsilonada10pos1) arc (180:90:0.3375) coordinate (epsilonada10center) arc (90:0:0.3375) coordinate (epsilonada10pos2)
;
\draw (0.3,0.45) node [minimum height=0, minimum width=0.4cm] (epsilonaa72) {$a$};
\draw (0.975,0.45) node [minimum height=0, minimum width=0.55cm] (epsilonaduala5) {$a^*$};
\draw[solid] ([xshift=0.cm]epsilonaa72.north) to (epsilonada10pos1);
\draw[solid] ([xshift=0.cm]epsilonaduala5.north) to (epsilonada10pos2);
\draw[densely dotted] (epsilonada10center) to ([xshift=0.cm]epsilona149.south);
 
\end{tikzpicture}}
\end{center}

 Such a choice implies bases 
\begin{eqnarray*}
&& \{\eta_{ec}^{u,j} \circ (\eta_{ab}^{e,i} \otimes 1_c)\}_{i,j=1}^{N_{ab}^e,N_{ec}^u} \subset \cC((a \otimes b) \otimes c,u) \\
&& \{\eta_{ae'}^{u,i'} \circ (1_a \otimes \eta_{bc}^{e',j'})\}_{i',j'=1}^{N_{ae'}^u,N_{bc}^{e'}} \subset \cC(a \otimes (b \otimes c),u)
\end{eqnarray*}

We let $F_{abc}^u$ be the matrix representing $\a^*_{a,b,c}$ in these respective bases.  Specifically,
\begin{equation}
\label{eq:F}
\eta_{ec}^{u,j} \circ (\eta_{ab}^{e,i} \otimes 1_c) \circ \a_{a,b,c} = \sum_{e' \in L} \sum_{i'=1}^{N_{ae'}^u} \sum_{j'=1}^{N_{bc}^{e'}} F_{abc}^u \left[ \scalebox{.6}{$\begin{array}{ccc} i & e & j \\ i' & e' & j' \end{array}$} \right] \eta_{ae'}^{u,i'} \circ (1_a \otimes \eta_{bc}^{e',j'})
\end{equation}
\begin{equation*}
\raisebox{1cm}{or graphically:   }
\scalebox{.6}{\begin{tikzpicture}
\makeatletter{}\draw (1.05,4.5) node [minimum height=0, minimum width=0.4cm] (leftbasisu10) {$u$};
\draw (1.05,3.6) node [draw, rounded corners=3pt, fill=gray!10, minimum width=1.1cm] (leftbasisfecuj0) {$j$};
\draw (0.6,2.7) node [draw, rounded corners=3pt, fill=gray!10, minimum width=0.8cm] (leftbasisfabei0) {$i$};
\draw (1.5,2.7) node [minimum height=0] (leftbasisidc7) {$$};
\draw (0.9,1.8) node [draw, minimum width=1.4cm] (leftbasisSubscriptalphaabc4) {$\alpha_{a,b,c}$};
\draw (0.3,0.45) node [minimum height=0, minimum width=0.4cm] (leftbasisa30) {$a$};
\draw (0.9,0.45) node [minimum height=0, minimum width=0.4cm] (leftbasisb20) {$b$};
\draw (1.5,0.45) node [minimum height=0, minimum width=0.4cm] (leftbasisc20) {$c$};
\draw[solid] ([xshift=0.cm]leftbasisa30.north) to ([xshift=-0.6cm]leftbasisSubscriptalphaabc4.south);
\draw[solid] ([xshift=0.cm]leftbasisb20.north) to ([xshift=0.cm]leftbasisSubscriptalphaabc4.south);
\draw[solid] ([xshift=0.cm]leftbasisc20.north) to ([xshift=0.6cm]leftbasisSubscriptalphaabc4.south);
\draw[solid] ([xshift=-0.6cm]leftbasisSubscriptalphaabc4.north) to node [right,font=\footnotesize] {$a$} ([xshift=-0.3cm]leftbasisfabei0.south);
\draw[solid] ([xshift=0.cm]leftbasisSubscriptalphaabc4.north) to node [right,font=\footnotesize] {$b$} ([xshift=0.3cm]leftbasisfabei0.south);
\draw[solid] ([xshift=0.cm]leftbasisfabei0.north) to node [right,font=\footnotesize] {$e$} ([xshift=-0.45cm]leftbasisfecuj0.south);
\draw[solid] ([xshift=0.6cm]leftbasisSubscriptalphaabc4.north) to node [right,font=\footnotesize] {$c$} ([xshift=0.45cm]leftbasisfecuj0.south);
\draw[solid] ([xshift=0.cm]leftbasisfecuj0.north) to ([xshift=0.cm]leftbasisu10.south);
 
\end{tikzpicture}}
\raisebox{1cm}{$=\sum_{e' \in L} \sum_{i'=1}^{N_{ae'}^u} \sum_{j'=1}^{N_{bc}^{e'}} F_{abc}^u \left[\scalebox{.6}{$\begin{array}{ccc} i & e & j \\ i' & e' & j' \end{array}$} \right]$}
\scalebox{.6}{\begin{tikzpicture}
\makeatletter{}\draw (0.75,3.6) node [minimum height=0, minimum width=0.4cm] (rightbasisu11) {$u$};
\draw (0.75,2.7) node [draw, rounded corners=3pt, fill=gray!10, minimum width=1.1cm] (rightbasisfaDerivative1euDerivative1i0) {$i'$};
\draw (0.3,1.8) node [minimum height=0] (rightbasisida35) {$$};
\draw (1.2,1.8) node [draw, rounded corners=3pt, fill=gray!10, minimum width=0.8cm] (rightbasisfbcDerivative1eDerivative1j0) {$j'$};
\draw (0.3,0.45) node [minimum height=0, minimum width=0.4cm] (rightbasisa31) {$a$};
\draw (0.9,0.45) node [minimum height=0, minimum width=0.4cm] (rightbasisb21) {$b$};
\draw (1.5,0.45) node [minimum height=0, minimum width=0.4cm] (rightbasisc21) {$c$};
\draw[solid] ([xshift=0.cm]rightbasisb21.north) to ([xshift=-0.3cm]rightbasisfbcDerivative1eDerivative1j0.south);
\draw[solid] ([xshift=0.cm]rightbasisc21.north) to ([xshift=0.3cm]rightbasisfbcDerivative1eDerivative1j0.south);
\draw[solid] ([xshift=0.cm]rightbasisa31.north) to ([xshift=-0.45cm]rightbasisfaDerivative1euDerivative1i0.south);
\draw[solid] ([xshift=0.cm]rightbasisfbcDerivative1eDerivative1j0.north) to node [right,font=\footnotesize] {$e'$} ([xshift=0.45cm]rightbasisfaDerivative1euDerivative1i0.south);
\draw[solid] ([xshift=0.cm]rightbasisfaDerivative1euDerivative1i0.north) to ([xshift=0.cm]rightbasisu11.south);
 
\end{tikzpicture}}
\end{equation*}
\begin{center}
\raisebox{2cm}{$=$}
\end{center}

The coefficients in (\ref{eq:F}) are known as the \emph{$F$-matrix coefficients}.  Note that $\cC$ being strictly associative does not imply $F_{abc}^u$ is necessarily the identity matrix.  

Recall $N_{a \mathbf{1}}^a = 1 = N_{\mathbf{1} b}^b$ and we chose $\eta_{a\mathbf{1}}^{a,1} = \r_a$, $\eta_{\mathbf{1} b}^{b,1} = \l_b$.  The triangular diagram (\ref{eqn:tri_diag}) implies for every 
$j \in \{1,\ldots,N_{ab}^u\}$
\[ \eta_{ab}^{u,j} \circ (\eta_{a\mathbf{1}}^{a,1} \otimes 1_b) \circ \a_{a,\mathbf{1},b} = 
\eta_{ab}^{u,j} \circ (1_a \otimes \eta_{\mathbf{1} b}^{b,1} ). \]

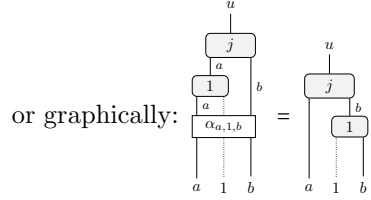
\begin{center}
\raisebox{1cm}{or graphically:}
\scalebox{.6}{
}
\end{center}

On the other hand, by definition of $F$ in (\ref{eq:F})
\[ \eta_{ab}^{u,j} \circ (\eta_{a\mathbf{1}}^{a,1} \otimes 1_b) \circ \a_{a,\mathbf{1},b} =
\sum_{i'=1}^{N_{ab}^u} F_{a\mathbf{1}b}^u \left[ \scalebox{.6}{$%
}
\end{center}

So it follows that
\[ F_{a\mathbf{1}b}^u \left[ \scalebox{.6}{$\begin{array}{ccc} 1 & a & j \\ i' & b & 1 \end{array}$} \right] = \d_{i'j}  \quad , \quad i',j \in \{1, \ldots, N_{ab}^u\} \]
meaning $F_{a\mathbf{1}b}^u = I_{N_{ab}^u}$, which establishes the constraint in (\ref{eqn:triangle}).

Let $\{ \e_{c,i}^{ab} \} \subset \cC(c,a \otimes b)$ be the dual basis to  $\{\e_{ab}^{c,i}\} \subset \cC(a \otimes b,c)$ (see \S\ref{ss:lac} for $x=a \otimes b$).  This implies
\begin{center}
\raisebox{2cm}{$F_{abc}^u \left[ \scalebox{.6}{$%
}
\end{center}
 By our choice of bases also $\eta_{aa^*}^{\mathbf{1},1} = e_a$ and $\eta_{\mathbf{1},1}^{a^*a} = u_a c_a$ for some $u_a \neq 0$, therefore
\begin{center}
\raisebox{2cm}{$F_{aa^*a}^a \left[ \scalebox{.6}{$%
$} \right] = u_a \neq 0 \]
which establishes the constraint in (\ref{eqn:duality}).

Let $a,b,c,d,u \in L$.  The pentagonal diagram (\ref{eqn:pent_diag}) reads,
\begin{equation} 
\label{eq:penta_comm}
\raisebox{.5cm}{\scalebox{.6}{%
}}
\nonumber
\end{align}
Start with the LHS of (\ref{eqn:pent}).  By naturality of $\a$
\[
\xymatrix{
(a \otimes b) \otimes (c \otimes d) \ar[rr]^{\a_{ab,c,d} }\ar[d]_{\e_{ab}^{e,i} \otimes 1_c \otimes 1_d} && 
((a \otimes b) \otimes c) \otimes d \ar[d]^{\e_{ab}^{e,i} \otimes 1_c \otimes 1_d} \\
e \otimes (c \otimes d) \ar[rr]_{\a_{e,c,d}} && (e \otimes c) \otimes d
}
\]
that is 
\begin{center}
\scalebox{.6}{%
}
\end{center}
By definition of $F_{ecd}^u$, the left hand side of (\ref{eqn:pent}) is equal to
\begin{eqnarray*} 
\sum_{g \in L} \sum_{l=1}^{N_{eg}^u} \sum_{m=1}^{N_{cd}^{g}} 
F_{ecd}^u \left[ \scalebox{.6}{$%
}
\end{eqnarray*}
By functoriality of the tensor product,
\begin{align*} 
\text{LHS of (\ref{eqn:pent})} &= 
\sum_{g \in L} \sum_{l=1}^{N_{eg}^u} \sum_{m=1}^{N_{cd}^{g}} 
F_{ecd}^u \left[ \scalebox{.6}{$%
}
\end{align*}
Applying naturality of $\a$ again,
\begin{center}
\scalebox{.6}{%
}
\end{center}
And so we get
\begin{align*} 
\text{LHS of (\ref{eqn:pent})} &= 
\sum_{g \in L} \sum_{l=1}^{N_{eg}^u} \sum_{m=1}^{N_{cd}^{g}} 
F_{ecd}^u \left[ \scalebox{.6}{$%
}
\end{align*}
By definition of $F_{abg}^u$,
\begin{eqnarray*} 
\text{LHS of (\ref{eqn:pent})} & = & \\
&& \hspace{-1in} 
\sum_{g \in L} \sum_{l=1}^{N_{eg}^u} \sum_{m=1}^{N_{cd}^g} 
F_{ecd}^u \left[ \scalebox{.6}{$%
}
\end{eqnarray*}
Now to the RHS of (\ref{eqn:pent}):
\begin{align*}
\allowdisplaybreaks
\text{RHS of (\ref{eqn:pent})} = && \nonumber  \\
& & \hspace{-1in} \raisebox{2cm}{$\e_{fd}^{u,k} \circ (\e_{ec}^{f,j} \otimes 1_d) \circ (\e_{ab}^{e,i} \otimes 1_c \otimes 1_d) \circ (\a_{a,b,c} \otimes 1_d) \circ \a_{a,bc,d} \circ (1_a \otimes \a_{b,c,d})=$}
\scalebox{.6}{%
}
 \nonumber \\
&& \hspace{-1in} = \e_{fd}^{u,k} \circ ((\e_{ec}^{f,j} \circ (\e_{ab}^{e,i} \otimes 1_c) \circ \a_{a,b,c}) \otimes 1_d) \circ \a_{a,bc,d} \circ (1_a \otimes \a_{b,c,d}).
 \nonumber 
\end{align*}
By definition of $F_{abc}^f$,
\begin{align*}
\text{RHS of (\ref{eqn:pent})} =& \\
&\hspace{-1in} \e_{fd}^{u,k} \circ 
\sum_{q \in L} \sum_{p=1}^{N_{aq}^f} \sum_{r=1}^{N_{bc}^q} F_{abc}^f 
\left[ \scalebox{.6}{$%
}&&
\end{align*}
Applying naturality of $\a$
\begin{eqnarray} 
\text{RHS of (\ref{eqn:pent})} = && \nonumber  \\
&& \hspace{-1in} 
\sum_{q \in L} \sum_{p=1}^{N_{aq}^f} \sum_{r=1}^{N_{bc}^q} F_{abc}^f 
\left[ \scalebox{.6}{$%
}
\nonumber
\end{eqnarray}
By definition of $F_{aqd}^f$,
\begin{eqnarray} 
\text{RHS of (\ref{eqn:pent})} = && \nonumber  \\
&& \hspace{-1in} = 
\sum_{q \in L} \sum_{p=1}^{N_{aq}^f} \sum_{r=1}^{N_{bc}^q} F_{abc}^f 
\left[ \scalebox{.6}{$%
}
\nonumber
\end{eqnarray}
By definition of $F_{bcd}^t$
\begin{eqnarray} 
\text{RHS of (\ref{eqn:pent})} = && \nonumber  \\
&& \hspace{-1in} =
\sum_{q \in L} \sum_{p=1}^{N_{aq}^f} \sum_{r=1}^{N_{bc}^q} F_{abc}^f 
\left[ \scalebox{.6}{$\begin{array}{ccc} i & e & j \\ p & q & r \end{array}$} \right] 
\sum_{t \in L} \sum_{s=1}^{N_{at}^u} \sum_{v=1}^{N_{qd}^t} F_{aqd}^u
\left[ \scalebox{.6}{$\begin{array}{ccc} p & f & k \\ s & t & v \end{array}$} \right]
\sum_{x \in L} \sum_{w=1}^{N_{bx}^t} \sum_{y=1}^{N_{cd}^x} F_{bcd}^t 
\left[ \scalebox{.6}{$\begin{array}{ccc} r & q & v \\ w & x & y \end{array}$} \right] 
\nonumber \\
&& \e_{at}^{u,s} \circ ( 1_a \otimes \e_{bx}^{t,w} ) \circ (1_a \otimes 1_b \otimes \e_{cd}^{x,y}) \nonumber
\end{eqnarray}
We comper the LHS and RHS of (\ref{eqn:pent}).  Equating and fixing 
\[ h=t \in L \quad,\quad g=x \in L\] 
\[ n=s \in \{1,\ldots,N_{ah}^u=N_{at}^u\} \] 
\[ o=w \in \{1,\ldots,N_{bg}^h=N_{bx}^t\} \]
\[ m=y \in \{1,\ldots,N_{cd}^g=N_{cd}^x\} \]
we end up with
\begin{eqnarray}
\label{eq:penta_eqn}
\sum_{l=1}^{N_{eg}^u} 
F_{ecd}^u 
\left[ \scalebox{.6}{$\begin{array}{ccc} j & f & k \\ l & g & m \end{array}$} \right] 
F_{abg}^u 
\left[ \scalebox{.6}{$\begin{array}{ccc} i & e & l \\ n & h & o \end{array}$} \right] 
= && \\ && \hspace{-1in}
\sum_{q \in L} \sum_{p=1}^{N_{aq}^f} \sum_{r=1}^{N_{bc}^q} \sum_{v=1}^{N_{qd}^h}
F_{abc}^f 
\left[ \scalebox{.6}{$\begin{array}{ccc} i & e & j \\ p & q & r \end{array}$} \right] 
 F_{aqd}^u
\left[ \scalebox{.6}{$\begin{array}{ccc} p & f & k \\ n & h & v \end{array}$} \right]
F_{bcd}^h
\left[ \scalebox{.6}{$\begin{array}{ccc} r & q & v \\ o & g & m \end{array}$} \right] 
\nonumber
\end{eqnarray}
which establishes the constraint in (\ref{eqn:pentagon}).  This concludes the proof.
\end{proof} 

\begin{remark} Since the commutativity of one triangular diagram implies the commutativity of the two other triangular diagrams (see \S\ref{ss:dual}) equations \eqref{eqn:triangle} and  \eqref{eqn:pentagon} imply the two other triangle identities: $F_{\mathbf{1}ab}^u = I_{N_{ab}^u}$ and $F_{ab\mathbf{1}}^u = I_{N_{ab}^u}$. \label{rem:triangle} \end{remark}

A fusion system $(L,N,F)(\cC)$, extracted from a fusion category $\cC$, has an additional symmetry.  Let $G_{abc}^u$ denote the matrix representing $(\a^{-1}_{a,b,c})^*$ in the chosen gauge. 
By definition,
\[ G_{abc}^u \left[ \scalebox{.6}{$\begin{array}{ccc} i & e & j \\ i' & e' & j' \end{array}$} \right] 1_u = 
\eta_{ae}^{u,i} \circ (1_a \otimes \eta_{bc}^{e,j}) \circ \a^{-1}_{a,b,c} \circ (\eta_{e',i'}^{ab} \otimes 1_c) \circ \eta_{u,j'}^{e'c} \]
Let us define
\[ v_a \defeq G_{aa^*a}^{a} \left[ \scalebox{.6}{$\begin{array}{ccc} 1 & \mathbf{1} & 1 \\ 1 & \mathbf{1} & 1 \end{array}$} \right] \]

\begin{lemma}
$ v_{a^*} = u_a$.
\end{lemma}

\begin{proof}
By definition,
\[ G_{a^*aa^*}^{a^*} \left[ \scalebox{.6}{$\begin{array}{ccc} 1 & \mathbf{1} & 1 \\ 1 & \mathbf{1} & 1 \end{array}$} \right] 1_{a^*} = 
\eta_{a^*\mathbf{1}}^{a^*,1} \circ (1_{a^*} \otimes \eta_{aa^*}^{\mathbf{1},1}) \circ \a^{-1}_{a^*,a,a^*} \circ (\eta_{\mathbf{1},1}^{a^*a} \otimes 1_{a^*}) \circ \eta_{a^*,1}^{\mathbf{1}a^*} \]
By our choice of bases $\eta_{a^*\mathbf{1}}^{a^*,1}=\r_{a^*}$, $\eta_{a^*,1}^{\mathbf{1}a^*} = \l_{a^*}^{-1}$, $\eta_{aa^*}^{\mathbf{1},1} = e_a$ and $\eta_{\mathbf{1},1}^{a^*a} = u_a c_a$.
Therefore,
\[ G_{a^*aa^*}^{a^*} \left[ \scalebox{.6}{$\begin{array}{ccc} 1 & \mathbf{1} & 1 \\ 1 & \mathbf{1} & 1 \end{array}$} \right] 1_{a^*} = u_a
\r_{a^*} \circ (1_{a^*} \otimes e_a) \circ \a^{-1}_{a^*,a,a^*} \circ (c_a \otimes 1_{a^*}) \circ \l_{a^*}^{-1} \]
By (RDA-II),
\[ \r_{a^*} \circ (1_{a^*} \otimes e_a) \circ \a^{-1}_{a^*,a,a^*} \circ (c_a \otimes 1_{a^*}) \circ \l_{a^*}^{-1} = 1_{a^*} \]
Therefore,
\[ G_{a^*aa^*}^{a^*} \left[ \scalebox{.6}{$\begin{array}{ccc} 1 & \mathbf{1} & 1 \\ 1 & \mathbf{1} & 1 \end{array}$} \right] = v_{a^*} = u_a = 
F_{aa^*a}^a \left[ \scalebox{.6}{$\begin{array}{ccc} 1 & 1 & 1 \\ 1 & 1 & 1 \end{array}$} \right] \]
\end{proof}

\begin{prop}
\label{prop:fusion}
Given a fusion system $(L,N,F)$ defined over $k$, one can construct from it a fusion category $\cC(L,N,F)$ over $k$, such that $\cC((L,N,F)(\cC)) \simeq \cC $ as fusion categories.
\end{prop}

\begin{proof} 

Let $(L,N,F)$ be a fusion system defined over $k$.  Define a category
\[ \cC:=\prod_{a\in L} (\Vect_k)^{\mb{skel}} \]
where $(\Vect_k)^{\mb{skel}}$ is the skeleton of the category of finite-dimensional $k$-vector spaces.  Alternately, one can view $\cC$ as the skeleton of the category $\Vect_k^L$ of vector bundles over the set $L$.  
\[ Ob(\cC) = \prod_{a\in L} \NN \cup \{0\}= \{f : L \rightarrow \NN \cup \{0\}\} \]
and for every $f,f' \in Ob(\cC)$,
\[ \cC(f,f')=\bigoplus_{a\in L} \Mat_{f'(a)\times f(a) }(k). \]
Composition of morphisms is given by matrix multiplication
\[ \cC(f',f'') \times \cC(f,f') \lra \cC(f,f'') \]
\[ \left( \sum_{a \in L} A'(a) \right) \circ \left( \sum_{a \in L} A(a) \right) = \sum_{a \in L} A'(a) A(a) \]
and the identity: $I = \sum_{a \in L} I_{f(a) \times f(a)}$.

By construction, $\cC$ is a skeletal semi-simple abelian category defined over $k$, with simple objects being the delta functions $\{\d_a\}_{a\in L}$.  To avoid cumbersome notation, we write $a \in \cC$ to mean $\d_a \in \cC$.

\noindent \textbf{Monoidal Structure:}  An additive monoidal bi-functor $\otimes : \cC \times \cC \ra \cC$ is defined on objects $f,g\in\cC$ to be
\[ (f \otimes g) (c) \defeq \sum_{a,b} f(a)g(b)N_{ab}^c \]
On morphisms
\[ \sum_{a \in L} A(a) \otimes \sum_{b \in L} B(b) \defeq \sum_{c \in L} \left( \sum_{a,b} A(a) \otimes B(b) \otimes I_{N_{ab}^c} \right) \]
where $A(a) \in \Mat_{f'(a) \times f(a)}$ and $B(b) \in \Mat_{g'(b) \times g(b)}$.

By the constraints in (\ref{eqn:unit}),
\[ (\d_\mathbf{1} \otimes f) (c) = \sum_{a,b} \d_{\mathbf{1}} (a) f(b) N_{ab}^c = f (b) N_{\mathbf{1} b}^c = f(b) \d_{bc} = f(c) \]
\[ (f \otimes \d_\mathbf{1}) (c) = \sum_{a,b}  f(a) \d_{\mathbf{1}} (b) N_{ab}^c = f (a) N_{a \mathbf{1}}^c = f(a) \d_{ac} = f(c) \]
The neutral object is therefore $\mathbf{1} \equiv \d_{\mathbf{1}}$, and we take $\l_f : \mathbf{1} \otimes f \ra f$ and $\r_f : f \otimes \mathbf{1} \ra f$ both to be the identity, making $\cC$ strictly unital.

We apply Lemma \ref{lem:nonsense} to the case $F=\otimes \circ (1_{\cC} \times \otimes)$ and $G=\otimes \circ (\otimes \times 1_{\cC})$.  For any triple of simple objects $a,b,c \in \cC$, we define an isomorphism
$\a_{a,b,c} : a \otimes (b \otimes c) \ra (a \otimes b) \otimes c$ via
\[ \a_{a,b,c} = \sum_{u \in L} F_{abc}^u \in \bigoplus_{u \in L} \Mat_{N_{abc}^u \times N_{abc}^u}(k) \]
The diagram in Lemma \ref{lem:nonsense} commutes since both $F(f_1,\ldots,f_n)$ and $G(f_1,\ldots,f_n)$ amount to multiplication by a scalar.  Therefore $\a$ extends uniquely to an additive natural isomorphism $\a : F \ra G$.  

By construction, $\cC$ comes with a canonical choice of gauge, i.e. a canonical choice of bases for $(2,1)$-Hom-spaces, otherwise known as {\em fusion state spaces}. For any triple $a,b,c \in L$
\[ \cC(a \otimes b,c) \defeq \sum_{u \in L} \Mat_{c(u) \times (a \otimes b)(u)} (k) = \Mat_{1 \times (a \otimes b)(c)} (k) = \Mat_{1 \times N_{ab}^c} \equiv k^{N_{ab}^c} \]
Denote this canonical basis by $\{ e_{ab}^{c,i} \}_{i=1}^{N_{ab}^c}$.  Tautologically, $F_{abc}^u$ is the matrix representing the linear transformation $\a_{a,b,c}^* : \cC(a \otimes (b \otimes c),u) \ra \cC((a \otimes b) \otimes c,u)$  in the corresponding bases $ \{e_{ec}^{u,j} \circ (e_{ab}^{e,i} \otimes 1_c)\} \subset \cC((a \otimes b) \otimes c,u) $ and $ \{e_{ae'}^{u,i'} \circ (1_a \otimes e_{bc}^{e',j'})\} \subset \cC(a \otimes (b \otimes c),u)$. 

If we now play the proof to Lemma \ref{lem:fusion_extract} backwards, starting at equation 
\eqref{eq:penta_eqn} on page \pageref{eq:penta_eqn}, and retracing our steps all the way back to equation \eqref{eqn:pent} on page \pageref{eqn:pent}, while substituting $e_{ab}^{c,i}$ for $\eta_{ab}^{c,i}$, we end up establishing the commutativity of the pentagonal diagram for $\otimes$ defined in $\cC\equiv\cC(L,N,F)$.

The triangular diagram (\ref{eqn:tri_diag}) is satisfied for each pair of simple objects $a,b \in L$.  Since we chose $\l_b = I_{1 \times 1} = \r_a$, proving the commutativity of the triangular diagram reduces to showing $\a_{a,\mathbf{1},b} = I$, that is, $\a_{a,\mathbf{1},b}(u) = I_{((a\mathbf{1})b)(u) \times (a(\mathbf{1}b))(u)}$.  However, $((a \otimes \mathbf{1}) \otimes b)(u) = N_{ab}^u = (a \otimes (\mathbf{1} \otimes b))(u)$, by definition $\a_{a,\mathbf{1},b}(u) = F_{a \mathbf{1} b}^u$, and by constraint (\ref{eqn:triangle}): $F_{a \mathbf{1} b}^u = I_{N_{ab}^u \times N_{ab}^u}$.

\noindent \textbf{Duals:} Given an object $f \in Ob(\cC)$, define its right dual $f^*$ to be the function $f^* : L \ra \NN \cup \{0\}$ such that $(f^*)(a) = f(a^*)$, in particular, $(\d_a)^*=\d_{a^*}$.   For simple objects $a \equiv \d_a$, the evaluation and co-evaluation are morphisms:
\[ c_a \in \cC(\mathbf{1},a^* \otimes a) = \bigoplus_{u \in L} \Mat_{(a^* a)(u) \times \mathbf{1}(u)} (k) = \Mat_{N_{a^* a}^{\mathbf{1}} \times 1} (k) \equiv k\]
\[ e_a \in \cC(a \otimes a^*, \mathbf{1}) = \bigoplus_{u \in L} \Mat_{\mathbf{1}(u) \times (a a^*)(u)} (k) = \Mat_{1 \times N_{a a^*}^{\mathbf{1}}} (k) \equiv k\]
in other words, these are scalars which we pick to be:
\begin{equation} c_a = 1/u_{a} \quad , \quad e_a = 1 \label{def:coev} \end{equation}
To check the first right duality axiom:
\begin{eqnarray*} (e_a \otimes 1_a) \circ \a_{a,a^*,a} \circ (1_a \otimes c_a) = && \\
&& \hspace{-2in} = \sum_{u \in L} \left( \sum_{(v,w)} e_a(v) \otimes 1_a(w) \otimes I_{N_{vw}^u} \right)
F_{aa^*a}^u
\left( \sum_{(v',w')} 1_a(v') \otimes c_a(w') \otimes I_{N_{v'w'}^u} \right) \end{eqnarray*}
since $1_a(w) = \d_{aw}$, $1_a(v') = \d_{av'}$
\begin{eqnarray*} 
&& = \sum_{u \in L} \left( \sum_{v} e_a(v) \otimes I_{N_{va}^u} \right)
F_{aa^*a}^u \left[ \scalebox{.6}{$\begin{array}{ccc} - & v & - \\ - & w' & - \end{array}$} \right] 
\left( \sum_{w'} c_a(w') \otimes I_{N_{aw'}^u} \right) \end{eqnarray*}
since $e_a(v) = \d_{v \mathbf{1}}$, $c_a(w') = 1/u_a \d_{aw'}$
\begin{eqnarray*} 
&& = \sum_{u \in L} \left( I_{N_{\mathbf{1}a}^u} \right)
F_{aa^*a}^u \left[ \scalebox{.6}{$\begin{array}{ccc} - & \mathbf{1} & - \\ - & \mathbf{1} & - \end{array}$} \right] 
\left( 1/u_a I_{N_{a\mathbf{1}}^u} \right) \end{eqnarray*}
since $N_{\mathbf{1}a}^u = \d_{au} = N_{a\mathbf{1}}^u$, $F_{aa^*a}^a \left[ \scalebox{.6}{$\begin{array}{ccc} - & \mathbf{1} & - \\ - & \mathbf{1} & - \end{array}$} \right]$ is in fact a $1 \times 1$ matrix and
\begin{eqnarray*} 
&& = 1 /u_a
F_{aa^*a}^u \left[ \scalebox{.6}{$\begin{array}{ccc} 1 & \mathbf{1} & 1 \\ 1 & \mathbf{1} & 1 \end{array}$} \right] = 1/u_a \cdot u_a = 1\end{eqnarray*}
thus proving that $(e_a \otimes 1_a) \circ \a_{a,a^*,a} \circ (1_a \otimes c_a) : a \ra a$ is the identity.
By Lemma \ref{lem:RA} this is all we need to establish the existence of right duals in $\cC$ and subsequently the rigidity of $\cC$.

We just constructed a skeletal, semi-simple, abelian category defined over $k$, with a monoidal structure, which is rigid, in other words, $\cC$ is a fusion category.  Denote this fusion category $\cC(L,N,F) \equiv \cC$.

\noindent \textbf{Equivalence:} Let $\cC$ be a fusion category defined over $k$, and let $(L,N,F)(\cC)$ be a fusion system extracted from it.  Recall the fusion system $(L,N,F)(\cC)$ depends on a choice of basis $\{\e_{ab}^{u,i}\}_{i=1}^{N_{ab}^u}$ for each morphism space $\cC(a \otimes b,u)$ (a choice restricted by $\e_{a\mathbf{1}}^{a,1} = \r_a$, $\e_{\mathbf{1}a}^{a,1} = \l_a$ and $\e_{aa^*}^{\mathbf{1},1} = e_a$).  To prove equivalence we define a monoidal functor $F = (F_0,F_1,F_2) : \cC(L,N,F)(\cC) \ra \cC$.  
\[ F_1 : f  \in Ob \, \cC(L,N,F)(\cC) \mapsto \bigoplus_{a \in L} a^{\oplus f(a)} \in Ob \, \cC \]
the isomorphism which holds in any $k$-linear abelian category (given our finiteness assumptions \S\ref{ss:lac})
\[ \bigoplus_{a \in L} \Mat_{g(a) \times f(a)}(k) \xra{\ \cong \ } 
{\cC} \left( \bigoplus_{a \in L} a^{\oplus f(a)},  
\bigoplus_{a \in L} a^{\oplus g(a)} \right) \]
defines $F_1$ on morphisms
\[ F_1 : \cC((L,N,F,R,\eps)(\cC))(f,g) \longmapsto \cC(F_1(f),F_1(g)) \] 
Choice of bases $\{ \e_{ab}^{u,i} \} \subset \cC(a \otimes b,u)$ induces a choice of isomorphisms
\begin{eqnarray*}
&& F_2(a,b) : F_1(a) \otimes F_1(b) \xra{\ \cong \ } F_1(a \otimes b) \defeq \bigoplus_{u \in L} u^{\oplus N_{ab}^u} \\
&& F_2(a,b) = \sum_{u \in L} \sum_{i=1}^{N_{ab}^u} \e_{ab}^{u,i} \end{eqnarray*}
which extends linearly for generic objects.  And we choose $F_0 : F_1(\mathbf{1}) \ra \mathbf{1}$ to be $F_0=1_\mathbf{1}$.  To conclude, we need to consider the commutativity of the following diagrams
\[
\xymatrix{
F_1(\mathbf{1}) \otimes F_1(a) \ar[r]^{F_0 \otimes 1} \ar[d]_{F_2(\mathbf{1},a)} &
\mathbf{1} \otimes F_1(a) \ar[d]^{\l_{F_1(a)}} \\
F_1(\mathbf{1}\otimes a) \ar[r]_{F_1(\l_a)} & F_1(a)
}
\xymatrix{
F_1(a)  \otimes F_1(\mathbf{1}) \ar[r]^{1 \otimes F_0} \ar[d]_{F_2(a,\mathbf{1})} &
F_1(a) \otimes \mathbf{1} \ar[d]^{\r_{F_1(a)}} \\
F_1(a \otimes \mathbf{1}) \ar[r]_{F_1(\r_a)} & F_1(a)
}
\]
\[
\xymatrix{
F_1(a) \otimes (F_1(b) \otimes F_1(c)) \ar[rr]^{\a_{F_1(a),F_1(b),F_1(c)}} 
\ar[d]_{1_{F_1(a)} \otimes F_2(b,c)} &&
(F_1(a) \otimes F_1(b)) \otimes F_1(c) 
\ar[d]^{F_2(a,b) \otimes 1_{F_1(c)}} \\
F_1(a) \otimes F_1(b \otimes c) \ar[d]_{F_2(a,b \otimes c)} && 
F_1(a \otimes b) \otimes F_1(c) \ar[d]^{F_2(a \otimes b, c)} \\
F_1(a \otimes (b \otimes c)) \ar[rr]_{F_1(\a_{a,b,c})} && F_1((a \otimes b) \otimes c)
}
\]
By definition, diagrams equal
\[
\xymatrix{
\mathbf{1} \otimes a \ar[r]^{1 \otimes 1} \ar[d]_{F_2(\mathbf{1},a)} &
\mathbf{1} \otimes a \ar[d]^{\l_a} \\
a \ar[r]_{F_1(1) = 1} & a
} \quad \quad
\xymatrix{
a  \otimes \mathbf{1} \ar[r]^{1 \otimes 1} \ar[d]_{F_2(a,\mathbf{1})} &
a \otimes \mathbf{1} \ar[d]^{\r_a} \\
a \ar[r]_{F_1(1) = 1} & a
}
\]
\[
\xymatrix{
a \otimes (b \otimes c) \ar[rr]^{\a_{a,b,c}} 
\ar[d]_{\sum_{u,i,j} \e_{aw}^{u,i} \circ (1_a \otimes \e_{bc}^{w,j})} &&
(a \otimes b) \otimes c 
\ar[d]^{\sum_{u,i',j'}  \e_{w'c}^{u,j'} \circ (\e_{ab}^{w',i'} \otimes 1_c)}  \\
\sum_{u \in L} u^{\oplus N_{abc}^u} \ar[rr]_{\sum_{u \in L} F_{abc}^u} && 
\sum_{u \in L} u^{\oplus N_{abc}^u}
}
\]
Top two diagrams commute by
\[ F_2(\mathbf{1},a) = \sum_{u \in L} \sum_i \e_{\mathbf{1} a}^{u,i} = \e_{\mathbf{1} a}^{a,1} = \l_a \]
\[ F_2(a,\mathbf{1}) = \sum_{u \in L} \sum_i \e_{a \mathbf{1}}^{u,i} = \e_{a \mathbf{1}}^{a,1} = \r_a \]
Bottom diagram commutes by definition of $F_{abc}^u$.  Hence $F$ is a monoidal functor which induces an isomorphism on skeletons.  Hence $\cC((L,N,F,R,\eps)(\cC)) \simeq \cC$ as fusion categories.
\end{proof} 

\begin{remark}
In his paper \cite{Y}, Yamagami introduces the notion of a monoidal system.  Our fusion system is a numerical version of that.  Yamagami goes on to prove a reconstruction theorem similar to Proposition \ref{prop:fusion}.  Both proofs follow similar lines.  
\end{remark}

\begin{remark}
In general, it is not true that every monoidal category is equivalent to a skeleton which is strictly unital.  We just proved that for fusion categories this is the case.
\end{remark}

\subsection{Choice of Gauge.}

\label{ss:gauge}

The extraction of a fusion system $(L,N,F)(\cC)$ from a fusion category $\cC$ depended on a choice of gauge, namely, a choice of bases $\{ \eta_{ab}^{c,i} \}$ for $(2,1)$-Hom-spaces $\cC(a \otimes b,c)$, otherwise known as {\em fusion state spaces}.  That choice was restricted so that
\begin{eqnarray*}
\eta_{\mathbf{1}a}^{a,1} & \defeq & \l_a \in \cC(\mathbf{1} \otimes a,a) \\
\eta_{a\mathbf{1}}^{a,1} & \defeq & \r_a \in \cC(a \otimes \mathbf{1},a) \\
\eta_{aa^*}^{\mathbf{1},1} & \defeq & e_a \in \cC(a \otimes a^*, \mathbf{1})
\end{eqnarray*}
Recall from \S\ref{ss:lac} we let $\{\e_{c,j}^{ab}\} \subset \cC(c,a \otimes b)$ be the dual basis to $\{ \e_{ab}^{c,i} \} \subset \cC(a \otimes b,c)$ with respect to the perfect pairing given by composition:
\[ \cC(c,a \otimes b) \times \cC(a \otimes b,c) \lra \cC(c,c) \equiv k \]

By construction, $\cC(L,N,F)$ comes with a canonical choice of gauge. For any triple $a,b,c \in L$
\[ \cC(L,N,F)(a \otimes b,c) := \sum_{u \in L} \Mat_{c(u) \times (a \otimes b)(u)} (k) = \Mat_{1 \times (a \otimes b)(c)} (k) = \Mat_{1 \times N_{ab}^c} \equiv k^{N_{ab}^c} \]
For future reference we denote this canonical choice by $\{ e_{ab}^{c,i} \} \subset \cC(L,N,F)(a \otimes b,c)$, to distinguish it from a non-canonical choice of gauge $\{ \eta_{ab}^{c,i} \} \subset \cC(a \otimes b,c)$ for a general fusion category $\cC$.  By construction,
\begin{eqnarray*}
\l_a & \defeq & e_{\mathbf{1}a}^{a,1}  \in \cC(L,N,F)(\mathbf{1} \otimes a,a) \\
 \r_a & \defeq & e_{a\mathbf{1}}^{a,1} \in \cC(L,N,F)(a \otimes \mathbf{1},a) \\
e_a & \defeq & e_{ab}^{\mathbf{1},1}  \in \cC(L,N,F)(a \otimes b, \mathbf{1}) \quad , \quad b=a^* \in L \\
u_a \cdot c_a & \defeq & e_{\mathbf{1},1}^{ba} \in \cC(L,N,F)(\mathbf{1},b \otimes a) \quad , \quad b=a^* \in L
\end{eqnarray*}

\subsection{Modular System.} 

\label{ss:modular}

\begin{definition}
\label{def:modular}
A \emph{modular system} $(L,N,F,R,\eps)$ over $k$ consists of the following:
\begin{list}{(\roman{listi})}{\usecounter{listi}\setlength{\leftmargin}{0ex}\setlength{\labelsep}{0ex}}
\item \ \ A fusion system $(L,N,F)$ over $k$ such that for every triple $a,b,c \in L$
\begin{equation} N_{ab}^c = N_{ba}^c \label{eq:comm} \end{equation}
\item \ \ A collection of invertible square matrices $R_{ab}^c$ of dimension $N_{ab}^c$ defined over $k$ with inverses $Q_{ab}^c$, associated with every triplet $a,b,c \in L$ subject to the constraints: for every $d,d' \in L$, $i=1,\ldots,N_{ca}^d$, $j=1,\ldots,N_{db}^u$, $i'=1,\ldots,N_{ad'}^u$ and $j'=1,\ldots,N_{bc}^{d'}$
\begin{eqnarray}
\sum_{x=1}^{N_{ac}^d} 
\sum_{y=1}^{N_{bc}^{d'}}
R_{ac}^d \left[ \scalebox{.6}{ $\begin{array}{c} i \\ x \end{array}$ } \right]
F_{acb}^u \left[ \scalebox{.6}{$\begin{array}{ccc} x & d & j \\ i' & d' & y \end{array}$} \right] 
R_{bc}^{d'} \left[ \scalebox{.6}{ $\begin{array}{c} y \\ j' \end{array}$ } \right]
& = & \label{eq:hex16} \\  
& \hspace{-7.8cm} = & \hspace{-4cm}
\sum_{e \in  L} 
\sum_{u=1}^{N_{ce}^u} 
\sum_{v=1}^{N_{ab}^e} 
\sum_{w=1}^{N_{ec}^u}
F_{cab}^u \left[ \scalebox{.6}{$\begin{array}{ccc} i & d & j \\ u & e & v \end{array}$} \right]
R_{ec}^u \left[ \scalebox{.6}{ $\begin{array}{c} u \\ w \end{array}$ } \right]
F_{abc}^u \left[ \scalebox{.6}{$\begin{array}{ccc} v & e & w \\ i' & d' & j' \end{array}$} \right]
\nonumber
\label{eq:hexa}
\end{eqnarray}
\begin{eqnarray}
\sum_{x=1}^{N_{ac}^d} 
\sum_{y=1}^{N_{bc}^{d'}}
Q_{ac}^d \left[ \scalebox{.6}{ $\begin{array}{c} i \\ x \end{array}$ } \right]
F_{acb}^u \left[ \scalebox{.6}{$\begin{array}{ccc} x & d & j \\ i' & d' & y \end{array}$} \right] 
Q_{bc}^{d'} \left[ \scalebox{.6}{ $\begin{array}{c} y \\ j' \end{array}$ } \right]
& = & \label{eq:hex16} \\  
& \hspace{-7.8cm} = & \hspace{-4cm}
\sum_{e \in  L} 
\sum_{u=1}^{N_{ce}^u} 
\sum_{v=1}^{N_{ab}^e} 
\sum_{w=1}^{N_{ec}^u}
F_{cab}^u \left[ \scalebox{.6}{$\begin{array}{ccc} i & d & j \\ u & e & v \end{array}$} \right]
Q_{ec}^u \left[ \scalebox{.6}{ $\begin{array}{c} u \\ w \end{array}$ } \right]
F_{abc}^u \left[ \scalebox{.6}{$\begin{array}{ccc} v & e & w \\ i' & d' & j' \end{array}$} \right]
\nonumber
\label{eq:hexb}
\end{eqnarray}
\item \ \ A collection $\eps_a = \pm 1$ for every $a \in L$, subject to the constraint
\begin{eqnarray}
&& \label{eq:piv} \hspace{-1cm} 
\sum_{s =1}^{N_{bc^*}^{a^*}} \sum_{t = 1}^{N_{c^*a}^{b^*}} 
F_{abc^*}^{\mathbf{1}} \left[ \scalebox{.6}{$ \begin{array}{ccc} i & c & 1 \\ 1 & a^* & s \end{array} $} \right] 
F_{bc^*a}^{\mathbf{1}} \left[ \scalebox{.6}{$ \begin{array}{ccc} s & a^* & 1 \\ 1 & b^* & t \end{array} $} \right] 
F_{c^*ab}^{\mathbf{1}} \left[ \scalebox{.6}{$ \begin{array}{ccc} t & b^* & 1 \\ 1 & c & i \end{array} $} \right] 
= \eps_c^{-1} \eps_a \eps_b 
\end{eqnarray}
for every $i \in \{1,\ldots,N_{ab}^c\}$.
\item  \ \ The matrix $\hat{S}$ of dimension $|L| \times |L|$ defined over $k$, whose entries are given by 
\begin{equation} 
\hat{S}_{ab} = \sum_{c \in L} 
\sum_{i=1}^{N_{ab^*}^c} 
\sum_{j=1}^{N_{cb}^{a}}
\sum_{i'=1}^{N_{b^*a}^c}
\sum_{i''=1}^{N_{ab^*}^c}
G_{ab^*b}^{a} \left[ \scalebox{.6}{$ \begin{array}{ccc} 1 & \mathbf{1} & 1 \\ i & c & j \end{array} $} \right]
R_{b^*a}^c \left[ \scalebox{.6}{$ \begin{array}{c} i \\ i' \end{array} $} \right]
R_{ab^*}^c \left[ \scalebox{.6}{$ \begin{array}{c} i' \\ i'' \end{array} $} \right]
F_{ab^*b}^{a} \left[ \scalebox{.6}{$ \begin{array}{ccc} i'' & c & j \\ 1 & \mathbf{1} & 1 \end{array} $} \right] \label{eq:modularity}
\end{equation}
is invertible.
\end{list}
\end{definition}

\begin{remark} As a consequence of the Hexagon equations, we have
\[ R_{a \mathbf{1}}^a \left[ \scalebox{.6}{$ \begin{array}{c} 1 \\ 1 \end{array} $} \right] = 1 = 
R_{\mathbf{1} a}^a \left[ \scalebox{.6}{$ \begin{array}{c} 1 \\ 1 \end{array} $} \right] 
\quad , \quad
Q_{a \mathbf{1}}^a \left[ \scalebox{.6}{$ \begin{array}{c} 1 \\ 1 \end{array} $} \right] = 1 = 
Q_{\mathbf{1} a}^a \left[ \scalebox{.6}{$ \begin{array}{c} 1 \\ 1 \end{array} $} \right] \]
\label{rem:R_trivial}
\end{remark}

\begin{prop}
\label{prop:modular}
\begin{mylist}
\item Given a modular category $\cC$ defined over $k$, one can extract from it a modular system $(L,N,F,R,\eps)(\cC)$ defined over $k$.
\item Given a modular system $(L,N,F,R,\eps)$ defined over $k$, one can construct from it a modular category $\cC(L,N,F,R,\eps)$ defined over $k' = k(\sqrt{u_a})$.  
\item Moreover, given a modular category $\cC$ defined over $k$, we have $\cC((L,N,F,R,\eps)(\cC)) \simeq \cC \otimes_k k'$.
\end{mylist}
\end{prop}

\begin{remark}
The extension $k' = k(\sqrt{u_a})$ facilitates our construction of the modular category $\cC(L,N,F,R,\eps)$.  At this point, we do not know of a proof that such an extension is necessary, only sufficient.  
\end{remark}

\begin{proof}
\begin{mylist}
\item Let $\cC$ be a modular category defined over $k$.  Let $L$ be a set of representatives of isomorphism classes of simple objects in $\cC$.  For every triple $a,b,c \in L$, choose a basis $\{\eta_{ab}^{c,i}\}_{i=1}^{N_{ab}^c}$ for the morphism $k$-space $\cC(a \otimes b,c)$.

Such a choice is restricted so that
\begin{eqnarray*}
\eta_{\mathbf{1}a}^{a,1} & \defeq & \l_a \in \cC(\mathbf{1} \otimes a,a) \\
\eta_{a\mathbf{1}}^{a,1} & \defeq & \r_a \in \cC(a \otimes \mathbf{1},a) \\
\eta_{aa^*}^{\mathbf{1},1} & \defeq & e_a \in \cC(a \otimes a^*, \mathbf{1}) \\
\eta_{a^*a}^{\mathbf{1},1} & \defeq & e_{a^*} \circ (1_{a^*} \otimes \eps_a'^{-1}) \in \cC(a^* \otimes a, \mathbf{1})
\end{eqnarray*}
where $\eps'$ designates a given pivotal structure in $\cC$.

Recall from Lemma \ref{lem:fusion_extract} that such a choice of bases facilitates the extraction of a fusion system $(L,N,F)(\cC)$ from $\cC$.  We show $(L,N,F)(\cC)$ can be completed to a modular system defined over $k$.  

First, note that since $\cC$ is braided, $a \otimes b \cong b \otimes a$, which implies $N_{ab}^c = \dim_k \cC(a \otimes b,c) = \dim_k \cC(b \otimes a,c) = N_{ba}^c$, thus satisfying constraint \eqref{eq:comm} in Definition \ref{def:modular}.

The braiding and its inverse in $\cC$
\[ \b_{a,b} : a \otimes b \ra b \otimes a \] 
\[ \b_{b,a}^{-1} : a \otimes b \ra b \otimes a \] 
define pullbacks
\begin{eqnarray*}
&& \b^*_{a,b} : \cC(b \otimes a, c) \ra \cC(a \otimes b, c) \\
&& (\b_{b,a}^{-1})^* : \cC(b \otimes a, c) \ra \cC(a \otimes b, c)
\end{eqnarray*}
for every $c\in L$.  We define $R_{ab}^c$ to be the matrix representing $\b^*_{a,b}$, and $Q_{ab}^c$ to be the matrix representing $(\b_{b,a}^{-1})^*$, in the above respective bases, namely,
\begin{eqnarray}
&& \eta_{ba}^{c,i} \circ \b_{a,b} = \sum_{i'=1}^{N_{ab}^c} R_{ab}^c
\left[ \scalebox{.6}{$\begin{array}{c} i \\ i' \end{array}$} \right] \eta_{ab}^{c,i'} \label{eq:R} \\
&& \eta_{ba}^{c,i} \circ \b_{b,a}^{-1} = \sum_{i'=1}^{N_{ab}^c} Q_{ab}^c \left[ \scalebox{.6}{ $\begin{array}{c} i \\ i' \end{array}$ } \right] \eta_{ab}^{c,i'} \label{eq:Q}
\end{eqnarray}
The coefficients in (\ref{eq:R}) are known as the {\em $R$-matrix coefficients}.  

Showing the collection $\{R_{ab}^c\}$ satisfies constraint (\ref{eq:hexa}), and the collection $\{Q_{ab}^c\}$ satisfies constraint (\ref{eq:hexb}) in Definition \ref{def:modular} follows the exact same argument in reverse as the one used in part (ii) of the proof.  We therefore refer the reader to equations \eqref{eq:hex11} -- \eqref{eq:hex16} on pages \pageref{eq:hex11} -- \pageref{eq:hex16}.

As a ribbon fusion category, $\cC$ has a spherical pivotal structure $\eps' : \Delta = (\Delta_0,\Delta_1,\Delta_2) \ra 1_\cC$, a braiding $\b$, and a balancing $\t$.  For every object $x \in \cC$, define an isomorphism $\psi_x : x^{**} \ra x$ as in \eqref{def:psi}.  Define $\eps_x'' : x \ra x$ to be the isomorphism
\[ \eps_x'' \defeq \t_x^{-1} \circ \psi_x \]
Now define $\eps_x : x \ra x$ to be 
\[ \eps_x \defeq \eps_x'' \circ (\eps_x')^{-1} \]
Consider $\eps_a$ for every representative $a \in L$ of an isomorphism class of simple objects in $\cC$.  By abuse of notation, think of $\eps_a \in \cC(a,a) \equiv k$ also as a scalar in $k$.

As discussed in \S\ref{ss:balance}, $\eps_x''$ constitutes a spherical pivotal structure on $\cC$. This means $\eps_x$, as defined above, measures the possible discrepancy between $\eps'$, given a-priori, and $\eps''$ defined via the braiding and balancing.  It turns out this discrepancy amounts to at most a multiplication by $ \pm 1$.

Let us show the collection $\{\eps_a\}_{a \in L}$ satisfies condition (\ref{eq:piv}).  We know $\eps_x'' = \eps_x \circ \eps_x'$ is monoidal, that is,
\[ \eps_{a \otimes b}'' \circ \Delta_2(a,b) \circ (\eps_a'')^{-1} \otimes (\eps_b'')^{-1} = 1_{a \otimes b} \]
For every $\eta_{ab}^{c,i} \in \cC(a \otimes b,c)$
\[ \eta_{ab}^{c,i} \circ \eps_{a \otimes b}'' \circ \Delta_2(a,b) \circ (\eps_a'')^{-1} \otimes (\eps_b'')^{-1} = \eta_{ab}^{c,i} \]
Because $\eps''$ is natural
\[\eps_{c}''  \circ \Delta_1(\eta_{ab}^{c,i}) \circ \Delta_2(a,b) \circ (\eps_a'')^{-1} \otimes (\eps_b'')^{-1} = \eta_{ab}^{c,i} \]
Using Lemma \ref{lem:g_in_Delta} we have
\[\eps_{c}''  \circ \Delta(a,b,c,\eta_{ab}^{c,i}) \circ (\eps_a'')^{-1} \otimes (\eps_b'')^{-1} = \eta_{ab}^{c,i} \]
Using definition of $\eps$
\begin{equation}
\eps_{c} \circ \eps_c' \circ \Delta(a,b,c,\eta_{ab}^{c,i}) \circ (\eps_a')^{-1} \otimes (\eps_b')^{-1} \circ  (\eps_a)^{-1} \otimes (\eps_b)^{-1} = \eta_{ab}^{c,i}
\label{eqn:pivotal_generic}
\end{equation}
Let's compute $\eps_c' \circ \Delta(a,b,c,\eta_{ab}^{c,i}) \circ (\eps_a')^{-1} \otimes (\eps_b')^{-1}$.  Note that, by our conventions, 
\[ \eta_{a^*a}^{\mathbf{1},1} = e_{a^*} \circ (1_{a^*} \otimes \eps_a'^{-1}). \]
and in graphical notation,

\begin{center}
\scalebox{.6}{
}
\end{center}

This is because if $(a^*,a^{**},e_{a^*},c_{a^*})$ is a chosen right duality data for $a^*$ then so is $(a^*,a,e_{a^*} \circ (1_{a^*} \otimes \eps_a'^{-1}),(\eps_a' \otimes 1) \circ c_{a^*})$, and if $b$ is right dual to $a$ we take $\eta_{ab}^{\mathbf{1},1}$ to always stand for the appropriate evaluation morphism.

We have
\[
 \raisebox{3cm}{$\eps_c' \circ \Delta(a,b,c,e_{ab}^{c,i}) \circ \eps_a'^{-1} \otimes \eps_b'^{-1} =$ \label{eq:piv_calc}}
\scalebox{.6}{%
}
\]
Applying the definition of $F_{abc^*}^{\mathbf{1}}$ gives
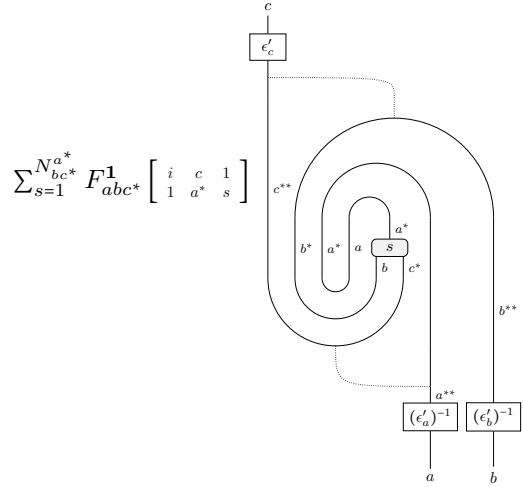
\[
\raisebox{4cm}{$
\sum_{s =1}^{N_{bc^*}^{a^*}} F_{abc^*}^{\mathbf{1}} \left[ \scalebox{.6}{$ %
}
\]
From the second right duality axiom one obtains
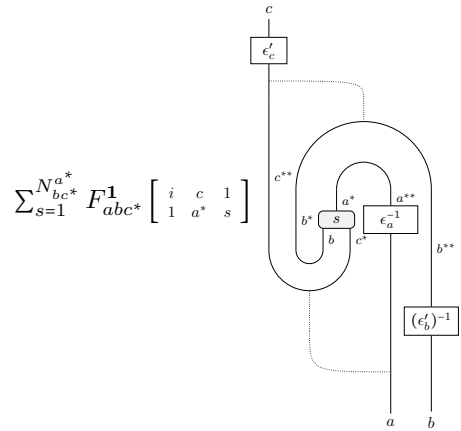
\[
\raisebox{3cm}{
$\sum_{s =1}^{N_{bc^*}^{a^*}} F_{abc^*}^{\mathbf{1}} \left[ \scalebox{.6}{$ %
}
\]
Due to our conventions of choice of basis in $\cC(a^* \otimes a, \mathbf{1})$, this is equal to
\[
\raisebox{4cm}{
$\sum_{s =1}^{N_{bc^*}^{a^*}} F_{abc^*}^{\mathbf{1}} \left[ \scalebox{.6}{$ %
}
\]
The definition of $F_{bc^*a}^{\mathbf{1}}$ gives
\[
\raisebox{3cm}{
$\sum_{s =1}^{N_{bc^*}^{a^*}} F_{abc^*}^{\mathbf{1}} \left[ \scalebox{.6}{$ %
}
\]
Applying the second right duality axiom, one obtains
\[
\raisebox{2cm}{
$\sum_{s =1}^{N_{bc^*}^{a^*}} F_{abc^*}^{\mathbf{1}} \left[ \scalebox{.6}{$ %
}
\]
Choice of basis in $\cC(b^* \otimes b,\mathbf{1})$ makes this equal to
\[
\raisebox{2cm}{
$\sum_{s =1}^{N_{bc^*}^{a^*}} F_{abc^*}^{\mathbf{1}} \left[ \scalebox{.6}{$ %
}
\]
By definition of $F_{c^*ab}^{\mathbf{1}}$, one obtains
\[
\raisebox{2cm}{
$\sum_{s =1}^{N_{bc^*}^{a^*}} F_{abc^*}^{\mathbf{1}} \left[ \scalebox{.6}{$ %
}
\]
As a result of right duality data $(c^*,c,e_{c^*} \circ (1_{c^*} \otimes \eps_c'^{-1}),(\eps_c' \otimes 1_{c^*}) \circ c_{c^*})$ we may apply right duality axiom, giving
\[
\sum_{s =1}^{N_{bc^*}^{a^*}} F_{abc^*}^{\mathbf{1}} \left[ \scalebox{.6}{$ \begin{array}{ccc} i & c & 1 \\ 1 & a^* & s \end{array} $} \right] 
\sum_{t = 1}^{N_{c^*a}^{b^*}} F_{bc^*a}^{\mathbf{1}} \left[ \scalebox{.6}{$ \begin{array}{ccc} s & a^* & 1 \\ 1 & b^* & t \end{array} $} \right] 
\sum_{j = 1}^{N_{ab}^{c}} F_{c^*ab}^{\mathbf{1}} \left[ \scalebox{.6}{$ \begin{array}{ccc} t & b^* & 1 \\ 1 & c & j \end{array} $} \right] 
\eta_{ab}^{c,j} \]
Back to Equation \eqref{eqn:pivotal_generic} we have
\[ \eps_c \circ
\sum_{s =1}^{N_{bc^*}^{a^*}} F_{abc^*}^{\mathbf{1}} \left[ \scalebox{.6}{$ \begin{array}{ccc} i & c & 1 \\ 1 & a^* & s \end{array} $} \right] 
\sum_{t = 1}^{N_{c^*a}^{b^*}} F_{bc^*a}^{\mathbf{1}} \left[ \scalebox{.6}{$ \begin{array}{ccc} s & a^* & 1 \\ 1 & b^* & t \end{array} $} \right] 
\sum_{j = 1}^{N_{ab}^{c}} F_{c^*ab}^{\mathbf{1}} \left[ \scalebox{.6}{$ \begin{array}{ccc} t & b^* & 1 \\ 1 & c & j \end{array} $} \right] 
\eta_{ab}^{c,j} \circ
\eps_a^{-1} \otimes \eps_b^{-1}= \eta_{ab}^{c,i} \]
Setting $i=j$ we get exactly constraint \eqref{eq:piv} in Definition \ref{def:modular}
\[ 
\sum_{s =1}^{N_{bc^*}^{a^*}} \sum_{t = 1}^{N_{c^*a}^{b^*}} 
F_{abc^*}^{\mathbf{1}} \left[ \scalebox{.6}{$ \begin{array}{ccc} i & c & 1 \\ 1 & a^* & s \end{array} $} \right] 
F_{bc^*a}^{\mathbf{1}} \left[ \scalebox{.6}{$ \begin{array}{ccc} s & a^* & 1 \\ 1 & b^* & t \end{array} $} \right] 
F_{c^*ab}^{\mathbf{1}} \left[ \scalebox{.6}{$ \begin{array}{ccc} t & b^* & 1 \\ 1 & c & i \end{array} $} \right] 
= \eps_c^{-1} \eps_a \eps_b \]

Both $\eps'$ and $\eps''$ are spherical pivotal structures on $\cC$.  Therefore the left quantum dimension $q'_l(x) : \mathbf{1} \ra \mathbf{1} $ ($q''_l(x) : \mathbf{1} \ra \mathbf{1}$) and right quantum dimension $q'_r(x) : \mathbf{1} \ra \mathbf{1} $ ($q''_r(x) : \mathbf{1} \ra \mathbf{1}$) of any object $x \in \cC$ agree, namely, $q'_l(x) = q'_r(x)$ and $q''_l(x) = q''_r(x)$. By definition of $\eps_a$,
\[ q'_l(a) = \eps^{-1}_a q''_l(a) \]
\[ q'_r(a) = \eps_a q''_r(a) \]
This implies $\eps_a = \pm 1$.

Lastly, the invertibility of the $|L|\times|L|$ matrix $\hat{S}$ defined in \eqref{eq:modularity} is tied to the invertibility of the $S$-matrix of $\cC$ via the following relation established on page \pageref{eq:S_calc}:
\[ S = D \hat{S} D \]
Here $D$ denotes the diagonal matrix whose diagonal entries are given by the quantum dimensions of the simple objects of $\cC$.  And since quantum dimensions $\neq 0$, the invertibility of $\hat{S}$ follows.  

\item Let us denote
\[ \cC \defeq \cC(L,N,F) \]
where $\cC(L,N,F)$ is the fusion category constructed in Proposition \ref{prop:fusion} from the fusion system $(L,N,F) \subset (L,N,F,R,\eps)$. Our plan is first to endow $\cC$ with a braiding and pivotal structure.  To show the pivotal structure is spherical we consider a scalar extension $\cC'$ of $\cC$, which is shown to be spherical and as such also tortile.  This makes $\cC'$ into a ribbon category, which we then show to be modular.  It is then $\cC'$ which we define to be $\cC(L,N,F,R,\eps)$.

\noindent \textbf{Braiding:} Lemma \ref{lem:nonsense} applied to $F=\otimes : \cC \times \cC \ra \cC$ and $G=\otimes ^{\text{op}} : \cC \times \cC \ra \cC$ implies that a braiding in $\cC$ is uniquely determined by a set of isomorphisms
\[ \b_{a,b} : a \otimes b \lra b \otimes a \]
for every $a,b \in L$.  Define $\b_{a,b} \in \cC(a \otimes b, b \otimes a)$ to be
\[ \b_{a,b} = \sum_{c \in L} R_{ab}^c \in \bigoplus_{c \in L} \Mat_{b\otimes a(c) \times a \otimes b(c)}(k) =  \bigoplus_{c \in L} \Mat_{N_{ba}^c \times N_{ab}^c}(k) = \bigoplus_{c \in L} \Mat_{N_{ab}^c \times N_{ab}^c}(k)\]

\noindent Equation (\ref{eq:hex1}) is satisfied iff for every $u \in L$ and every basis element $e_{db}^{u,j} \circ e_{ca}^{d,i} \otimes 1_b \in \cC((c \otimes a) \otimes b,u)$ we have
\begin{eqnarray}
e_{db}^{u,j} \circ (e_{ca}^{d,i} \otimes 1_b) \circ (\b_{a,c} \otimes 1_b) \circ \a_{a,c,b} \circ (1_a \otimes \b_{b,c}) & = & \label{eq:hex11} \\
& \hspace{-6cm} = & \hspace{-3cm} e_{db}^{u,j} \circ (e_{ca}^{d,i} \otimes 1_b) \circ \a_{c,a,b} \circ \b_{a \otimes b,c} \circ \a_{a,b,c} \nonumber
\end{eqnarray}

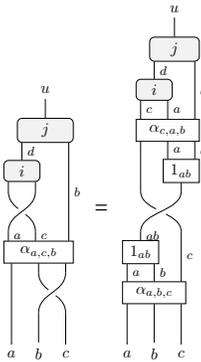
\begin{center}
\scalebox{.6}{\begin{tikzpicture}
\makeatletter{}\draw (1.05,6.3) node [minimum height=0, minimum width=0.4cm] (braiding1u0) {$u$};
\draw (1.05,5.4) node [draw, rounded corners=3pt, fill=gray!10, minimum width=1.2388cm] (braiding1fdbuj0) {$j$};
\draw (0.5306,4.5) node [draw, rounded corners=3pt, fill=gray!10, minimum width=0.8cm] (braiding1fcadi0) {$i$};
\draw (1.5694,4.5) node [minimum height=0] (braiding1idb13) {$$};
\draw (0.5306,3.6) node [] (braiding1brac0) {} (0.8306,3.15) coordinate (braiding1brac0dom2)..controls (0.8306,3.6) and (0.2306,3.6)..(0.2306,4.05) coordinate (braiding1brac0codom1);
\draw[preaction={draw,line width=2mm,white}] (0.2306,3.15) coordinate (braiding1brac0dom1)..controls (0.2306,3.6) and (0.8306,3.6)..(0.8306,4.05) coordinate (braiding1brac0codom2);
\draw (1.5694,3.6) node [minimum height=0] (braiding1idb12) {$$};
\draw (0.9,2.7) node [draw, minimum width=1.5388cm] (braiding1Subscriptalphaacb0) {$\alpha_{a,c,b}$};
\draw (0.3,1.8) node [minimum height=0] (braiding1ida18) {$$};
\draw (1.2,1.8) node [] (braiding1brbc0) {} (1.5,1.35) coordinate (braiding1brbc0dom2)..controls (1.5,1.8) and (0.9,1.8)..(0.9,2.25) coordinate (braiding1brbc0codom1);
\draw[preaction={draw,line width=2mm,white}] (0.9,1.35) coordinate (braiding1brbc0dom1)..controls (0.9,1.8) and (1.5,1.8)..(1.5,2.25) coordinate (braiding1brbc0codom2);
\draw (0.3,0.45) node [minimum height=0, minimum width=0.4cm] (braiding1a11) {$a$};
\draw (0.9,0.45) node [minimum height=0, minimum width=0.4cm] (braiding1b9) {$b$};
\draw (1.5,0.45) node [minimum height=0, minimum width=0.4cm] (braiding1c9) {$c$};
\draw[solid] ([xshift=0.cm]braiding1b9.north) to (braiding1brbc0dom1);
\draw[solid] ([xshift=0.cm]braiding1c9.north) to (braiding1brbc0dom2);
\draw[solid] ([xshift=0.cm]braiding1a11.north) to ([xshift=-0.6cm]braiding1Subscriptalphaacb0.south);
\draw[solid] (braiding1brbc0codom1) to ([xshift=0.cm]braiding1Subscriptalphaacb0.south);
\draw[solid] (braiding1brbc0codom2) to ([xshift=0.6cm]braiding1Subscriptalphaacb0.south);
\draw[solid] ([xshift=-0.6694cm]braiding1Subscriptalphaacb0.north) to node [right,font=\footnotesize] {$a$} (braiding1brac0dom1);
\draw[solid] ([xshift=-0.0694cm]braiding1Subscriptalphaacb0.north) to node [right,font=\footnotesize] {$c$} (braiding1brac0dom2);
\draw[solid] (braiding1brac0codom1) to ([xshift=-0.3cm]braiding1fcadi0.south);
\draw[solid] (braiding1brac0codom2) to ([xshift=0.3cm]braiding1fcadi0.south);
\draw[solid] ([xshift=0.cm]braiding1fcadi0.north) to node [right,font=\footnotesize] {$d$} ([xshift=-0.5194cm]braiding1fdbuj0.south);
\draw[solid] ([xshift=0.6694cm]braiding1Subscriptalphaacb0.north) to node [right,font=\footnotesize] {$b$} ([xshift=0.5194cm]braiding1fdbuj0.south);
\draw[solid] ([xshift=0.cm]braiding1fdbuj0.north) to ([xshift=0.cm]braiding1u0.south);
 
\end{tikzpicture}}
\raisebox{2cm}{$=$}
\scalebox{.6}{\begin{tikzpicture}
\makeatletter{}\draw (1.35,8.1) node [minimum height=0, minimum width=0.4cm] (braiding2u1) {$u$};
\draw (1.35,7.2) node [draw, rounded corners=3pt, fill=gray!10, minimum width=1.1cm] (braiding2fdbuj1) {$j$};
\draw (0.9,6.3) node [draw, rounded corners=3pt, fill=gray!10, minimum width=0.8cm] (braiding2fcadi1) {$i$};
\draw (1.8,6.3) node [minimum height=0] (braiding2idb14) {$$};
\draw (1.2,5.4) node [draw, minimum width=1.4cm] (braiding2Subscriptalphacab0) {$\alpha_{c,a,b}$};
\draw (0.6,4.5) node [minimum height=0] (braiding2idc1) {$$};
\draw (1.5,4.5) node [draw, minimum width=0.8cm] (braiding2Powertensorab-10) {$1_{ab}$};
\draw (1.05,3.6) node [] (braiding2brCircleTimesabc0) {} (1.5,3.15) coordinate (braiding2brCircleTimesabc0dom2)..controls (1.5,3.6) and (0.6,3.6)..(0.6,4.05) coordinate (braiding2brCircleTimesabc0codom1);
\draw[preaction={draw,line width=2mm,white}] (0.6,3.15) coordinate (braiding2brCircleTimesabc0dom1)..controls (0.6,3.6) and (1.5,3.6)..(1.5,4.05) coordinate (braiding2brCircleTimesabc0codom2);
\draw (0.6,2.7) node [draw, minimum width=0.8cm] (braiding2tensorab0) {$1_{ab}$};
\draw (1.5,2.7) node [minimum height=0] (braiding2idc0) {$$};
\draw (0.9,1.8) node [draw, minimum width=1.4cm] (braiding2Subscriptalphaabc0) {$\alpha_{a,b,c}$};
\draw (0.3,0.45) node [minimum height=0, minimum width=0.4cm] (braiding2a12) {$a$};
\draw (0.9,0.45) node [minimum height=0, minimum width=0.4cm] (braiding2b10) {$b$};
\draw (1.5,0.45) node [minimum height=0, minimum width=0.4cm] (braiding2c10) {$c$};
\draw[solid] ([xshift=0.cm]braiding2a12.north) to ([xshift=-0.6cm]braiding2Subscriptalphaabc0.south);
\draw[solid] ([xshift=0.cm]braiding2b10.north) to ([xshift=0.cm]braiding2Subscriptalphaabc0.south);
\draw[solid] ([xshift=0.cm]braiding2c10.north) to ([xshift=0.6cm]braiding2Subscriptalphaabc0.south);
\draw[solid] ([xshift=-0.6cm]braiding2Subscriptalphaabc0.north) to node [right,font=\footnotesize] {$a$} ([xshift=-0.3cm]braiding2tensorab0.south);
\draw[solid] ([xshift=0.cm]braiding2Subscriptalphaabc0.north) to node [right,font=\footnotesize] {$b$} ([xshift=0.3cm]braiding2tensorab0.south);
\draw[solid] ([xshift=0.cm]braiding2tensorab0.north) to node [right,font=\footnotesize] {$ab$} (braiding2brCircleTimesabc0dom1);
\draw[solid] ([xshift=0.6cm]braiding2Subscriptalphaabc0.north) to node [right,font=\footnotesize] {$c$} (braiding2brCircleTimesabc0dom2);
\draw[solid] (braiding2brCircleTimesabc0codom2) to ([xshift=0.cm]braiding2Powertensorab-10.south);
\draw[solid] (braiding2brCircleTimesabc0codom1) to ([xshift=-0.6cm]braiding2Subscriptalphacab0.south);
\draw[solid] ([xshift=-0.3cm]braiding2Powertensorab-10.north) to node [right,font=\footnotesize] {$a$} ([xshift=0.cm]braiding2Subscriptalphacab0.south);
\draw[solid] ([xshift=0.3cm]braiding2Powertensorab-10.north) to node [right,font=\footnotesize] {$b$} ([xshift=0.6cm]braiding2Subscriptalphacab0.south);
\draw[solid] ([xshift=-0.6cm]braiding2Subscriptalphacab0.north) to node [right,font=\footnotesize] {$c$} ([xshift=-0.3cm]braiding2fcadi1.south);
\draw[solid] ([xshift=0.cm]braiding2Subscriptalphacab0.north) to node [right,font=\footnotesize] {$a$} ([xshift=0.3cm]braiding2fcadi1.south);
\draw[solid] ([xshift=0.cm]braiding2fcadi1.north) to node [right,font=\footnotesize] {$d$} ([xshift=-0.45cm]braiding2fdbuj1.south);
\draw[solid] ([xshift=0.6cm]braiding2Subscriptalphacab0.north) to node [right,font=\footnotesize] {$b$} ([xshift=0.45cm]braiding2fdbuj1.south);
\draw[solid] ([xshift=0.cm]braiding2fdbuj1.north) to ([xshift=0.cm]braiding2u1.south);
 
\end{tikzpicture}}
\end{center}

\noindent We can re-write (\ref{eq:hex11}) using the $R$-matrix coefficients (defined in line (\ref{eq:R})) and $F$-matrix coefficients (defined in line (\ref{eq:F}))
\begin{eqnarray}
\sum_{i'=1}^{N_{ac}^d} 
R_{ac}^d \left[ \scalebox{.6}{ $\begin{array}{c} i \\ i' \end{array}$ } \right] 
e_{db}^{u,j} 
\circ (e_{ac}^{d,i'} \otimes 1_b) 
\circ \a_{a,c,b} \circ (1_a \otimes \b_{b,c}) 
& = & \label{eq:hex12} \\
& \hspace{-11.8cm} = & \hspace{-6cm}
\sum_{e \in L} \sum_{i''=1}^{N_{ce}^u} \sum_{j'=1}^{N_{ab}^e}
F_{cab}^u \left[ \scalebox{.6}{$\begin{array}{ccc} i & d & j \\ i'' & e & j' \end{array}$} \right] 
e_{ce}^{u,i''} 
\circ (1_c \otimes e_{ab}^{e,j'})
\circ \b_{a \otimes b,c} 
\circ \a_{a,b,c} \nonumber
\end{eqnarray}

\[\raisebox{2cm}{$\sum_{i'=1}^{N_{ac}^d} 
R_{ac}^d \left[ \scalebox{.6}{$\begin{array}{c} i \\ i' \end{array}$} \right] $}
\scalebox{.6}{\begin{tikzpicture}
\makeatletter{}\draw (1.05,5.4) node [minimum height=0, minimum width=0.4cm] (braiding3u2) {$u$};
\draw (1.05,4.5) node [draw, rounded corners=3pt, fill=gray!10, minimum width=1.1cm] (braiding3fdbuj2) {$j$};
\draw (0.6,3.6) node [draw, rounded corners=3pt, fill=gray!10, minimum width=0.8cm] (braiding3facdDerivative1i0) {$i'$};
\draw (1.5,3.6) node [minimum height=0] (braiding3idb15) {$$};
\draw (0.9,2.7) node [draw, minimum width=1.4cm] (braiding3Subscriptalphaacb1) {$\alpha_{a,c,b}$};
\draw (0.3,1.8) node [minimum height=0] (braiding3ida19) {$$};
\draw (1.2,1.8) node [] (braiding3brbc1) {} (1.5,1.35) coordinate (braiding3brbc1dom2)..controls (1.5,1.8) and (0.9,1.8)..(0.9,2.25) coordinate (braiding3brbc1codom1);
\draw[preaction={draw,line width=2mm,white}] (0.9,1.35) coordinate (braiding3brbc1dom1)..controls (0.9,1.8) and (1.5,1.8)..(1.5,2.25) coordinate (braiding3brbc1codom2);
\draw (0.3,0.45) node [minimum height=0, minimum width=0.4cm] (braiding3a13) {$a$};
\draw (0.9,0.45) node [minimum height=0, minimum width=0.4cm] (braiding3b11) {$b$};
\draw (1.5,0.45) node [minimum height=0, minimum width=0.4cm] (braiding3c11) {$c$};
\draw[solid] ([xshift=0.cm]braiding3b11.north) to (braiding3brbc1dom1);
\draw[solid] ([xshift=0.cm]braiding3c11.north) to (braiding3brbc1dom2);
\draw[solid] ([xshift=0.cm]braiding3a13.north) to ([xshift=-0.6cm]braiding3Subscriptalphaacb1.south);
\draw[solid] (braiding3brbc1codom1) to ([xshift=0.cm]braiding3Subscriptalphaacb1.south);
\draw[solid] (braiding3brbc1codom2) to ([xshift=0.6cm]braiding3Subscriptalphaacb1.south);
\draw[solid] ([xshift=-0.6cm]braiding3Subscriptalphaacb1.north) to node [right,font=\footnotesize] {$a$} ([xshift=-0.3cm]braiding3facdDerivative1i0.south);
\draw[solid] ([xshift=0.cm]braiding3Subscriptalphaacb1.north) to node [right,font=\footnotesize] {$c$} ([xshift=0.3cm]braiding3facdDerivative1i0.south);
\draw[solid] ([xshift=0.cm]braiding3facdDerivative1i0.north) to node [right,font=\footnotesize] {$d$} ([xshift=-0.45cm]braiding3fdbuj2.south);
\draw[solid] ([xshift=0.6cm]braiding3Subscriptalphaacb1.north) to node [right,font=\footnotesize] {$b$} ([xshift=0.45cm]braiding3fdbuj2.south);
\draw[solid] ([xshift=0.cm]braiding3fdbuj2.north) to ([xshift=0.cm]braiding3u2.south);
 
\end{tikzpicture}}
\raisebox{2cm}{$= \sum_{e \in L} \sum_{i''=1}^{N_{ce}^u} \sum_{j'=1}^{N_{ab}^e}
F_{cab}^u \left[ \scalebox{.6}{$\begin{array}{ccc} i & d & j \\ i'' & e & j' \end{array}$} \right] 
$}
\scalebox{.6}{\begin{tikzpicture}
\makeatletter{}\draw (1.05,7.2) node [minimum height=0, minimum width=0.4cm] (braiding4u3) {$u$};
\draw (1.05,6.3) node [draw, rounded corners=3pt, fill=gray!10, minimum width=1.1cm] (braiding4fceuDerivative2i0) {$i''$};
\draw (0.6,5.4) node [minimum height=0] (braiding4idc4) {$$};
\draw (1.5,5.4) node [draw, rounded corners=3pt, fill=gray!10, minimum width=0.8cm] (braiding4fabeDerivative1j0) {$j'$};
\draw (0.6,4.5) node [minimum height=0] (braiding4idc3) {$$};
\draw (1.5,4.5) node [draw, minimum width=0.8cm] (braiding4Powertensorab-11) {$1_{ab}$};
\draw (1.05,3.6) node [] (braiding4brCircleTimesabc1) {} (1.5,3.15) coordinate (braiding4brCircleTimesabc1dom2)..controls (1.5,3.6) and (0.6,3.6)..(0.6,4.05) coordinate (braiding4brCircleTimesabc1codom1);
\draw[preaction={draw,line width=2mm,white}] (0.6,3.15) coordinate (braiding4brCircleTimesabc1dom1)..controls (0.6,3.6) and (1.5,3.6)..(1.5,4.05) coordinate (braiding4brCircleTimesabc1codom2);
\draw (0.6,2.7) node [draw, minimum width=0.8cm] (braiding4tensorab1) {$1_{ab}$};
\draw (1.5,2.7) node [minimum height=0] (braiding4idc2) {$$};
\draw (0.9,1.8) node [draw, minimum width=1.4cm] (braiding4Subscriptalphaabc1) {$\alpha_{a,b,c}$};
\draw (0.3,0.45) node [minimum height=0, minimum width=0.4cm] (braiding4a14) {$a$};
\draw (0.9,0.45) node [minimum height=0, minimum width=0.4cm] (braiding4b12) {$b$};
\draw (1.5,0.45) node [minimum height=0, minimum width=0.4cm] (braiding4c12) {$c$};
\draw[solid] ([xshift=0.cm]braiding4a14.north) to ([xshift=-0.6cm]braiding4Subscriptalphaabc1.south);
\draw[solid] ([xshift=0.cm]braiding4b12.north) to ([xshift=0.cm]braiding4Subscriptalphaabc1.south);
\draw[solid] ([xshift=0.cm]braiding4c12.north) to ([xshift=0.6cm]braiding4Subscriptalphaabc1.south);
\draw[solid] ([xshift=-0.6cm]braiding4Subscriptalphaabc1.north) to node [right,font=\footnotesize] {$a$} ([xshift=-0.3cm]braiding4tensorab1.south);
\draw[solid] ([xshift=0.cm]braiding4Subscriptalphaabc1.north) to node [right,font=\footnotesize] {$b$} ([xshift=0.3cm]braiding4tensorab1.south);
\draw[solid] ([xshift=0.cm]braiding4tensorab1.north) to node [right,font=\footnotesize] {$ab$} (braiding4brCircleTimesabc1dom1);
\draw[solid] ([xshift=0.6cm]braiding4Subscriptalphaabc1.north) to node [right,font=\footnotesize] {$c$} (braiding4brCircleTimesabc1dom2);
\draw[solid] (braiding4brCircleTimesabc1codom2) to ([xshift=0.cm]braiding4Powertensorab-11.south);
\draw[solid] ([xshift=-0.3cm]braiding4Powertensorab-11.north) to node [right,font=\footnotesize] {$a$} ([xshift=-0.3cm]braiding4fabeDerivative1j0.south);
\draw[solid] ([xshift=0.3cm]braiding4Powertensorab-11.north) to node [right,font=\footnotesize] {$b$} ([xshift=0.3cm]braiding4fabeDerivative1j0.south);
\draw[solid] (braiding4brCircleTimesabc1codom1) to ([xshift=-0.45cm]braiding4fceuDerivative2i0.south);
\draw[solid] ([xshift=0.cm]braiding4fabeDerivative1j0.north) to node [right,font=\footnotesize] {$e$} ([xshift=0.45cm]braiding4fceuDerivative2i0.south);
\draw[solid] ([xshift=0.cm]braiding4fceuDerivative2i0.north) to ([xshift=0.cm]braiding4u3.south);
 
\end{tikzpicture}}
\]

\noindent Recall for every generic object $x \in \cC$ we have (see \S\ref{sec:conventions}),
\[ \sum_{d \in L} \sum_i e_{d,i}^x \circ e_x^{d,i} = 1_x \quad , \quad e_x^{e,j'} \circ e_{d,i}^x = \d_{d,e} \d_{i,j'} 1_e \]
\[ \Rightarrow \sum_{d \in L} \sum_i (e_{d,i}^x \otimes 1_c) \circ (e_x^{d,i} \otimes 1_c) = 1_x \otimes 1_c \]
\[ \Rightarrow \sum_{d \in L} \sum_i \b_{x,c} \circ (e_{d,i}^x \otimes 1_c) \circ (e_x^{d,i} \otimes 1_c) = \b_{x,c} \]
Since $\b$ is natural
\begin{equation}
\sum_{d \in L} \sum_i
(1_c \otimes e_{d,i}^x) \circ \b_{d,c} \circ (e_x^{d,i} \otimes 1_c) = \b_{x,c}
\end{equation}
Applying to $x = a \otimes b$,
\begin{equation}
\label{eq:lin1} 
\sum_{d \in L} \sum_{i=1}^{N_{ab}^d}
(1_c \otimes e_{d,i}^{ab}) \circ \b_{d,c} \circ (e_{ab}^{d,i} \otimes 1_c) = \b_{a \otimes b,c}
\end{equation}
Post-composing both sides of (\ref{eq:lin1}) with $1_c \otimes e_{ab}^{e,j'}$ we get
\begin{eqnarray}
(1_c \otimes e_{ab}^{e,j'}) \circ \b_{a \otimes b,c} & = & \sum_{d \in L} \sum_{i=1}^{N_{ab}^d}
(1_c \otimes e_{ab}^{e,j'}) \circ (1_c \otimes e_{d,i}^{ab}) \circ \b_{d,c} \circ (e_{ab}^{d,i} \otimes 1_c) 
\nonumber \\
& = & \sum_{d \in L} \sum_{i=1}^{N_{ab}^d}
(1_c \otimes (e_{ab}^{e,j'} \circ e_{d,i}^{ab})) \circ \b_{d,c} \circ (e_{ab}^{d,i} \otimes 1_c) 
\nonumber \\
& = & (1_c \otimes 1_e) \circ \b_{e,c} \circ (e_{ab}^{e,j'} \otimes 1_c) 
\rule{0mm}{.6cm} \nonumber \\
& = & \b_{e,c} \circ (e_{ab}^{e,j'} \otimes 1_c) \label{eq:lin2}
\rule{0mm}{.8cm} 
\end{eqnarray}
Applying (\ref{eq:lin2}), we can re-write (\ref{eq:hex12}) in the form
\begin{eqnarray}
\sum_{i'=1}^{N_{ac}^d} 
R_{ac}^d \left[ \scalebox{.6}{ $\begin{array}{c} i \\ i' \end{array}$ } \right] 
e_{db}^{u,j} 
\circ (e_{ac}^{d,i'} \otimes 1_b) 
\circ \a_{a,c,b} \circ (1_a \otimes \b_{b,c}) 
& = & \label{eq:hex13} \\
& \hspace{-12cm} = & \hspace{-6cm}
\sum_{e \in L} \sum_{i''=1}^{N_{ce}^u} \sum_{j'=1}^{N_{ab}^e}
F_{cab}^u \left[ \scalebox{.6}{$\begin{array}{ccc} i & d & j \\ i'' & e & j' \end{array}$} \right] 
e_{ce}^{u,i''} 
\circ \b_{e,c} \circ (e_{ab}^{e,j'} \otimes 1_c)
\circ \a_{a,b,c} \nonumber
\end{eqnarray}

\[
\raisebox{2cm}{$\sum_{i'=1}^{N_{ac}^d} 
R_{ac}^d \left[ \scalebox{.6}{ $\begin{array}{c} i \\ i' \end{array}$ } \right]
$}
\scalebox{.6}{\begin{tikzpicture}
\makeatletter{}\draw (1.05,5.4) node [minimum height=0, minimum width=0.4cm] (braiding5u4) {$u$};
\draw (1.05,4.5) node [draw, rounded corners=3pt, fill=gray!10, minimum width=1.1cm] (braiding5fdbuj3) {$j$};
\draw (0.6,3.6) node [draw, rounded corners=3pt, fill=gray!10, minimum width=0.8cm] (braiding5facdDerivative1i1) {$i'$};
\draw (1.5,3.6) node [minimum height=0] (braiding5idb16) {$$};
\draw (0.9,2.7) node [draw, minimum width=1.4cm] (braiding5Subscriptalphaacb2) {$\alpha_{a,c,b}$};
\draw (0.3,1.8) node [minimum height=0] (braiding5ida20) {$$};
\draw (1.2,1.8) node [] (braiding5brbc2) {} (1.5,1.35) coordinate (braiding5brbc2dom2)..controls (1.5,1.8) and (0.9,1.8)..(0.9,2.25) coordinate (braiding5brbc2codom1);
\draw[preaction={draw,line width=2mm,white}] (0.9,1.35) coordinate (braiding5brbc2dom1)..controls (0.9,1.8) and (1.5,1.8)..(1.5,2.25) coordinate (braiding5brbc2codom2);
\draw (0.3,0.45) node [minimum height=0, minimum width=0.4cm] (braiding5a15) {$a$};
\draw (0.9,0.45) node [minimum height=0, minimum width=0.4cm] (braiding5b13) {$b$};
\draw (1.5,0.45) node [minimum height=0, minimum width=0.4cm] (braiding5c13) {$c$};
\draw[solid] ([xshift=0.cm]braiding5b13.north) to (braiding5brbc2dom1);
\draw[solid] ([xshift=0.cm]braiding5c13.north) to (braiding5brbc2dom2);
\draw[solid] ([xshift=0.cm]braiding5a15.north) to ([xshift=-0.6cm]braiding5Subscriptalphaacb2.south);
\draw[solid] (braiding5brbc2codom1) to ([xshift=0.cm]braiding5Subscriptalphaacb2.south);
\draw[solid] (braiding5brbc2codom2) to ([xshift=0.6cm]braiding5Subscriptalphaacb2.south);
\draw[solid] ([xshift=-0.6cm]braiding5Subscriptalphaacb2.north) to node [right,font=\footnotesize] {$a$} ([xshift=-0.3cm]braiding5facdDerivative1i1.south);
\draw[solid] ([xshift=0.cm]braiding5Subscriptalphaacb2.north) to node [right,font=\footnotesize] {$c$} ([xshift=0.3cm]braiding5facdDerivative1i1.south);
\draw[solid] ([xshift=0.cm]braiding5facdDerivative1i1.north) to node [right,font=\footnotesize] {$d$} ([xshift=-0.45cm]braiding5fdbuj3.south);
\draw[solid] ([xshift=0.6cm]braiding5Subscriptalphaacb2.north) to node [right,font=\footnotesize] {$b$} ([xshift=0.45cm]braiding5fdbuj3.south);
\draw[solid] ([xshift=0.cm]braiding5fdbuj3.north) to ([xshift=0.cm]braiding5u4.south);
 
\end{tikzpicture}}
\raisebox{2cm}{$= \sum_{e \in L} \sum_{i''=1}^{N_{ce}^u} \sum_{j'=1}^{N_{ab}^e}
F_{cab}^u \left[ \scalebox{.6}{$\begin{array}{ccc} i & d & j \\ i'' & e & j' \end{array}$} \right]$}
\scalebox{.6}{\begin{tikzpicture}
\makeatletter{}\draw (1.05,5.4) node [minimum height=0, minimum width=0.4cm] (braiding6u5) {$u$};
\draw (1.05,4.5) node [draw, rounded corners=3pt, fill=gray!10, minimum width=0.8cm] (braiding6fceuDerivative2i1) {$i''$};
\draw (1.05,3.6) node [] (braiding6brec0) {} (1.5,3.15) coordinate (braiding6brec0dom2)..controls (1.5,3.6) and (0.75,3.6)..(0.75,4.05) coordinate (braiding6brec0codom1);
\draw[preaction={draw,line width=2mm,white}] (0.6,3.15) coordinate (braiding6brec0dom1)..controls (0.6,3.6) and (1.35,3.6)..(1.35,4.05) coordinate (braiding6brec0codom2);
\draw (0.6,2.7) node [draw, rounded corners=3pt, fill=gray!10, minimum width=0.8cm] (braiding6fabeDerivative1j1) {$j'$};
\draw (1.5,2.7) node [minimum height=0] (braiding6idc5) {$$};
\draw (0.9,1.8) node [draw, minimum width=1.4cm] (braiding6Subscriptalphaabc2) {$\alpha_{a,b,c}$};
\draw (0.3,0.45) node [minimum height=0, minimum width=0.4cm] (braiding6a16) {$a$};
\draw (0.9,0.45) node [minimum height=0, minimum width=0.4cm] (braiding6b14) {$b$};
\draw (1.5,0.45) node [minimum height=0, minimum width=0.4cm] (braiding6c14) {$c$};
\draw[solid] ([xshift=0.cm]braiding6a16.north) to ([xshift=-0.6cm]braiding6Subscriptalphaabc2.south);
\draw[solid] ([xshift=0.cm]braiding6b14.north) to ([xshift=0.cm]braiding6Subscriptalphaabc2.south);
\draw[solid] ([xshift=0.cm]braiding6c14.north) to ([xshift=0.6cm]braiding6Subscriptalphaabc2.south);
\draw[solid] ([xshift=-0.6cm]braiding6Subscriptalphaabc2.north) to node [right,font=\footnotesize] {$a$} ([xshift=-0.3cm]braiding6fabeDerivative1j1.south);
\draw[solid] ([xshift=0.cm]braiding6Subscriptalphaabc2.north) to node [right,font=\footnotesize] {$b$} ([xshift=0.3cm]braiding6fabeDerivative1j1.south);
\draw[solid] ([xshift=0.cm]braiding6fabeDerivative1j1.north) to node [right,font=\footnotesize] {$e$} (braiding6brec0dom1);
\draw[solid] ([xshift=0.6cm]braiding6Subscriptalphaabc2.north) to node [right,font=\footnotesize] {$c$} (braiding6brec0dom2);
\draw[solid] (braiding6brec0codom1) to ([xshift=-0.3cm]braiding6fceuDerivative2i1.south);
\draw[solid] (braiding6brec0codom2) to ([xshift=0.3cm]braiding6fceuDerivative2i1.south);
\draw[solid] ([xshift=0.cm]braiding6fceuDerivative2i1.north) to ([xshift=0.cm]braiding6u5.south);
 
\end{tikzpicture}}
\]

\noindent Applying the $R$- and $F$-matrix coefficients to (\ref{eq:hex13}) we get
\begin{eqnarray}
\sum_{i'=1}^{N_{ac}^d} \sum_{f \in L} \sum_{i^{(3)}=1}^{N_{af}^{u}} \sum_{j''=1}^{N_{cb}^{f}}
R_{ac}^d \left[ \scalebox{.6}{ $\begin{array}{c} i \\ i' \end{array}$ } \right]
F_{acb}^u \left[ \scalebox{.6}{$\begin{array}{ccc} i' & d & j \\ i^{(3)} & f & j'' \end{array}$} \right]  
e_{af}^{u,i^{(3)}} 
\circ (1_a \otimes e_{cb}^{f,j''}) 
\circ (1_a \otimes \b_{b,c}) 
& & \label{eq:hex14} \\
& \hspace{-19.8cm} = & \hspace{-10cm}
\sum_{e \in L} 
\sum_{i''=1}^{N_{ce}^u}
\sum_{j'=1}^{N_{ab}^e} 
\sum_{i^{(4)}=1}^{N_{ec}^u}
F_{cab}^u \left[ \scalebox{.6}{$\begin{array}{ccc} i & d & j \\ i'' & e & j' \end{array}$} \right]
R_{ec}^u \left[ \scalebox{.6}{ $\begin{array}{c} i'' \\ i^{(4)} \end{array}$ } \right]
e_{ec}^{u,i^{(4)}} 
\circ (e_{ab}^{e,j'} \otimes 1_c)
\circ \a_{a,b,c} \nonumber
\end{eqnarray}

\[\raisebox{1cm}{$\sum_{i'=1}^{N_{ac}^d} \sum_{f \in L} \sum_{i^{(3)}=1}^{N_{af}^{u}} \sum_{j''=1}^{N_{cb}^{f}}
R_{ac}^d \left[ \scalebox{.6}{ $\begin{array}{c} i \\ i' \end{array}$ } \right]
$}
\scalebox{.6}{\begin{tikzpicture}
\makeatletter{}\draw (0.75,4.5) node [minimum height=0, minimum width=0.4cm] (braiding7u6) {$u$};
\draw (0.75,3.6) node [draw, rounded corners=3pt, fill=gray!10, minimum width=1.1cm] (braiding7fafuPoweri30) {$i^3$};
\draw (0.3,2.7) node [minimum height=0] (braiding7ida22) {$$};
\draw (1.2,2.7) node [draw, rounded corners=3pt, fill=gray!10, minimum width=0.8cm] (braiding7fcbfDerivative2j0) {$j''$};
\draw (0.3,1.8) node [minimum height=0] (braiding7ida21) {$$};
\draw (1.2,1.8) node [] (braiding7brbc3) {} (1.5,1.35) coordinate (braiding7brbc3dom2)..controls (1.5,1.8) and (0.9,1.8)..(0.9,2.25) coordinate (braiding7brbc3codom1);
\draw[preaction={draw,line width=2mm,white}] (0.9,1.35) coordinate (braiding7brbc3dom1)..controls (0.9,1.8) and (1.5,1.8)..(1.5,2.25) coordinate (braiding7brbc3codom2);
\draw (0.3,0.45) node [minimum height=0, minimum width=0.4cm] (braiding7a17) {$a$};
\draw (0.9,0.45) node [minimum height=0, minimum width=0.4cm] (braiding7b15) {$b$};
\draw (1.5,0.45) node [minimum height=0, minimum width=0.4cm] (braiding7c15) {$c$};
\draw[solid] ([xshift=0.cm]braiding7b15.north) to (braiding7brbc3dom1);
\draw[solid] ([xshift=0.cm]braiding7c15.north) to (braiding7brbc3dom2);
\draw[solid] (braiding7brbc3codom1) to ([xshift=-0.3cm]braiding7fcbfDerivative2j0.south);
\draw[solid] (braiding7brbc3codom2) to ([xshift=0.3cm]braiding7fcbfDerivative2j0.south);
\draw[solid] ([xshift=0.cm]braiding7a17.north) to ([xshift=-0.45cm]braiding7fafuPoweri30.south);
\draw[solid] ([xshift=0.cm]braiding7fcbfDerivative2j0.north) to node [right,font=\footnotesize] {$f$} ([xshift=0.45cm]braiding7fafuPoweri30.south);
\draw[solid] ([xshift=0.cm]braiding7fafuPoweri30.north) to ([xshift=0.cm]braiding7u6.south);
 
\end{tikzpicture}}
\]\[
\raisebox{1cm}{$= \sum_{e \in L} 
\sum_{i''=1}^{N_{ce}^u}
\sum_{j'=1}^{N_{ab}^e} 
\sum_{i^{(4)}=1}^{N_{ec}^u}
F_{cab}^u \left[ \scalebox{.6}{$\begin{array}{ccc} i & d & j \\ i'' & e & j' \end{array}$} \right]
R_{ec}^u \left[ \scalebox{.6}{ $\begin{array}{c} i'' \\ i^{(4)} \end{array}$ } \right]
$}
\scalebox{.6}{\begin{tikzpicture}
\makeatletter{}\draw (1.05,4.5) node [minimum height=0, minimum width=0.4cm] (braiding8u7) {$u$};
\draw (1.05,3.6) node [draw, rounded corners=3pt, fill=gray!10, minimum width=1.1cm] (braiding8fecuPoweri40) {$i^4$};
\draw (0.6,2.7) node [draw, rounded corners=3pt, fill=gray!10, minimum width=0.8cm] (braiding8fabeDerivative1j2) {$j'$};
\draw (1.5,2.7) node [minimum height=0] (braiding8idc6) {$$};
\draw (0.9,1.8) node [draw, minimum width=1.4cm] (braiding8Subscriptalphaabc3) {$\alpha_{a,b,c}$};
\draw (0.3,0.45) node [minimum height=0, minimum width=0.4cm] (braiding8a18) {$a$};
\draw (0.9,0.45) node [minimum height=0, minimum width=0.4cm] (braiding8b16) {$b$};
\draw (1.5,0.45) node [minimum height=0, minimum width=0.4cm] (braiding8c16) {$c$};
\draw[solid] ([xshift=0.cm]braiding8a18.north) to ([xshift=-0.6cm]braiding8Subscriptalphaabc3.south);
\draw[solid] ([xshift=0.cm]braiding8b16.north) to ([xshift=0.cm]braiding8Subscriptalphaabc3.south);
\draw[solid] ([xshift=0.cm]braiding8c16.north) to ([xshift=0.6cm]braiding8Subscriptalphaabc3.south);
\draw[solid] ([xshift=-0.6cm]braiding8Subscriptalphaabc3.north) to node [right,font=\footnotesize] {$a$} ([xshift=-0.3cm]braiding8fabeDerivative1j2.south);
\draw[solid] ([xshift=0.cm]braiding8Subscriptalphaabc3.north) to node [right,font=\footnotesize] {$b$} ([xshift=0.3cm]braiding8fabeDerivative1j2.south);
\draw[solid] ([xshift=0.cm]braiding8fabeDerivative1j2.north) to node [right,font=\footnotesize] {$e$} ([xshift=-0.45cm]braiding8fecuPoweri40.south);
\draw[solid] ([xshift=0.6cm]braiding8Subscriptalphaabc3.north) to node [right,font=\footnotesize] {$c$} ([xshift=0.45cm]braiding8fecuPoweri40.south);
\draw[solid] ([xshift=0.cm]braiding8fecuPoweri40.north) to ([xshift=0.cm]braiding8u7.south);
 
\end{tikzpicture}}
\]

\noindent We can apply $R$- and $F$-matrix coefficients again to (\ref{eq:hex14})
\begin{eqnarray}
\sum_{\scalebox{.5}{$\begin{array}{c} i',f \\ i^{(3)},j''\end{array}$}}
\sum_{j^{(3)}=1}^{N_{bc}^{f}}
R_{ac}^d \left[ \scalebox{.6}{ $\begin{array}{c} i \\ i' \end{array}$ } \right]
F_{acb}^u \left[ \scalebox{.6}{$\begin{array}{ccc} i' & d & j \\ i^{(3)} & f & j'' \end{array}$} \right] 
R_{bc}^{f} \left[ \scalebox{.6}{ $\begin{array}{c} j'' \\ j^{(3)} \end{array}$ } \right]
e_{af}^{u,i^{(3)}} 
\circ (1_a \otimes e_{bc}^{f,j^{(3)}})  
& = & \label{eq:hex15} \\  
& \hspace{-20.7cm} = & \hspace{-10.5cm}
\sum_{\scalebox{.5}{$\begin{array}{c}e,i'' \\ j',i^{(4)}\end{array}$}}
\sum_{g \in L} \sum_{i^{(5)}=1}^{N_{ag}^{u}}
\sum_{j^{(4)}=1}^{N_{bc}^{g}}
F_{cab}^u \left[ \scalebox{.6}{$\begin{array}{ccc} i & d & j \\ i'' & e & j' \end{array}$} \right]
R_{ec}^u \left[ \scalebox{.6}{ $\begin{array}{c} i'' \\ i^{(4)} \end{array}$ } \right]
F_{abc}^u \left[ \scalebox{.6}{$\begin{array}{ccc} j' & e & i^{(4)} \\ i^{(5)} & g & j^{(4)} \end{array}$} \right]
e_{ag}^{u,i^{(5)}} 
\circ (1_a \otimes e_{bc}^{g,j^{(4)}})
\nonumber
\end{eqnarray}

\[
\raisebox{1cm}{$\sum_{\scalebox{.5}{$\begin{array}{c} i',f \\ i^{(3)},j''\end{array}$}}
\sum_{j^{(3)}=1}^{N_{bc}^{f}}
R_{ac}^d \left[ \scalebox{.6}{ $\begin{array}{c} i \\ i' \end{array}$ } \right]
F_{acb}^u \left[ \scalebox{.6}{$\begin{array}{ccc} i' & d & j \\ i^{(3)} & f & j'' \end{array}$} \right] 
R_{bc}^{f} \left[ \scalebox{.6}{ $\begin{array}{c} j'' \\ j^{(3)} \end{array}$ } \right]
$}
\scalebox{.6}{\begin{tikzpicture}
\makeatletter{}\draw (0.75,3.6) node [minimum height=0, minimum width=0.4cm] (braiding9u8) {$u$};
\draw (0.75,2.7) node [draw, rounded corners=3pt, fill=gray!10, minimum width=1.1cm] (braiding9fafuPoweri31) {$i^3$};
\draw (0.3,1.8) node [minimum height=0] (braiding9ida23) {$$};
\draw (1.2,1.8) node [draw, rounded corners=3pt, fill=gray!10, minimum width=0.8cm] (braiding9fbcfPowerj30) {$j^3$};
\draw (0.3,0.45) node [minimum height=0, minimum width=0.4cm] (braiding9a19) {$a$};
\draw (0.9,0.45) node [minimum height=0, minimum width=0.4cm] (braiding9b17) {$b$};
\draw (1.5,0.45) node [minimum height=0, minimum width=0.4cm] (braiding9c17) {$c$};
\draw[solid] ([xshift=0.cm]braiding9b17.north) to ([xshift=-0.3cm]braiding9fbcfPowerj30.south);
\draw[solid] ([xshift=0.cm]braiding9c17.north) to ([xshift=0.3cm]braiding9fbcfPowerj30.south);
\draw[solid] ([xshift=0.cm]braiding9a19.north) to ([xshift=-0.45cm]braiding9fafuPoweri31.south);
\draw[solid] ([xshift=0.cm]braiding9fbcfPowerj30.north) to node [right,font=\footnotesize] {$f$} ([xshift=0.45cm]braiding9fafuPoweri31.south);
\draw[solid] ([xshift=0.cm]braiding9fafuPoweri31.north) to ([xshift=0.cm]braiding9u8.south);
 
\end{tikzpicture}}
\]
\[
\raisebox{1cm}{$= \sum_{\scalebox{.5}{$\begin{array}{c}e,i'' \\ j',i^{(4)}\end{array}$}}
\sum_{g \in L} \sum_{i^{(5)}=1}^{N_{ag}^{u}}
\sum_{j^{(4)}=1}^{N_{bc}^{g}}
F_{cab}^u \left[ \scalebox{.6}{$\begin{array}{ccc} i & d & j \\ i'' & e & j' \end{array}$} \right]
R_{ec}^u \left[ \scalebox{.6}{ $\begin{array}{c} i'' \\ i^{(4)} \end{array}$ } \right]
F_{abc}^u \left[ \scalebox{.6}{$\begin{array}{ccc} j' & e & i^{(4)} \\ i^{(5)} & g & j^{(4)} \end{array}$} \right]
$}
\scalebox{.6}{\begin{tikzpicture}
\makeatletter{}\draw (0.75,3.6) node [minimum height=0, minimum width=0.4cm] (braiding10u9) {$u$};
\draw (0.75,2.7) node [draw, rounded corners=3pt, fill=gray!10, minimum width=1.1cm] (braiding10faguPoweri50) {$i^5$};
\draw (0.3,1.8) node [minimum height=0] (braiding10ida24) {$$};
\draw (1.2,1.8) node [draw, rounded corners=3pt, fill=gray!10, minimum width=0.8cm] (braiding10fbcgPowerj40) {$j^4$};
\draw (0.3,0.45) node [minimum height=0, minimum width=0.4cm] (braiding10a20) {$a$};
\draw (0.9,0.45) node [minimum height=0, minimum width=0.4cm] (braiding10b18) {$b$};
\draw (1.5,0.45) node [minimum height=0, minimum width=0.4cm] (braiding10c18) {$c$};
\draw[solid] ([xshift=0.cm]braiding10b18.north) to ([xshift=-0.3cm]braiding10fbcgPowerj40.south);
\draw[solid] ([xshift=0.cm]braiding10c18.north) to ([xshift=0.3cm]braiding10fbcgPowerj40.south);
\draw[solid] ([xshift=0.cm]braiding10a20.north) to ([xshift=-0.45cm]braiding10faguPoweri50.south);
\draw[solid] ([xshift=0.cm]braiding10fbcgPowerj40.north) to node [right,font=\footnotesize] {$g$} ([xshift=0.45cm]braiding10faguPoweri50.south);
\draw[solid] ([xshift=0.cm]braiding10faguPoweri50.north) to ([xshift=0.cm]braiding10u9.south);
 
\end{tikzpicture}}
\]

\noindent Equation (\ref{eq:hex15}) holds iff for fixed $f=g$, $i^{(3)} = i^{(5)}$ and $j^{(3)} = j^{(4)}$
\begin{eqnarray}
\sum_{i'=1}^{N_{ac}^d} 
\sum_{j''=1}^{N_{cb}^f}
R_{a,c}^d \left[ \scalebox{.6}{ $\begin{array}{c} i \\ i' \end{array}$ } \right]
F_{acb}^u \left[ \scalebox{.6}{$\begin{array}{ccc} i' & d & j \\ i^{(3)} & f & j'' \end{array}$} \right] 
R_{b,c}^{f} \left[ \scalebox{.6}{ $\begin{array}{c} j'' \\ j^{(3)} \end{array}$ } \right]
& = & \label{eq:hex16} \\  
& \hspace{-7.8cm} = & \hspace{-4cm}
\sum_{e \in  L} 
\sum_{i''=1}^{N_{ce}^u} 
\sum_{j'=1}^{N_{ab}^e} 
\sum_{i^{(4)}=1}^{N_{ec}^u}
F_{cab}^u \left[ \scalebox{.6}{$\begin{array}{ccc} i & d & j \\ i'' & e & j' \end{array}$} \right]
R_{e,c}^u \left[ \scalebox{.6}{ $\begin{array}{c} i'' \\ i^{(4)} \end{array}$ } \right]
F_{abc}^u \left[ \scalebox{.6}{$\begin{array}{ccc} j' & e & i^{(4)} \\ i^{(3)} & f & j^{(3)} \end{array}$} \right]
\nonumber
\end{eqnarray}
Equation (\ref{eq:hex16}) above agrees with constraint (\ref{eq:hexa}) in Definition \ref{def:modular}.
The very same analysis can be applied to the second Hexagon commutative diagram (\ref{eq:hex2}).  The resulting equation will be
\begin{eqnarray}
\sum_{i'=1}^{N_{ac}^d} 
\sum_{j''=1}^{N_{cb}^f}
Q_{a,c}^d \left[ \scalebox{.6}{ $\begin{array}{c} i \\ i' \end{array}$ } \right]
F_{acb}^u \left[ \scalebox{.6}{$\begin{array}{ccc} i' & d & j \\ i^{(3)} & f & j'' \end{array}$} \right] 
Q_{b,c}^{f} \left[ \scalebox{.6}{ $\begin{array}{c} j'' \\ j^{(3)} \end{array}$ } \right]
& = & \label{eq:hex26} \\  
& \hspace{-7.8cm} = & \hspace{-4cm}
\sum_{e \in  L} 
\sum_{i''=1}^{N_{ce}^u} 
\sum_{j'=1}^{N_{ab}^e} 
\sum_{i^{(4)}=1}^{N_{ec}^u}
F_{abc}^u \left[ \scalebox{.6}{$\begin{array}{ccc} i & d & j \\ i'' & e & j' \end{array}$} \right]
Q_{e,c}^u \left[ \scalebox{.6}{ $\begin{array}{c} i'' \\ i^{(4)} \end{array}$ } \right]
F_{abc}^u \left[ \scalebox{.6}{$\begin{array}{ccc} j' & e & i^{(4)} \\ i^{(3)} & f & j^{(3)} \end{array}$} \right]
\nonumber
\end{eqnarray}
which agrees with constraint (\ref{eq:hexb}) in Definition \ref{def:modular}.  We conclude that a collection $\{R_{ab}^c\}$ of invertible matrices, with inverses $\{Q_{ab}^c\}$, given as part of a modular system $(L,N,F,R,\eps)$, defines a braiding on $\cC \equiv \cC(L,N,F)$.

\noindent \textbf{Pivotal Structure:} Recall from \S\ref{ss:ps}, a pivotal structure in a skeletal fusion category $\cC$ is a monoidal natural isomorphism $\eps: \Delta \ra 1_\cC$, where $\Delta=(\Delta_0,\Delta_1,\Delta_2)$ with $\Delta_0 = 1_{\mathbf{1}}$ and $\Delta_1 = 1_\cC$.  Forgetting for a moment the monoidal constraint, $\eps$ is a natural automorphism of the identity functor on $\cC$.

By applying Lemma \ref{lem:nonsense} to $F=G=1_\cC$ one learns that a pivotal structure $\eps$ in a skeletal fusion category $\cC$ is uniquely determined, as an additive natural isomorphism, by a collection of isomorphisms
\[ \eps(a) : a \xra{\ \cong\ }a \]
For $\eps$ to also be monoidal we need
\begin{equation}
\label{eq:eta11}
\eps(a \otimes b) \circ \Delta_2(a,b) \circ \eps(a)^{-1} \otimes \eps(b)^{-1}  = 1_{a \otimes b}
\end{equation}
(see diagram \eqref{eq:eta1} \S\ref{ss:ps}.)

Let us define 
\[ \eps(a) \defeq \eps_a 1_a \quad , \quad a \in L \]
where $\{\eps_a\}_{a \in L}$ is part of the modular system $(L,N,F,R,\eps)$.  To show $\eps$ gives rise to a pivotal structure on $\cC$, we need to show Equation (\ref{eq:eta11}) holds.

By Remark \ref{ss:gauge} on page (\pageref{ss:gauge}), there is a canonical choice of basis $\{ e_{ab}^{c,i} \}_{i=1}^{N_{ab}^c}  \subset \cC(a \otimes b,c)$, for any $a,b,c \in L$. Subsequently, Equation (\ref{eq:eta11}) holds iff for any $c \in L$ and $e_{ab}^{c,i} \in \cC(a \otimes b,c)$, 
\begin{equation}  
\label{eq:eta2}
e_{ab}^{c,i} \circ \eps(a \otimes b) \circ \Delta_2(a,b) \circ \eps(a)^{-1} \otimes \eps(b)^{-1}  = e_{ab}^{c,i}
\end{equation}
Because $\eps$ is natural we can re-write Equation (\ref{eq:eta2}) 
\begin{equation} 
\label{eq:eta3}
\eps(c) \circ e_{ab}^{c,i} \circ \Delta_2(a,b) \circ \eps(a)^{-1} \otimes \eps(b)^{-1}  = e_{ab}^{c,i}
\end{equation}

Recall from \S\ref{ss:dd} a morphism $\Delta(a,b,x,f) : \Delta_1(a) \otimes \Delta_1(b) \ra \Delta_1(x)$ was defined for $a,b,x \in \cC$ and $f : a \otimes b \ra x$.  In Lemma \ref{lem:g_in_Delta} it was demonstrated that for any $g : x \ra y$ in $\cC$
\[ \Delta_1(g) \circ \Delta(a,b,x,f) = \Delta(a,b,y,g \circ f) \]
Let us take $x = a \otimes b$, $f = 1_x$, $y = c$ and $g = e_{ab}^{c,i} : a \otimes b \ra c$.  Then we can re-write Equation (\ref{eq:eta3}) in the form
\[
\eps(c) \circ  \Delta(a,b,c, e_{ab}^{c,i}) \circ \eps(a)^{-1} \otimes \eps(b)^{-1}  = e_{ab}^{c,i}
\]
or alternately,
\begin{equation} 
\label{eq:eta4}
\eps_c \eps_a^{-1} \eps_b^{-1}  \Delta(a,b,c,e_{ab}^{c,i}) = e_{ab}^{c,i}
\end{equation}

We proceed by showing Equation (\ref{eq:eta4}) is equivalent to condition (\ref{eq:piv}) in Definition \ref{def:modular}.  In particular we calculate $\Delta(a,b,c, e_{ab}^{c,i})$ the same way as we did on page \pageref{eq:piv_calc}, taking $\eps' \equiv 1$.  As before we rely on our choice of gauge $e_{aa^*}^{\mathbf{1},1} = e_a$.

As a result, Equation (\ref{eq:eta4}) holds if and only if 
\begin{equation} 
\label{eq:eta5}
\eps_c \eps_a^{-1} \eps_b^{-1}
\sum_{s =1}^{N_{bc^*}^{a^*}} F_{abc^*}^{\mathbf{1}} \left[ \scalebox{.6}{$ \begin{array}{ccc} i & c & 1 \\ 1 & a^* & s \end{array} $} \right] 
\sum_{t = 1}^{N_{c^*a}^{b^*}} F_{bc^*a}^{\mathbf{1}} \left[ \scalebox{.6}{$ \begin{array}{ccc} s & a^* & 1 \\ 1 & b^* & t \end{array} $} \right] 
\sum_{j = 1}^{N_{ab}^{c}} F_{c^*ab}^{\mathbf{1}} \left[ \scalebox{.6}{$ \begin{array}{ccc} t & b^* & 1 \\ 1 & c & j \end{array} $} \right] 
e_{ab}^{c,j}
= e_{ab}^{c,i}
\end{equation}
Setting up $i=j$ in Equation (\ref{eq:eta5}) results in condition (\ref{eq:piv}) in Definition \ref{def:modular}.  
We conclude the collection $\{\eps_a\}_{a \in L}$, given as part of the modular data $(L,N,F,R,\eps)$, gives rise to a pivotal structure on $\cC$. 

Recall from Lemma \ref{lem:eps_sym} one has an additional symmetry to $\eps$
\[ \eps_{a^*} = (\eps_a^{-1})^*\]
which by $\eps_a = \pm 1$ implies
\[ \eps_{a^*} = \eps_a \]

\noindent \textbf{Sphericality:} So far $\cC$ was endowed with braiding and pivotal structure.  Consider the field extension $k' = k(\sqrt{u_a})$.  Pick a root $\l_a \in k'$ of $p_a(z) = z^2 - u_a$ for every $a \in L$.  Consider the scalar extension
\[ \cC' \defeq \cC \otimes_k k' \]
By Lemma \ref{lem:ext}, $\cC'$ is a fusion category.  The braiding and pivotal structure in $\cC$ extend trivially to $\cC'$.  We re-normalize the evaluation and co-evaluation morphisms in $\cC'$.  
\begin{eqnarray}
e_a & \defeq & e_{aa^*}^{\mathbf{1},1} \otimes \l_a^{-1} \in \cC(a \otimes a^*,\mathbf{1}) \otimes_k k' = \cC'(a \otimes a^*, \mathbf{1}) \label{eq:renorm1} \\
c_a & \defeq & e^{a^*a}_{\mathbf{1},1} \otimes \l_a^{-1} \in \cC(\mathbf{1},a^* \otimes a) \otimes_k k' = \cC'(\mathbf{1},a^* \otimes a) \label{eq:renorm2} 
\end{eqnarray}
The left and right quantum dimensions of $a \in L$ are (see lines (\ref{def:qlx}) and (\ref{def:qrx}))
\[ 
q_l(a)  = (e_{aa^*}^{\mathbf{1},1} \circ \eps(a) \otimes 1_{a^*} \circ e_{\mathbf{1},1}^{aa^*}) \otimes \l_a^{-1} \l_{a^*}^{-1} = \eps_a 1_{\mathbf{1}} \otimes \l_a^{-1} \l_{a^*}^{-1}
\]
\[ 
q_r(a)  = (e_{a^*a}^{\mathbf{1},1} \circ 1_{a^*} \otimes \eps(a)^{-1} \circ e_{\mathbf{1},1}^{a^*a}) \otimes \l_{a^*}^{-1} \l_{a}^{-1} = \eps_a^{-1} 1_{\mathbf{1}} \otimes \l_{a^*}^{-1} \l_{a}^{-1}
\]
Therefore the pivotal structure $\eps$ in $\cC'$ is spherical iff $\eps_a = \eps_a^{-1}$ i.e. $\eps_a = \pm 1$ for all $ a \in L$, and the latter holds by Definition \ref{def:modular}.  The quantum dimension of a simple object $a \in L$ is then
\[ q_a = \frac{\eps_a}{\l_a \l_{a^*}} \]
a number which is not canonical but depends on our choice of roots $\{\l_a\}_{a \in L}$.  As expected the quantum dimension of $a \in L$ is not zero (see for example Lemma 2.4.1 in \cite{BK}).

\noindent \textbf{Balancing:} $\cC'$ is tortile with balancing $\t$ given by $\t = \psi \circ \eps^{-1}$ (see \S\ref{ss:balance}).  So set $\cC(L,N,F,R,\eps) \equiv \cC'$ having the structure of a ribbon fusion category.

\noindent \textbf{Modularity:} Lastly, we check the ribbon fusion category $\cC'$ we constructed is modular.  Modularity holds iff the $S$-matrix is invertible.  The entries of the $S$-matrix are defined to be
\[
\raisebox{2cm}{$S_{ab} = 
$}
\scalebox{.6}{
}
\]

We then compute

{\allowdisplaybreaks
\begin{eqnarray*}
\frac{S_{ab}}{q_a q_b} 
&&
= \raisebox{-3cm}{\scalebox{.6}{}}
= \sum_{c \in L} \sum_{i=1}^{N_{ab^*}^c} \sum_{j=1}^{N_{cb}^{a}}
G_{ab^*b}^{a} \left[ \scalebox{.6}{$ \begin{array}{ccc} 1 & \mathbf{1} & 1 \\ i & c & j \end{array}%
 $} \right]
\end{eqnarray*}
}

Let us denote \label{eq:S_calc}
\[ \hat{S}_{ab} = \frac{S_{ab}}{q_a q_b}\]
Then by the above calculation
\[\hat{S}_{ab} = \sum_{c \in L} 
\sum_{i=1}^{N_{ab^*}^c} 
\sum_{j=1}^{N_{cb}^{a}}
\sum_{i'=1}^{N_{b^*a}^c}
\sum_{i''=1}^{N_{ab^*}^c}
G_{ab^*b}^{a} \left[ \scalebox{.6}{$ \begin{array}{ccc} 1 & \mathbf{1} & 1 \\ i & c & j \end{array} $} \right]
R_{b^*a}^c \left[ \scalebox{.6}{$ \begin{array}{c} i \\ i' \end{array} $} \right]
R_{ab^*}^c \left[ \scalebox{.6}{$ \begin{array}{c} i' \\ i'' \end{array} $} \right]
F_{ab^*b}^{a} \left[ \scalebox{.6}{$ \begin{array}{ccc} i'' & c & j \\ 1 & \mathbf{1} & 1 \end{array} $} \right]. \]
Let $D$ denote the diagonal matrix of dimensions $|L| \times |L|$ whose diagonal entries are the quantum dimensions of $a \in L$. Then
\[ S = D \hat{S} D \]
Since the quantum dimensions of simple objects $a \in L$ are non-zero, the invertibility of the $S$-matrix is equivalent to the invertibility of $\hat{S}$ which is given in condition (\ref{eq:modularity}) of Definition \ref{def:modular}.

\noindent \textbf{Equivalence:} Let $\cC$ be a modular category defined over $k$, and $(L,N,F,R,\eps)(\cC)$ a modular system extracted from it.  The modular system $(L,N,F,R,\eps)(\cC)$ depends on a choice of basis $\{\e_{ab}^{u,i}\}_{i=1}^{N_{ab}^u}$ for each morphism space $\cC(a \otimes b,u)$ (restricted by certain ad-hoc choices).  As in proof of Proposition \ref{prop:fusion} we define a monoidal functor 
\[ F = (F_0,F_1,F_2) : \cC(L,N,F,R,\eps)(\cC) \ra \cC \otimes_k k' \]
\[ F_1 : f  \in Ob \, \cC(L,N,F,R,\eps)(\cC) \mapsto \bigoplus_{a \in L} a^{\oplus f(a)} \in Ob \, \cC \]
\[ F_1 : \cC((L,N,F,R,\eps)(\cC))(f,g) \longmapsto \cC(F_1(f),F_1(g)) \otimes_k k' \] 
and the isomorphism
\[ \bigoplus_{a \in L} \Mat_{g(a) \times f(a)}(k') \xra{\ \cong \ } 
{\cC} \left( \bigoplus_{a \in L} a^{\oplus f(a)},  
\bigoplus_{a \in L} a^{\oplus g(a)} \right)  \otimes_k k' \]
defines $F_1$ on morphisms
\[ F_2(a,b) = \sum_{u \in L} \sum_{i=1}^{N_{ab}^u} \e_{ab}^{u,i} \otimes 1 \]
\[ F_0 : F_1(\mathbf{1}) \ra \mathbf{1} = 1_\mathbf{1} \otimes 1 \]  
We need to check $F$ is a braided functor, namely,
\[
\xymatrix{
F_1(a) \otimes F_1(b) \ar[rr]^{F_2(a,b)} \ar[d]_{\b_{F_1(a),F_1(b)}} && F_1(a \otimes b) \ar[d]^{F_1(\b)} \\
F_1(b) \otimes F_1(a) \ar[rr]_{F_2(b,a)} && F_1(b \otimes a)
}
\]
Diagram equals
\[
\xymatrix{
a \otimes b \ar[rrr]^{F_2(a,b) = \sum_u \sum_i \e_{ab}^{u,i} \otimes 1} \ar[d]_{\b_{a,b} \otimes 1} 
&&& \bigoplus_{u \in L} u^{\oplus N_{ab}^u} \ar[d]^{\sum_{u \in L} R_{ab}^u \otimes 1} \\
b \otimes a \ar[rrr]_{F_2(b,a) = \sum_u \sum_j \e_{ba}^{u,j} \otimes 1} &&&  \bigoplus_{u \in L} u^{\oplus N_{ba}^u}
}
\]
commuting by definition of $R_{ab}^u$.  Hence $\cC((L,N,F,R,\eps)(\cC)) \simeq \cC$ as modular categories.
\end{mylist}
\end{proof}

\section{Applications and Examples} 

\subsection{Fusion and Modular Varieties} 
\label{sec:fusion-and-modular-varieties}

Following the language of \cite{O}, a pair $(L,N)$, as in Definition \ref{def:fusion}, gives rise to a unital based ring $R=R(L,N)$ with $\ZZ_{\geq 0}$-basis $L$ and relations given by $N$. This is the Grothendieck ring of the fusion category $\cC(L,N,F)$ and the modular category $\cC(L,N,F,R,\eps)$. 

\begin{definition} 
Given a pair $(L,N)$, as in Definition \ref{def:fusion}, we define its associated {\em fusion variety} $X(L,N)$ to be the set of complex solutions to the algebraic equations (\ref{eqn:triangle})-(\ref{eqn:pentagon}) including the requirement that the $F$-matrices are invertible.  
\end{definition}

\begin{definition}
Given a triple $(L,N,\eps)$, as in Definition \ref{def:modular}, we define its associated {\em modular variety} $X(L,N,\eps)$ to be the set of complex solutions to the algebraic equations (\ref{eqn:triangle})-(\ref{eqn:pentagon}), (\ref{eq:hexa})-(\ref{eq:piv}) including the requirement that the $F$-matrices, $R$-matrices and $\hat{S}$-matrix are invertible.  
\end{definition}

Clearly, the fusion and modular varieties, $X(L,N)$ and $X(L,N,\eps)$, are complex affine algebraic variety (where it is possible that $X(L,N)$ or $X(L,N,\eps) = \emptyset$), and there is a forgetful map $X(L,N,\eps) \ra X(L,N)$.  Every point $F \in X(L,N)$ is a fusion system $(L,N,F)$ defined over $\CC$, and every point $(F,R) \in X(L,N,\eps)$ is a modular system $(L,N,F,R,\eps)$ defined over $\CC$.

Let $G = \prod_{a,b,u \in  L} \GL_{N_{ab}^u} (\CC)$.  The group $G$ acts algebraically on $X(L,N)$ and $X(L,N,\eps)$; for $g = \prod_{a,b,u \in  L} (g_{ab}^u) \in G$ and $g^{-1} = \prod_{a,b,u \in  L} (g_u^{ab}) \in G$, we have $g : F \mapsto F^g \in X(L,N)$ and $g : (F,R) \mapsto (F^g,R^g) \in X(L,N,\eps)$
\begin{equation}
\label{gauge-change}
(F^g)_{abc}^u \left[ \scalebox{.6}{$\begin{array}{ccc} i' & d & j' \\ m' & e & n' \end{array}$} \right] \defeq
\sum_{i=1}^{N_{ab}^d} \sum_{j=1}^{N_{dc}^u} \sum _{m=1}^{N_{ae}^u} \sum_{n=1}^{N_{bc}^e} (g_{ab}^{d})_j^{j'} (g_{dc}^u)_i^{i'} F_{abc}^u \left[ \scalebox{.6}{$\begin{array}{ccc} i & d & j \\ m & e & n \end{array}$} \right] (g_u^{ae})^m_{m'} (g_e^{bc})^n_{n'}
\end{equation}
\[(R^g)_{ab}^u  \left[ \scalebox{.6}{$\begin{array}{c} i' \\ m' \end{array}$} \right] \defeq \sum_{i=1}^{N_{ba}^u} \sum_{m=1}^{N_{ba}^u} (g_{ba}^u)_i^{i'} R_{ab}^u  \left[ \scalebox{.6}{$\begin{array}{c} i \\ m\end{array}$} \right] (g_u^{ab})^m_{m'} 
\]

Each $g \in G$ induces a monoidal equivalence 
\begin{eqnarray*}
&& \cF^g = (\cF^g_0,\cF^g_1,\cF^g_2) : \cC \ra \cC^g \\ 
&& \cC \defeq \cC(L,N,F) \text{ or } \cC(L,N,F,R,\eps) \\
&& \cC^g \defeq \cC(L,N,F^g) \text{ or } \cC(L,N,F^g,R^g,\eps)
\end{eqnarray*}
with $\cF^g_1 = 1_\cC$ and
\begin{eqnarray*}
&& \cF^g_2(a,b) : \cF^g_1(a) \otimes_{\cC^g} \cF^g_1(b) \xra{\ \cong \ } \cF^g_1(a \otimes_\cC b) \\ 
&& \cF^g_2(a,b) = \sum_{u \in L} \cF^g_2(a,b)(u) \in \bigoplus_{u \in L} \GL_{N_{ab}^u} (k) \subset \bigoplus_{u \in L} \Mat_{(a \otimes b)(u) \times (a \otimes b)(u)} (k) \\
&& \cF^g_2(a,b) = \sum_{u \in L} \cF^g_2(a,b)(u) = \sum_{u \in L} (g_{ab}^u)^{-1} = \sum_{u \in L} g_u^{ab}
\end{eqnarray*}
which is braided in the modular case.  Thus the elements of a $G$-orbit give rise to equivalent categories.

One can find those $g \in G$ for which $F^g = F$ and $R^g = R$.  The second Davydov-Yetter cohomology group $H^2(\cC)$ classifies infinitesimal basis changes that leave $F$ fixed up to natural automorphisms of the identity functor (see \cite{Ki}).  However, by Ocneanu Rigidity $H^n(\cC)=0$ for all $n > 0$ (see \cite{ENO}).  Therefore, $F^g = F$ and $R^g = R$ if and only if there exists $\zeta : L \ra \CC^\times$ such that
\[ g_{ab}^u = \frac{\zeta(a) \zeta(b)}{\zeta(u)} 1_{N_{ab}^u} \]
This condition is equivalent to $\cF^g$ being monoidally equivalent to the identity, i.e. $\cF^g \simeq^{\otimes} 1_{\cC(L,N,F = F^g)}$.  In particular, the stabilizer of $F \in X(L,N)$ in $G$ does not depend on $F$.    By \cite{Hum} \S8.3 then, all orbits are smooth and closed, hence $X(L,N)$ is smooth, and the same applies for $X(L,N,\eps)$.

Consider $\Aut(R)$, the group of auto-equivalences of $R=R(L,N)$ in the sense of unital based rings (see \cite{O} Definition 1(iv)).  Each $\vp \in \Aut(R)$ is given by a bijection $\vp : L \ra L$ fixing $\mathbf{1}$ and  satisfying $N_{ab}^{u} = N_{\vp(a)\vp(b)}^{\vp(u)}$.  For this reason $\Aut(R)$ is finite.  

Each $\vp$ defines a fusion (modular) system $F^\vp \in (L,N)$ (or $(F^\vp,R^\vp) \in (L,N,\eps)$) with 
\[ (F^\vp)_{abc}^u = F_{\vp^{-1}(a)\vp^{-1}(b)\vp^{-1}(c)}^{\vp^{-1}(u)} \  , \ 
(R^\vp)_{ab}^u = R_{\vp^{-1}(a)\vp^{-1}(b)}^{\vp^{-1}(u)} \]
and induces a monoidal equivalence 
\begin{eqnarray*}
&& \cF^\vp = (\cF^\vp_0,\cF^\vp_1,\cF^\vp_2) : \cC(L,N,F) \ra \cC(L,N^\vp,F^\vp) \\
&& \cF^\vp_1(f) = f \circ \vp^{-1}\,,\,\,\,\, f \in \text{Ob}\,\cC(L,N,F) \\
&& \cF^\vp_1(A)(u) = A(\vp^{-1}(u))\,,\,\,\,\, A \in \text{Mor}\,\cC(L,N,F) \\
&& \cF^\vp_2(a,b) = 1
\end{eqnarray*}
which is again braided in the modular case.  Note that $\cF_1^\vp(a) = \vp(a)$. 

The group $\Aut(R)$ acts on $G$ by permutation of labels of $L$.  We may combine the two actions together to an action of the semi-direct product $\Aut(R) \ltimes G$ on $X = X(L,N) \text{ or } X(L,N,\eps)$.

Let $\cC$ be a fusion (modular) category over $\CC$, and $\e$ be a choice of gauge.  By Lemma \ref{lem:fusion_extract}, $\e$ gives rise to a fusion (modular) system $ F \in X(L,N)$ (or $(F,R) \in X(L,N,\eps)$).  A choice $\e^g$ gives rise to the fusion (modular) system $F^g \in X(L,N)$ (or $(F^g,R^g) \in X(L,N,\eps)$), where
\[ (\e^g)_{ab}^{u,j'} = \sum_{j=1}^{N_{ab}^u} (g_{ab}^u)_j^{j'} \e_{ab}^{u,j} \quad , \quad g_{ab}^u \in \GL_{N_{ab}^u} (\CC) \quad , \quad j' = 1, \ldots, N_{ab}^u \]

In the proof to Proposition \ref{prop:fusion}, we defined a monoidal equivalence $\cF : \cC(L,N,F) \ra \cC$ with $\cF_2(a,b) = \sum_{u,i} \e_{ab}^{u,i}$.  We can use $\cF$ to define monoidal auto-equivalences of $\cC$
\[
\xymatrix{
\cC \ar[rr]^{\cF^{-1}} && \cC(L,N,F) \ar[dl]_{\cF^g}^{\cF^\vp} \\
& \cC(L,N,F^g) \ar[ul]^{\cF} &
}
\]
This way we have $G \hookrightarrow \Aut_0^{\otimes(\text{br})} (\cC)$ with the latter denoting the group of monoidal (braided) auto-equivalences of $\cC$ with underlying functor $1_\cC$, up to monoidal natural isomorphisms.  The only data involved for $\cF^g$ is
\[
\xymatrix{
\sum_u u^{\oplus N_{ab}^u} \ar[rr]^{\cF_2^{-1}(a,b) = \sum \e_{u,i}^{ab}} && a \otimes b \ar[dl]^{\cF_2^g(a,b) = \sum (g^u_{ab})^{-1}} \\
& a \otimes b \ar[ul]^{\cF_2(a,b) = \sum \e_{ab}^{u,j}} &
}
\]
Therefore any auto-equivalence with underlying functor the identity is given in terms of gauge change, namely,
\[ G \cong \Aut_0^{\otimes(\text{br})} (\cC) \]
Every monoidal auto-equivalence of $\cC$ induces an auto-equivalence of $R$, so $\Aut(R) \hookrightarrow \Aut^{\otimes(\text{br})} (\cC) \ra \Aut(R)$ composes to the identity.  Overall,
\[ \Aut^{\otimes(\text{br})} (\cC) \cong \Aut(R) \ltimes G \]
where $\Aut^{\otimes(\text{br})}$ denotes the group of monoidal (braided) auto-equivalences of $\cC$, up to monoidal natural isomorphisms.

Fixing a pair $(L,N)$ (or triple $(L,N,\eps)$) and $R=R(L,N)$, the $(\Aut(R) \rtimes G)$-orbits of $X = X(L,N) \text{ or } X(L,N,\eps)$ are in $1:1$ correspondence with equivalence classes of fusion (modular) categories with Grothendieck ring $R$, of which there are finitely many by Ocneanu Rigidity.  In particular, $G$ has finitely many orbits in $X$ (recall $\Aut(R)$ is finite).

In the event that $R$ is a hyperring, i.e. $N_{a b}^u \in \{0,1\}$ for all $a,b,u \in L$, additional characterization applies. In particular, one may compute a finite set of $G$-invariant quantities which determine the $G$-orbit of a point $F \in X(L,N)$. This is described in Appendix~\ref{appendix}.
 
\subsection{Arithmetics of Fusion and Modular Categories}

Let $\cC$ be a fusion (modular) category over $\CC$.  As was pointed out by Etingof, Nikshych and Ostrik in \cite{ENO}, every fusion (modular) category is defined over a number field, in other words:

\begin{prop}
\label{prop:extension}
Let $\cC$ be a fusion (modular) category over $\CC$.  There exists a finite extension $k$ of $\QQ$ and a fusion (modular) category $\cC'$ defined over $k$ such that $\cC \simeq \cC' \otimes_{k} \CC$ as fusion (modular) categories over $\CC$.
\end{prop}

It has become customary to refer to $k$ from Proposition \ref{prop:extension} as the {\em field of definition} or the {\em defining number field} of $\cC$ although the use of the definite article is possibly misleading, one could choose $k$ to be minimal, however it is not necessarily unique.

\begin{lemma}
\label{lemma:algebraic_point}
Let $X$ be a complex affine algebraic variety given by a set of polynomial equations defined over $\qbar$.  Let $G$ be a connected complex algebraic group acting on $X$ which is also defined by polynomial equations over $\qbar$.  Assume $G$ has finitely many orbits in $X$.  Then every $G$-orbit contains an algebraic point.
\end{lemma}

\begin{proof}
Since $G$ is connected, any irreducible component is $G$-invariant (see \cite{Hum} \S8.2).  Pick one.  It is an irreducible complex affine variety given by a set of polynomial equations defined over $\qbar$ equipped with a $G$-action.  It contains a unique Zariski-open orbit $\cO$.  A priori, $\cO$ is defined over $\CC$.  We claim it is in fact defined over $\qbar$.  Algebraic points of a complex affine algebraic variety defined over $\qbar$ are dense in the Euclidean and hence the Zariski topology (this is a consequence of the Nullstellensatz).  Therefore there exists an algebraic point $p\in \cO$.  Consider the sub-variety
\[
\{ (g,x) ,|\, gx=p \} \subseteq G \times X
\]
The condition $gx=p$ is defined over $\qbar$.  The projection of this sub-variety on the second component is exactly $\cO$, therefore a complex affine variety defined over $\qbar$.  We consider the complement of $\cO$ in $X$ which is also defined over $\qbar$.  We pick an irreducible component of $X \backslash \cO$ and repeat the argument. By choosing different irreducible components at every stage we eventually exhaust all orbits of $G$.
\end{proof}

\begin{proof}[Proof of Proposition \ref{prop:extension}]
Let $\cC$ be a fusion category over $\CC$ with associated pair $(L,N)$ and choice of gauge $\e$, by which we associate to $\cC$ a point $F \in X(L,N)$.   By Lemma \ref{lemma:algebraic_point}, we get an algebraic point $F_{\mb{alg}}$ in the $G$-orbit of $F$.  We let $k = \QQ(F_{\mb{alg}})$.  The category $\cC(L,N,F_{\mb{alg}})$ is defined over $k$.  And so $\cC \cong \cC(L,N,F_{\mb{alg}}) \otimes_k \CC$ as fusion categories.  The same proof applies to a modular category.
\end{proof}

When $\cC$ is modular, Vafa's theorem asserts that the $T$-matrix elements of $\cC$ are roots of unity (see \cite{V}).  Results due to de~Boere--Goeree, and Coste--Gannon assert that the entries of the $S$-matrix of a modular category lie in a cyclotomic extension of $\QQ$ (see \cite{dBG, CG}).  One may ask whether one may choose the {field of definition}, $k$, to be cyclotomic.  In some simple cases this is possible.

A {pointed modular category} is a modular category whose fusion rules are given by a finite abelian group $G$, namely, whose Grothendieck ring is the group ring $\ZZ G$ of a finite abelian $G$.  As a consequence of this simple fusion structure, unitary solutions to the pentagon equations (\ref{eqn:pentagon}), up to monoidal equivalence, are parameterized by $H^3(G; {U}(1))$ (here $G$ acts trivially on $U(1)$), and it is always possible to choose a representative cocycle with values in $\{\pm 1\}$ \cite{MS}. The entries of the $\tilde{S}$-matrix in this case are all roots of unity \cite{Wang12}.  Since all non-vanishing fusion state spaces are trivially identified with $\CC$, $F$ and $R$-matrices may be reduced to $1$. Therefore, for a unitary pointed modular category one can choose $k$ to be cyclotomic. 

Morrison--Snyder showed that for some `exotic' fusion categories, $k$ is not necessarily cyclotomic \cite{MoS}. However, they further showed that the Drinfeld centers of those exotic fusion categories do have cyclotomic defining number fields.

It seems to be a folklore theorem that all modular categories from the standard quantum group constructions have cyclotomic defining number fields.  But we cannot locate a definite proof in the literature.  Known examples, therefore, prompt the following conjecture.

\begin{con}
\label{con:cyclotomic}
Every modular category defined over $\CC$ has a cyclotomic defining number field.
\end{con}

One possible approach to proving the above conjecture is as follows.  First, show that all quantum double modular categories and all modular categories from minimal model CFTs have cyclotomic defining number fields.  Then prove that every modular category is Witt-equivalent to a minimal model CFT modular category, i.e., the Witt-group of modular categories up to quantum doubles is generated by CFT minimal models \cite{DMNO}.

Given a set of fusion rules, our numerical definition of a modular category - Definition \ref{def:modular} of a modular system - reduces the classification problem to that of solving polynomial equations.  There exist several software packages to perform such a task, but so far each runs into difficulty.  Restricting solutions to an a-priori number field such as a cyclotomic field might be useful, but we are not aware of any such efforts.

\subsection{Galois Twists} 

\label{sec:Gal}

The study of Galois twisting was initiated in the context of conformal field theory in the work of Gannon and others (see \cite{Gan}).

Let $(L,N,F)$ be a fusion system defined over a Galois extension $K/\QQ$.  Every Galois automorphism $\s \in \Gal(K/\QQ)$ defines a new fusion system $(L,N,\s(F))$ over $K$, where $\s$ applies entry-wise to a matrix $M$ defined over $K$, i.e. $\s(M)_{ij}=\s(M_{ij})$.  Same is true for a modular system $(L,N,F,R,\eps)$ - we get a new modular system by applying $\s$: $(L,N,\s(F),\s(R),\eps)$.

Let $\cC$ be a fusion (modular) category over $\CC$.  Let $\cC'$ be a fusion (modular) category defined over a number field $k$ as in Proposition \ref{prop:extension}.  Let $K$ denote the Galois closure of $k$, and $\cC_K \defeq \cC' \otimes_k K$.  Then $\cC \simeq \cC_K \otimes_K \CC$.

\begin{definition}
The {\em Galois twist} of $\cC$ with respect to $\s \in \Gal(K/\QQ)$ is defined to be
\[ \cC^{\s} \defeq (\cC_K)^\s \otimes_K \CC \]
where $(\cC_K)^\s$ was defined in \S\ref{sec:se}.  By Lemma \ref{lem:ext}, $\cC^\s$ is a fusion (modular) category over $\CC$.
\end{definition}

Galois twists arise in the context of quantum group categories.  Let \g be a complex simple Lie algebra.  Let $l$ be a positive integer greater or equal to the dual Coxeter number of \g and let $q$ be a complex number such that $q^2$ is a primitive $l-$th root of unity.  A well known construction produces from this data a pre-modular category $\cC(\g,l,q)$ defined over $\CC$, in particular, $\cC(\g,l,q)$ is defined as a scalar extension of a pre-modular category defined over $k=\QQ(\sqrt[N]{q})$ for some positive integer $N$ (see~\cite{R} for full account of construction and references, see \cite{S} for a complete table of cases for which $\cC(\g,l,q)$ is modular).  

Let us restrict attention to the type $A$ case, and consider in this context fusion structure only.  Let $\g=sl_n(\CC)$ of rank $r=n-1$.  A fusion category is said to be of type $A_{r,l}$ if its Grothendieck semiring is isomorphic to the Grothendieck semiring of $\cC(\g,l,q)$.  Let $\mathcal{M}(A_{r,l})$ denote the finite set of equivalence classes of fusion categories of type $A_{r,l}$.

For any $n-$th root of unity $\tau$, Kazhdan-Wenzl define a notion of $\tau$-twisting of the associator $\a$ of a class $[\cC] \in \mathcal{M}(A_{r,l})$ via cocycles of $H^3(\ZZ/n\ZZ;\CC^\times)$  \cite{KW}.  Let $\omega : \ZZ \times \ZZ \ra \ZZ$ be given by 
\[ \omega(a,b) = \left[ \frac{a+b}{n} \right] - \left[ \frac{a}{n} \right] - \left[ \frac{b}{n} \right] \]
Define
\[ \a^\tau_{x_\l,x_\mu,x_\nu} = \tau^{\omega(|\mu|,|\nu|) \cdot  |\nu|} \a_{x_\l,x_\mu,x_\nu} \]
where $x_\l,x_\mu,x_\nu \in \cC$ are simple objects labelled by elements $\l,\mu,\nu$ in the fundamental alcove $\Lambda_{n,l} = \{(m_1,\ldots,m_r) \in \ZZ^r\,|\,l-n \geq m_1 \geq \cdots \geq m_r \geq 0 \}$ of size $|(m_1,\ldots,m_r)| := \sum_{i=1}^r m_i$.  Since $\tau^{\omega(-,-) \cdot -}$ is a 3-cocycle, $\a^\tau$ satisfies the pentagon \eqref{eqn:pent_diag}, thereby defining a (possibly new) class $[\cC^\tau] \in \mathcal{M}(A_{r,l})$.  

As an example, a $\tau$-twisting of $SU(2)_2$ with $\tau = -1$ results in the Ising theory.  In the above notation, $SU(2)_2$ theory fits the case $n=2$, $l=4$, $q = e^{\pi i/4}$.  It has three simple objects $\{x_\l\}$, $\l \in \Lambda_{2,4} = \{0,1,2\}$, with $x_0 = \mathbf{1}$.  Its fusion rules are given by $x_1^2 = \mathbf{1} + x_2$, $x_1 x_2 = x_2 x_1 = x_1$ and $x_2^2=\mathbf{1}$.  All non-vanishing fusion spaces are one dimensional and the $F$-matrices of interest are given by 
\[ F_{x_1 x_1 x_1}^{x_1} = -\frac{1}{\sqrt{2}} \begin{pmatrix} 1& 1 \\ 1 & -1 \end{pmatrix} 
\quad , \quad F_{x_2 x_1 x_2}^{x_1} = -1 \quad , \quad F_{x_1 x_2 x_1}^{x_2} = -1 \]
Then $\tau$-twisting by $\tau=-1$ results in a single switch of sign
\[ \a^{\tau}_{x_1,x_1,x_1} = - \a_{x_1,x_1,x_1} \]
which switches the sign of the single $2 \times 2$ $F$-matrix $(F^\tau)_{x_1 x_1 x_1}^{x_1} = - F_{x_1 x_1 x_1}^{x_1}$.  The resulting theory is the Ising theory (as discussed in \cite{RSW}).

\begin{obv}
For an appropriate Galois extension $K/\QQ$, the combined action of $\tau$-twisting and $\s$-twisting in $\s \in \Gal(K/\QQ)$ is transitive in $\mathcal{M}(A_{r,l})$.
\end{obv}

The case of $A_{1,l}$ was proved in \cite{FK}.  For a general rank $r$, Theorem $\mathbf{A_l}$ in \cite{KW} states that all classes in $\mathcal{M}(A_{r,l})$ are of the form $[\cC(\g,l,q)^\tau]$ for some primitive $l-$th root of unity $q^2$, and some $n-$th root of unity $\tau$.  In other words, all monoidal equivalence classes of fusion categories with Grothendeick semiring $K_0\,\cC(\g,q,l)$ are $\tau$-twists and / or Galois twists of a single $\cC(\g,l,q)$.

However, Galois twisting alone does not suffice for transitivity.  The Toric code modular category can be realized as the quantum double of $\ZZ_2$ or the quantum group modular category $SO(16)_1$.  It has $\QQ$ as its defining number field \cite{RSW}.  There is another modular category $SO(8)_1$ with the same fusion rules and $\QQ$ as its defining number field.  Therefore, they belong to different Galois orbits for the same fusion rules given by the abelian group $\ZZ_2 \times \ZZ_2$ (the Galois group is trivial!).  Furthermore, $SU(2)_2$ and the Ising theory belong to different Galois orbits because they have different Frobenius-Schur indicators. 

The Fibonacci category plays a prominent role in topological quantum computing \cite{Wangbook}.  It is a unitary modular category with two labels denoted $\{\mathbf{1},x\}$, and only one non-trivial fusion rule $x^2=\mathbf{1} + x$ (see example \ref{ex:Fib} on page \pageref{ex:Fib}).  There are only two non-trivial $F$-matrices for the Fibonacci theory denoted
\[ F_{xxx}^{\mathbf{1}} = z \quad , \quad
F_{xxx}^x = \begin{pmatrix} z_{11}& z_{12} \\ z_{21}& z_{22} \end{pmatrix} \]
which satisfy
\[
\begin{pmatrix}
1& 0\\
0& z^2
 \end{pmatrix}
 =F_{xxx}^x \begin{pmatrix}
1& 0\\
0& z
 \end{pmatrix}
F_{xxx}^x,\]
\[
\begin{pmatrix}
1& 0\\
0& F_{xxx}^x
 \end{pmatrix}
\begin{pmatrix}
0&1& 0\\
1&0& 0\\
0&0&1
 \end{pmatrix}
\begin{pmatrix}
1& 0\\
0& F_{xxx}^x
 \end{pmatrix}
 =
 \begin{pmatrix}
z_{11}&0& z_{12}\\
0&z& 0\\
z_{21}&0&z_{22}
 \end{pmatrix}
\begin{pmatrix}
1& 0\\
0& F_{xxx}^x
 \end{pmatrix}
  \begin{pmatrix}
z_{11}&0& z_{12}\\
0&z& 0\\
z_{21}&0&z_{22}
 \end{pmatrix}
 .\]
A family of solutions, up to gauge freedom, is given by
\[ z=1 \ ,\ z_{11}=d-1 \ ,\  z_{12}=d+1 \ ,\  z_{21}=2d-3 \ ,\  z_{22}=1-d\] 
where $d$ satisfies $d^2=1+d$.  When $d=\phi=\frac{\sqrt{5}+1}{2}$, the golden ratio, the solutions lead to the unitary Fibonacci theory with a defining number field $\QQ(\zeta_{20}), \zeta_{20}=e^{{2\pi i}/{20}}$ \cite{FW}.  When $d=1-\phi$, the solutions lead to the non-unitary Yang-Lee theory, which also has a defining number field  $\QQ(\zeta_{20})$ \cite{Wangbook}. The Galois twist of the Fibonacci theory has an orbit of size $4$: Fibonacci theory and its complex conjugate, Yang-Lee theory and its complex conjugate.  Since the Yang-Lee theory is non-unitary, this example demonstrates that Galois twisting does not preserve unitarity.

When a modular category is unitary there is a choice of gauge making all $F$-matrices unitary, which, in applications to physics, is often required.  However, there is an incompatibility between making $F$ unitary and choosing $6j$ symbols inside a cyclotomic field as shown in \cite{FW}.  As we see above, there is a choice of a cyclotomic defining number field $\QQ(\zeta_{20})$ for the Fibonacci theory.  In \cite{FW}, we show that in order to have a real unitary $F$-matrix, the defining number field for Fibonacci theory has to be a Galois field with non-abelian Galois group.  One choice is $\QQ(\sqrt{\phi},\zeta_{20})$ with Galois group the dihedral group $D_4$ of $8$ elements.   

\begin{definition}
\begin{mylist}
\item Given a modular category $\cC$, the intrinsic data of $\cC$, denoted as $\textrm{ID}(\cC)$, consist of the fusion rules $N_{ij}^k$, the entries of the $S$-matrix, the $T$-matrix, the eigenvalues of the $R$-matrices, and the Frobenius-Schur indicators. 
\item The rational part of the $\textrm{ID}(\cC)$ is the set $\QQ\cap \textrm{ID}(\cC)$.   
\end{mylist}
\end{definition}

\begin{con}
\label{ID-conjecture}
The $\textrm{ID}(\cC)$ of a modular category is the \textrm{ID} of a modular category in the sense that it is a complete invariant of modular categories, i.e., $\textrm{ID}(\cC)$ determines a modular category.
\end{con}

\begin{obv}
\label{Q-ID}
Let $\cC^{\s}$ be the Galois twist of a modular category $\cC$ with respect to $(\s, K_\cC)$.  Then 
$\QQ\cap \textrm{ID}(\cC)=\QQ\cap \textrm{ID}(\cC^{\s})$, i.e., the rational part of the intrinsic data of a modular category is invariant under Galois twists.
\end{obv}

\begin{con}
Given a set of fusion rules, the rational parts of the intrinsic data from all modular categories with the same fusion rules are in one-one correspondence to the Galois orbits.
\end{con}

\subsection{Exotic modular categories}

Most of the known modular categories are related to the quantum group construction.  Exotic unitary fusion categories arise in the context of subfactor theory \cite{MoS}.  Taking the Drinfeld center, or quantum double, of those exotic fusion categories results in new unitary modular categories which do not resemble those constructed from quantum groups.  

It is a folklore conjecture that all unitary modular categories can be generated from quantum groups \cite{HRW}.  Loosely, we would like to call any unitary modular category that cannot be constructed from quantum group categories an exotic modular category.  But it is very difficult to mathematically characterize all modular categories from quantum group constructions, and thus define exoticism.  

Our arithmetic definition of modular categories provides one way to define exoticism.  Since all quantum group modular categories seem to have cyclotomic defining number fields, we will consider a modular category exotic if it does not have a cyclotomic defining number field.

Computable numbers are defined by Turing in his seminal paper \cite{T}.  It has been shown that computable numbers and polynomial-time computable numbers respectively form algebraically closed fields \cite{Ri}\cite{Ma}. 

\begin{definition}
\begin{mylist}
\item A modular category $\cC$ is cyclotmically exotic if $\cC$ cannot be presented within any cyclotomic field. 
\item A modular category $\cC$ is Turing poly-exotic if $\cC$ does not have an algebraic defining number field within polynomial time computable numbers.
\end{mylist} 
\end{definition}

Conjecture \ref{con:cyclotomic} amounts to there being no cyclotomically exotic modular categories.

Mathematics and physics study two different worlds: a man-created logic world, and a man-inhabited causal world.  There are plenty of mathematical tasks that are not Turing computable such as the classification of $4$-manifolds.  Since physics seems to be simulatable efficiently by computing devices, we believe there should be no Turing poly-exotic unitary modular categories.    

\section{Acknowledgements}  

The first author would like to thank D.\ Freed for inspiring and supporting her through this project.  The authors would like to thank T.\ Gannon, and R.\ Ng for their interests in this work.  The third author is partially supported by NSF DMS 1108736.

\appendix

\section{Gauge equivalence of fusion systems over hyperrings}\label{appendix}
\begin{center} Tobias Hagge$^2$ and Matthew Titsworth$^2$\end{center}

The Grothendieck ring $R$ of a fusion category is a hyperring if $N_{a b}^c \in \{0,1\}$ for all $a,b,c \in L$. In this case, equation~\eqref{gauge-change} reduces to the following:
\begin{equation}\label{gauge-change-hyperring}
(F^g)_{abc}^u \left[ \scalebox{.6}{$\begin{array}{ccc} 1 & d & 1 \\ 1 & e & 1 \end{array}$} \right] \defeq
(g_{ab}^{d}) (g_{dc}^u) F_{abc}^u \left[ \scalebox{.6}{$\begin{array}{ccc} 1 & d & 1 \\ 1 & e & 1 \end{array}$} \right] (g_u^{ae}) (g_e^{bc})
\end{equation}

In this case, $(F^g)_{abc}^u \left[ \scalebox{.6}{$\begin{array}{ccc} 1 & d & 1 \\ 1 & e & 1 \end{array}$} \right]= 0$ iff $F_{abc}^u \left[ \scalebox{.6}{$\begin{array}{ccc} 1 & d & 1 \\ 1 & e & 1 \end{array}$} \right]=0$.

Given $(a,b,c) \in L$, let $t_{ab}^c := (a,b,c)$, and let $T =\{(a,b,u) \in L \times L \times L | N_{a b}^u = 1\}$. If $x$ is a set, let $F(x)$ denote the free abelian group with basis $x$, with the group operation expressed multiplicatively. Let $Q \subset F(T)$ such that
\[Q = \{t_{a b}^d t_{d c}^u(t_{b c}^e t_{a e}^u)^{-1} | \{t_{a b}^d,t_{d c}^u,t_{b c}^e,t_{a e}^u\} \subset T\}.\]
Given fusion system $F \in X(L,R)$ and $q \in Q$, $q = t_{a b}^d t_{e c}^u(t_{b c}^e t_{a e}^u)^{-1}$ for some  $t_{a b}^d, t_{e c}^u,t_{b c}^e, t_{a e}^u \in T$, let
\[F_q = F_{abc}^u \left[ \scalebox{.6}{$\begin{array}{ccc} 1 & d & 1 \\ 1 & e & 1 \end{array}$} \right],\]
and let
\[\hat Q = \{q \in Q | F_q \ne 0\}.\]

Section \ref{sec:fusion-and-modular-varieties} shows that two fusion systems $F$ and $F'$ are gauge equivalent (write $F \sim F'$) iff, using the above definitions, there exists a function $s:T \to k^\times$, $s(t_{a b}^c) = s_{a b}^c$, such that for all $\{t_{a b}^c,t_{d c}^u,t_{b c}^e,t_{a e}^u\} \subset T$ we have
\begin{equation}\label{gauge-equivalence-s}
s_{a b}^d s_{d c}^u(s_{b c}^e s_{a e}^u)^{-1} F_{abc}^u \left[ \scalebox{.6}{$\begin{array}{ccc} 1 & d & 1 \\ 1 & e & 1 \end{array}$} \right] = (F'_{abc})^u \left[ \scalebox{.6}{$\begin{array}{ccc} 1 & d & 1 \\ 1 & e & 1 \end{array}$} \right].
\end{equation}

Given such a function $s$, $s$ extends to a group homomorphism $F(T) \to k^\times$, which then restricts to a group homomorphism $\phi_Q:\langle Q \rangle \to k^\times$ satisfying the following equation for each $q \in Q$:

\begin{equation}\label{somemap}
\phi_Q(q) F_q = F'_q,
\end{equation}

Suppose instead that $\phi_{\hat Q}:\langle \hat Q \rangle \to k^\times$ is a group homomorphism which satisfies equation (\ref{somemap}) for every $q \in \hat Q$. Since $k$ has characteristic 0, $k^\times$ is an injective $\mathbb Z$-module. Thus $\phi_{\hat Q}$ extends to a well-defined homomorphism $\phi_{Q}:\langle Q \rangle \to k^\times$, and again to $\phi_T:F(T) \to k^\times$ such that with $s_{a b}^c := \phi_T(t_{a b}^c)$, equations (\ref{gauge-equivalence-s}) hold.

Thus $F \sim F'$ iff $F$ and $F'$ have simultaneous zeros and the map $\psi:\hat Q \to k^\times$ such that $\psi(q) = \frac{F'_q}{F_q}$ extends to a homomorphism $\langle \hat Q \rangle \to k^\times$.

Since $\langle \hat Q \rangle \cong F(\hat Q) / S$, where $S \unlhd F(\hat Q)\unlhd F(Q)$ is the (free abelian) subgroup of words in $\hat Q$ trivial in $F(T)$, there is a basis $\{s_1, \ldots s_m\}$ for $S$, where each $s_i = q_{i,1}^{k_{i,1}} \ldots q_{i,n_i}^{k_{i,n_i}}$, with $q_{i,j} \in Q$ and $k_{i,j} \in \mathbb Z \backslash \{0\}$. Then $F \sim F'$ iff for each $s_i \in S$, 
\[\tilde \phi(s_i) =\left(\frac{F'_{q_{i,1}}}{F_{q_{i,1}}}\right)^{k_{i,1}}\ldots \left(\frac{F'_{q_{i,n_i}}}{F_{q_{i,n_i}}}\right)^{k_{i,n_i}} = 1.\]
Equivalently, $F \sim F'$ iff for all $s_i \in S$, 
\begin{equation}\label{fsymbols-equal}
(F_{q_{i,1}})^{k_{i,1}}\ldots (F_{q_{i,n_i}})^{k_{i,n_i}} = (F'_{q_{i,1}})^{k_{i,1}}\ldots (F'_{q_{i,n_i}})^{k_{i,n_i}}.
\end{equation}

\bibliographystyle{ams-alpha}

\begin{thebibliography}{ABC}

\bibitem[BK]{BK}B. Bakalov and A. Kirillov, Jr., 
\textit{Lectures on Tensor Categories and Modular Functors}, 
University Lecture Series \textbf{21},  Amer.\ Math.\ Soc., 2001.

\bibitem[BNRW]{BNRW} P.~Bruillard, S.~ Ng, E.~ Rowell, and Z.~ Wang, 
\textit{On modular categories}, 
in preparation.

\bibitem[CE]{CE} D.~Calaque, P.~Etingof,
\textit{Lectures on tensor categories} 
in \textit{Quantum groups}, pp. 1--38, IRMA Lect. Math. Theor. Phys. \textbf{12}, Eur. Math. Soc., 2008.
\bibitem[CG]{CG} A.~Coste, T.~Gannon,
\textit{Remarks on Galois symmetry in
rational conformal field theories}, 
Phys. Lett. B \textbf{323} 3-4 (1994), pp. 316--321.

\bibitem[dBG]{dBG} J.~de~Boere, J.~Goeree,
\textit{Markov traces and $II_1$ factors in conformal field theory}, 
Comm. Math. Phys. \textbf{139} 2 (1991), pp. 26--304.

\bibitem[DMNO]{DMNO}A.\ Davydov, M.\ Mueger, D.\ Nikshych, V.\ Ostrik, 
\textit{The Witt group of non-degenerate braided fusion categories},  
J. reine angew. Math. (2012),
\href{http://arxiv.org/abs/1009.2117}{arXiv:1009.2117}.

\bibitem[ENO]{ENO}P. Etingof, D. Nikshych, and V. Ostrik,  
\textit{On fusion categories}, Ann. of Math.\ \textbf{162} 2 (2005), pp. 581--642, 
\href{http://arxiv.org/abs/math/0203060}{arXiv:math/0203060}.

\bibitem[MGHTTW]{GaloisConjugate} M.~Freedman, J.~Gukelberger, M.~B.~Hastings, S.~Trebst, M.~Troyer, Z.~Wang,
\textit{Galois conjugates of topological phases}, 
Phys. Rev. B \textbf{85} 045414 (2012),
\href{http://arxiv.org/abs/1106.3267}{arXiv:1106.3267}.

\bibitem[FK]{FK} J.~Fr\"ohlich, T.~Kerler,
\textit{Quantum Groups, Quantum Categories and Quantum Field Theory}, 
Lecture Notes in Mathematics, Volume 1542. Springer Verlag, 1993.

\bibitem[FW]{FW} M.~Freedman, Z.~Wang, 
\textit{Large quantum Fourier transforms are never exactly realized by braiding conformal blocks}, 
Phys. Rev. A \textbf{75} 032322 (2007), 
\href{http://arxiv.org/abs/cond-mat/0609411}{arXiv:cond-mat/0609411}.

\bibitem[Gal]{Galindo} C.~Galindo, 
\textit{On braided and ribbon unitary fusion categories}, 
\href{http://arxiv.org/abs/1209.2022}{arXiv:1209.2022}.
    
\bibitem[Gan]{Gan} T.~Gannon,
\textit{Modular Data: The Algebraic Combinatorics of Conformal Field Theory},
J. Algebraic Combin. \textbf{22} 2 (2005), pp. 211--250.

\bibitem[HRW]{HRW} S.-M. Hong, E. Rowell, and Z. Wang,   
\textit{On exotic modular tensor categories},
Commun. Contemp. Math. \textbf{10} 1 (2008), pp. 1049--1074, 
\href{http://arxiv.org/abs/0710.5761}{arXiv:0710.5761}.

\bibitem[Hum]{Hum} J.~E.~Humphreys,
\textit{Linear Algebraic Groups},
Graduate Texts in Mathematics, Volume 21, Springer Verlag, 1995.

\bibitem[JWB]{JWB} H.~Jiang, Z.~Wang, L.~Balents, 
\textit{Identifying topological order by entanglement entropy}, 
Nature Physics \textbf{8} (2012), pp. 902--905, 
\href{http://arxiv.org/abs/1205.4289}{arXiv:1205.4289}.

\bibitem[Ka]{Ka} C.~Kassel,
\textit{Quantum groups},
Graduate Texts in Mathematics,
Volume 155, Springer Verlag, 1995.

\bibitem[Ki]{Ki} A.~Kitaev, 
\textit{Anyons in an exactly solved model and beyond},
Ann. Physics \textbf{321} 1 (2006), pp. 2--111, 
\href{http://arxiv.org/abs/cond-mat/0506438}{arXiv:cond-mat/0506438}.

\bibitem[KW]{KW} D. Kazdan; H. Wenzl, 
\textit{Reconstructing monoidal categories}, 
Adv. Soviet Math. \textbf{16} 2 (1993), pp. 111--136.

\bibitem[Mac]{Mac} S.~MacLane,
\textit{Categories for the Working Mathematician},
Graduate Texts in Mathematics,
Volume 5, Springer Verlag, 1971.

\bibitem[Ma]{Ma} T.\ Matsui, 
\textit{On polynomial time computable numbers}, 
\href{http://arxiv.org/abs/cs/0608067}{arXiv:cs/0608067}.

\bibitem[MS]{MS} G.~Moore, N.~Seiberg, 
\textit{Classical and quantum conformal field theory},
 Comm. Math. Phys. \textbf{123} 2 (1989), pp. 177--254.
 
\bibitem[MoS]{MoS} S.~Morrison, N.~Snyder,
\textit{Non-cyclotomic fusion categories},
Trans. Amer. Math. Soc. \textbf{364} 9 (2012), pp. 4713--4733.

\bibitem[NS]{NS} S.-H. Ng and P. Schauenburg, 
\textit{Congruence subgroups and generalized Frobenius-Schur indicators}, 
Comm. Math. Phys, \textbf{300} 1 (2010), pp 1-46, 
\href{http://arxiv.org/abs/0806.2493}{arXiv:0806.2493}.

\bibitem[Mu]{Mu} M.~M\"uger,
\textit{From subfactors to categories and topology I: Frobenius algebras in and Morita equivalence of tensor categories},
J. Pure Appl. Algebra \textbf{180} (2003), pp. 81--157.

\bibitem[O]{O} V.~Ostrik,
\textit{Module categoris, weak Hopf algebras and modular invariants}, 
Transform. Groups \textbf{8} (2003), pp. 177--206.

\bibitem[R]{R} E.~Rowell,
\textit{From quantum groups to unitary modular tensor categories}, 
Contemp. Math. \textbf{413} (2006), pp. 215--230.

\bibitem[Ri]{Ri} H.~G.~Rice, 
\textit{Recursive real numbers}, 
Proc. Amer. Math. Soc. \textbf{5} (1954), pp. 784--791.

\bibitem[RSW]{RSW} E.~Rowell, R.~Stong, Z.~Wang, 
\textit{On classification of modular tensor categories}, 
Comm. Math. Phys. \textbf{292} 2 (2009), pp. 343--389, 
\href{http://arxiv.org/abs/0712.1377}{arXiv:0712.1377}.

\bibitem[S]{S} S.~F.~Sawin,
\textit{Quantum groups at roots of unity and modularity}
J. Knot Theory Ramifications \textbf{15} 10 (2006), pp. 1245--1277.

\bibitem[St]{St} N.~Stalder,
\textit{Scalar Extension of Abelian and Tannakian Categories},
\href{http://arxiv.org/abs/0806.0308}{arXiv:0806.0308}.

\bibitem[T]{T} A.~C.~M.~Turing, 
\textit{On computable numbers, with an application to the entscheidungs problem}, 
Proc. London Math. Soc. \textbf{42} (1936), pp. 230--265.

\bibitem[Tu]{Tu} V.~Turaev, 
\textit{Quantum Invariants of Knots and 3-Manifolds}, 
De Gruyter Studies in Mathematics \textbf{18}, Walter de Gruyter \& Co., Berlin, 1994.

\bibitem[TW]{TW} I.~Tuba, H.~Wenzl, 
\textit{On braided tensor categories of type $BCD$},
J. reine angew. Math. \textbf{581} (2005), pp. 31--69.

\bibitem[V]{V} C.~Vafa,
\textit{Toward classification of conformal theories},
Phys. Lett. B \textbf{206} (1988), pp. 421--426.

\bibitem[Wa10]{Wangbook}Z.\ Wang, 
\textit{Topological quantum computation}, 
CBMS Regional Conference Series in Mathematics, Volume 112, Amer. Math. Soc., 2010.

\bibitem[Wa12]{Wang12} Z.~ Wang, 
\textit{Quantum Computing: a Quantum Group Approach}, 
\href{http://arxiv.org/abs/1301.4612}{arXiv:1301.4612}. 

\bibitem[Y]{Y} S.~Yamagami, 
\textit{Polygonal presentations of semisimple tensor categories},
J. Math. Soc. Japan  \textbf{54} (2002), no. 1, pp. 61--88.

\bibitem[Ye]{Ye} D.M.~Yetter,
\textit{Framed tangles and a theorem of Deligne on braided deformations of Tannakian categories},
in Deformation Theory and Quantum Groups with Applications to Mathematical Physics, 
Proc. AMS-IMS-SIAM Jt. Summer Res. Conf., Amherst MA USA 1990, 
Contemp. Math. \textbf{134} (1992), pp. 325--349

\end{thebibliography}

\end{document}